\documentclass[11pt]{article}
\usepackage{amsmath,amsthm,amssymb}
\usepackage[usenames,dvipsnames]{xcolor}
\usepackage{enumerate}
\usepackage{graphicx}
\usepackage{cite}
\usepackage{comment}
\usepackage{oands}
\usepackage{tikz}
\usepackage{changepage}
\usepackage{bbm}
\usepackage{mathtools}
\usepackage[margin=1in]{geometry}
\usepackage[pagewise,mathlines]{lineno}
\usepackage{appendix}
\usepackage{stmaryrd}
\usepackage{multicol}
\usepackage{microtype}
\usepackage[colorinlistoftodos]{todonotes}
\usepackage{dsfont}

\usepackage[pdftitle={Mismatched SLE/LQG couplings},
  pdfauthor={Morris Ang, Ewain Gwynne},
colorlinks=true,linkcolor=NavyBlue,urlcolor=RoyalBlue,citecolor=PineGreen,bookmarks=true,bookmarksopen=true,bookmarksopenlevel=2,unicode=true,linktocpage]{hyperref}

\theoremstyle{plain}
\newtheorem{thm}{Theorem}[section]

\newtheorem{lemma}[thm]{Lemma}
\newtheorem{prop}[thm]{Proposition}
\newtheorem{proposition}[thm]{Proposition}
\newtheorem{conj}[thm]{Conjecture}

\def\@rst #1 #2other{#1}
\newcommand\MR[1]{\relax\ifhmode\unskip\spacefactor3000 \space\fi
  \MRhref{\expandafter\@rst #1 other}{#1}}
\newcommand{\MRhref}[2]{\href{http://www.ams.org/mathscinet-getitem?mr=#1}{MR#2}}

\theoremstyle{definition}
\newtheorem{defn}[thm]{Definition}
\newtheorem{definition}[thm]{Definition}
\newtheorem{remark}[thm]{Remark}

\newtheorem{prob}[thm]{Problem}

\numberwithin{equation}{section}

\newcommand{\dsb}{\begin{adjustwidth}{2.5em}{0pt}
\begin{footnotesize}}
\newcommand{\dse}{\end{footnotesize}
\end{adjustwidth}}

\newcommand{\ssb}{\begin{adjustwidth}{2.5em}{0pt}}
\newcommand{\sse}{\end{adjustwidth}}

\newcommand{\aryb}{\begin{eqnarray*}}
\newcommand{\arye}{\end{eqnarray*}}
\def\alb#1\ale{\begin{align*}#1\end{align*}}
\def\allb#1\alle{\begin{align}#1\end{align}}
\newcommand{\eqb}{\begin{equation}}
\newcommand{\eqe}{\end{equation}}
\newcommand{\eqbn}{\begin{equation*}}
\newcommand{\eqen}{\end{equation*}}

\newcommand{\BB}{\mathbbm}
\newcommand{\ol}{\overline}
\newcommand{\ul}{\underline}
\newcommand{\op}{\operatorname}

\newcommand{\frk}{\mathfrak}

\newcommand{\ep}{\varepsilon}

\newcommand{\wt}{\widetilde}
\newcommand{\wh}{\widehat}
\newcommand{\mcl}{\mathcal}

\newcommand{\bdy}{\partial}

\newcommand{\eps}{\varepsilon}

\newcommand{\cc}{{\mathbf{c}}}
\newcommand{\ccL}{{\mathbf{c}_{\mathrm L}}}
\newcommand{\ccM}{{\mathbf{c}_{\mathrm M}}}
\DeclareMathOperator{\SLE}{SLE}
\DeclareMathOperator{\Cov}{Cov}
\DeclareMathOperator{\Var}{Var}

 \def\bbi{\mathbbm{i}}
 \def\Z{\mathbb{Z}} 
 \def\R{\mathbb{R}} 
 \renewcommand{\P}{\mathbb{P}}
  \def\bbH{\mathbb{H}}
 \def\D{\mathbb{D}} 
 \def\C{\mathbb{C}}

 \def\cU{\mathcal{U}}
 \def\cT{\mathcal{T}}
 \def\cS{\mathcal{S}}
 
 \def\cQ{\mathcal{Q}}
 \def\cP{\mathcal{P}}
 
 \def\cN{\mathcal{N}}
 \def\cM{\mathcal{M}}
 \def\cL{\mathcal{L}}

 \def\cG{\mathcal{G}}
 \def\cF{\mathcal{F}}
 
 \def\cD{\mathcal{D}}
 \def\cC{\mathcal{C}}
 \def\cB{\mathcal{B}}

\let\originalleft\left
\let\originalright\right
\renewcommand{\left}{\mathopen{}\mathclose\bgroup\originalleft}
\renewcommand{\right}{\aftergroup\egroup\originalright}

\title{Cutting $\gamma$-Liouville quantum gravity by Schramm-Loewner evolution for $\kappa \notin \{\gamma^2,16/\gamma^2\}$}
 \date{ }
 \author{
\begin{tabular}{c} Morris Ang\\[-5pt]\small Columbia University \end{tabular}
\begin{tabular}{c} Ewain Gwynne\\[-5pt]\small University of Chicago \end{tabular}  
}

\begin{document}

\maketitle

\begin{abstract}
There are many deep and useful theorems relating Schramm-Loewner evolution (SLE$_\kappa$) and Liouville quantum gravity ($\gamma$-LQG) in the case when the parameters satisfy $\kappa \in \{\gamma^2,16/\gamma^2\}$. 
Roughly speaking, these theorems say that the SLE$_\kappa$ curve cuts the $\gamma$-LQG surface into two or more independent $\gamma$-LQG surfaces.

We extend these theorems to the case when $\kappa \notin \{\gamma^2,16/\gamma^2\}$. Roughly speaking we show that if we have an appropriate variant of SLE$_\kappa$ and an independent $\gamma$-LQG disk, then the SLE curve cuts the LQG disk into two or more $\gamma$-LQG surfaces which are \emph{conditionally} independent given the values along the SLE curve of a certain collection of auxiliary imaginary geometry fields, viewed modulo conformal coordinate change. 
These fields are sampled independently from the SLE and the LQG and have the property that that the sum of the central charges associated with the SLE$_\kappa$ curve, the $\gamma$-LQG surface, and the auxiliary fields is 26. This condition on the central charge is natural from the perspective of bosonic string theory.  
We also prove analogous statements when the SLE curve is replaced by, e.g., an LQG metric ball or a Brownian motion path. 
Statements of this type were conjectured by Sheffield and are continuum analogs of certain Markov properties of random planar maps decorated by two or more statistical physics models. 

We include a substantial list of open problems.
\end{abstract}

\tableofcontents
\bigskip

\noindent\textbf{Acknowledgments.}
This work has benefited from enlightening discussions with many people, including Amol Aggarwal, Jacopo Borga, Ahmed Bou-Rabee, Nina Holden, Minjae Park, Josh Pfeffer, Guillaume Remy, Scott Sheffield, Xin Sun, Jinwoo Sung, and Pu Yu. 
M.A.\ was supported by the Simons Foundation as a Junior Fellow at the Simons Society of Fellows. 
E.G.\ was partially supported by a Clay research fellowship and by NSF grant DMS-2245832.

\section{Introduction}

\subsection{Overview} 
\label{sec-overview}

\textbf{Schramm-Loewner evolution} (SLE$_\kappa$) for $\kappa > 0$ is a one-parameter family of random fractal curves in the plane originally introduced by Schramm~\cite{schramm0}. SLE describes or is conjectured to describe the scaling limits of various discrete random curves which arise in statistical mechanics. See, e.g.,~\cite{lawler-book,bn-sle-notes} for introductory expository works on SLE. 

\textbf{Liouville quantum gravity} ($\gamma$-LQG) for $\gamma \in (0,2]$ is a one-parameter family of random fractal surfaces (2d Riemannian manifolds) which arise, e.g., in string theory and conformal field theory~\cite{polyakov-qg1}, and as the scaling limits of random planar maps. LQG surfaces are too rough to be Riemannian manifolds in the literal sense. Instead, a $\gamma$-LQG surface can be represented as a random metric measure space parametrized by a domain in $\BB C$ (or more generally a Riemann surface), viewed modulo a conformal change of coordinates rule. See Definition~\ref{def-QS} for a precise definition and, e.g.,~\cite{bp-lqg-notes,gwynne-ams-survey,sheffield-icm} for introductory expository articles on LQG. 

Instead of the parameters $\kappa$ and $\gamma$, it is often useful to instead describe SLE and LQG in terms of the \textbf{central charge} parameters, which are related to $\kappa$ and $\gamma$ by\footnote{Some works associate LQG with the \textbf{matter central charge} $\ccM$ instead of our $\ccL$. Our $\ccL$ is the central charge of Liouville conformal field theory and is related to $\ccM$ by $\ccM = 26-\ccL$.}
\eqb \label{eqn-cc}
\cc_{\op{SLE}} = 1 - 6\left(\frac{2}{\sqrt\kappa} - \frac{\sqrt\kappa}{2}\right)^2 \quad\text{and} \quad
\ccL = 1 + 6\left(\frac{2}{\gamma} + \frac{\gamma}{2} \right)^2 .
\eqe
Each of SLE and LQG can be associated with conformal field theories with their respective central charges, see, e.g.,~\cite{kang-makarov-cft,dkrv-lqg-sphere}.

\begin{figure}[ht!]
\begin{center}
\includegraphics[width=0.65\textwidth]{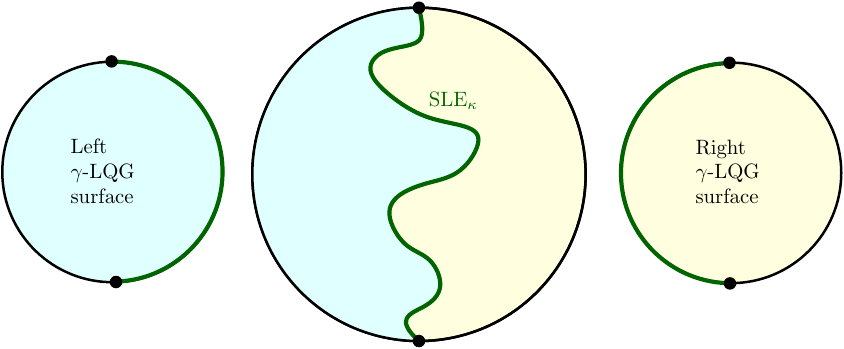}  
\caption{\label{fig-disk-welding} A Liouville quantum gravity surface parametrized by the disk together with an SLE$_\kappa$-type curve between two marked boundary points. Previous works (see in particular \cite{shef-zipper,wedges, ahs-welding}) have shown that if $\kappa = \gamma^2$ and the precise variant of SLE is chosen appropriately, then the sub-LQG surfaces parametrized by the regions to the left and right of the curve are conditionally independent given their LQG boundary lengths (see Definition~\ref{def-QS} for a precise definition of LQG surfaces). In Theorem~\ref{thm-sle-chordal}, we show that if $\kappa \not=\gamma^2$, then the sub-LQG surfaces parametrized by the left and right sides of the curve are conditionally independent given information about the values along the curve of certain auxiliary random fields, sampled independently from the SLE and the LQG. We also prove variants of this statement for other types of SLE$_\kappa$ curves, including ones with multiple complementary connected components. 
}
\end{center}
\end{figure}

In this paper we will study the relationships between SLE and LQG. The first such relationship, called the \textbf{quantum zipper}, was established by Sheffield in~\cite{shef-zipper}. Roughly speaking, this result and its many extensions (including the \textbf{mating of trees} theorem~\cite{wedges}) say the following. Suppose we have a certain SLE$_\kappa$-type curve and a certain $\gamma$-LQG surface, sampled independently from each other, and that the parameters are \textbf{matched} in the sense that
\eqb \label{eqn-matched}
\kappa \in \left\{\gamma^2 ,\frac{16}{\gamma^2} \right\} , \quad\text{equivalently} \quad \cc_{\op{SLE}} = 26-\ccL .
\eqe
Then the sub-LQG surfaces parametrized by the complementary connected components of the SLE$_\kappa$ curve are conditionally independent given the LQG lengths of their boundaries, and their laws can be described explicitly. See Figure~\ref{fig-disk-welding}. This independence property is the continuum analog of certain Markovian properties for random planar maps decorated by statistical physics models. See the survey article~\cite{ghs-mating-survey} for more explanation.

Results of the above type have a huge number of applications, to such topics as SLE and LQG individually, conformal field theory (see, e.g.,~\cite{ars-fzz}), the geometry of random planar maps (see, e.g.,~\cite{ghs-map-dist}), random permutations (see, e.g.,~\cite{borga-skew-permuton}), and the moduli of random surfaces (see, e.g.,~\cite{ars-annulus}). Particularly noteworthy consequences include the equivalence of $\gamma = \sqrt{8/3}$ LQG and the Brownian map \cite{lqg-tbm1} and convergences of conformally embedded random planar maps to LQG \cite{gms-tutte, hs-cardy-embedding}. Many of these applications are surveyed in~\cite{ghs-mating-survey}. 

In this paper, we will establish the first relationships between SLE and LQG in the case when the parameters are \textbf{mismatched}, meaning that $\gamma$ and $\kappa$ are not related as in~\eqref{eqn-matched}. Roughly speaking, we prove the following. Suppose we have an appropriate SLE$_\kappa$-type curve and a $\gamma$-LQG surface (specifically, an LQG disk), sampled independently from each other as above, but whose parameters do not satisfy~\eqref{eqn-matched}. Then the sub-LQG surfaces parametrized by the complementary connected components of the SLE$_\kappa$ curve are not independent, but they are \emph{conditionally} independent if we condition on certain extra information along the SLE$_\kappa$ curve. The necessary extra information is described by one or more random generalized functions, sampled independently from the SLE and the LQG, with the property that $\cc_{\op{SLE}} + \ccL$ plus the central charges associated with the extra generalized functions is equal to 26. These extra random generalized functions are described in terms of the theory of \textbf{imaginary geometry}~\cite{ig1,ig2,ig3,ig4}, which we review in Section~\ref{sec-sle-prelim}. See Theorem~\ref{thm-sle-chordal} for a precise statement. 
Conditional independence statements of the above type were conjectured by Sheffield in private communication.

We also prove similar conditional independence statements when the SLE curve is replaced by other interesting random sets, such as a conformal loop ensemble gasket, a Brownian motion path, or an LQG metric ball. See Theorems~\ref{thm-cle}, \ref{thm-bm} and~\ref{thm-ball}.

The condition that the total central charge should be 26 dates back to Polyakov's seminal path integral formulation of bosonic string theory. For the conformal field theory corresponding to LQG coupled with conformal matter, the central charge condition is equivalent to the Weyl invariance of the theory \cite{polyakov-qg1, dp-multiloop}, i.e., invariance when the underlying Riemannian manifold $(\Sigma, g)$ is replaced by $(\Sigma, e^{2\sigma} g)$. 
This perspective, together with the ansatz that the conformal field theory should correspond to random planar maps decorated by statistical physics models (see the next paragraph), gives physical context and justification for the results of this paper.
The relationships between SLE and LQG in the matched case~\eqref{eqn-matched} can be viewed as Markov properties for Liouville CFT with central charge $\ccL$ coupled to a single matter field of central charge $26-\ccL$. 
This paper shows that one also has Markov properties when instead Liouville CFT is coupled to multiple matter fields with total central charge $26-\ccL$.   

Analogously to the matched case, our results are continuum analogs of certain Markovian properties for random planar maps decorated by \emph{multiple} statistical physics models.
Roughly speaking, these properties say the following. Suppose we have a random planar map decorated by a two or more statistical physics models (e.g., a uniform spanning tree and two discrete Gaussian free fields) and we construct an interface from one of the models (e.g., a branch of the spanning tree). Then the planar maps in the complementary connected components of the interface are conditionally independent given the information about other statistical mechanics models along the interface (in our example, this corresponds to the restrictions of the discrete Gaussian free fields to the spanning tree branch). Similar Markovian properties also hold for related objects, such as uniform meanders. 
See Appendices~\ref{sec-rpm} and~\ref{sec-meander} for further explanations. 

Furthermore, just like in the matched case, our results have a large number of potential applications and extensions, some of which are discussed in Section~\ref{sec-open-problems}.

The proofs in this paper involve several novel ideas which we expect to be useful elsewhere. These include a conditional independence statement for uniform meanders (Appendix \ref{sec-meander}) and a ``rotational invariance" property for the central charges associated with a collection of independent Gaussian free fields in the setting of imaginary geometry (Proposition~\ref{prop-ig-rotate}). This rotational invariance property is the source of the ``total central charge 26" condition in our results. See Remark~\ref{remark-total-central-charge} and Section~\ref{sec-outline} for further discussion of the main ideas of the proofs.

Finally, a potential extension of the present work would allow us to rigorously construct a \textbf{Markovian string trajectory} in $\R^d$ corresponding to bosonic string theory, from a Liouville quantum gravity surface together with $d$ independent Gaussian free fields. See Section~\ref{sec-string-theory} for further discussion.

\subsection{LQG and imaginary geometry surfaces}

\subsubsection{Liouville quantum gravity}

For $\gamma \in (0,2)$, \textbf{$\gamma$-Liouville quantum gravity (LQG)} is, heuristically speaking, the random geometry described by the random Riemannian metric tensor
\eqbn
e^{\gamma \Phi} \, (dx^2 + dy^2) 
\eqen
where $dx^2 + dy^2$ denotes the Euclidean metric tensor on a domain $D\subset\BB C$ and $\Phi$ is a variant of the Gaussian free field (GFF) on $D$. This Riemannian metric tensor does not make literal sense since $\Phi$ is a generalized function. 
Following~\cite{shef-kpz,shef-zipper,wedges}, we rigorously define LQG surfaces as equivalence classes of field/domain pairs modulo a conformal coordinate change formula depending on the parameter $Q := \frac\gamma2 + \frac2\gamma$. Equivalently,
\[
Q(\ccL) = \sqrt{\frac{\ccL - 1}6} \quad \text{ for }\ccL > 25, \qquad\qquad\qquad \ccL(Q) = 1 + 6Q^2 \quad \text{ for } Q > 2.
\]
Let $\wh \C= \C \cup \{ \infty\}$ be the Riemann sphere.

\begin{definition}\label{def-QS}
	Consider pairs $(D,\Phi)$ where $D\subset \C$ is a bounded open set and $\Phi$ is a distribution defined on $D$.
	For $\ccL >25$, 
	a \textbf{(generalized) $\ccL$-LQG surface}
	is an equivalence class of such pairs under the equivalence relation $\sim_{\ccL}$ where $(D, \Phi) \sim_{\ccL} (\wt D, \wt \Phi)$ if there exists a homeomorphism $f:  \C \to \C$ with $f(D) = \wt D$ that is conformal on $D$ such that $\Phi = \wt \Phi \circ f + Q(\ccL) \log |f'|$. 
\end{definition}

We emphasize that the domain $D$ in Definition~\ref{def-QS} need not be simply connected or even connected.  
Definition~\ref{def-QS} differs slightly from the definitions of LQG surfaces found elsewhere in the literature (e.g., in~\cite{wedges}), in that we require that $D$ is bounded and we require that $f : \C \to \C$ instead of just $f : D \to \wt D$ or $f : \ol D \to \ol{\wt D}$. The reason for these modifications is that we will eventually consider LQG surfaces coupled to imaginary geometry fields, so we need to ensure that there is a well-defined notion of the argument of $f'$. See Section~\ref{subsec-coord-change} for further discussion.

We call an equivalence class representative $(D, \Phi)$ an \textbf{embedding} of the  LQG surface. When $\Phi$ is a variant of the Neumann (free-boundary) GFF on $D$, one can define the LQG area measure $\mu_\Phi$ on $D$, which is a limit of regularized versions of $e^{\gamma \Phi(x+i y)} \,dx\,dy$~\cite{shef-kpz,kahane,rhodes-vargas-log-kpz,berestycki-gmt-elementary}.
Similarly, one can define the LQG boundary length measure $\nu_\Phi$ on $\partial D$ and the LQG metric (distance function) $\mathfrak d_\Phi$ on $\ol D$ \cite{dddf-lfpp, gm-uniqueness} (see~\cite{hm-metric-gluing} for the extension of the metric from $D$ to $\ol D$). See Section~\ref{sec-metric-prelim} for more background on the LQG metric. These are compatible with the LQG coordinate change, i.e.,  in Definition~\ref{def-QS} we have $\mu_{\wt \Phi} = f_* \mu_\Phi$, $\nu_{\wt \Phi} = f_* \nu_\Phi$ and $\mathfrak d_{\wt \Phi} = f_* \mathfrak d_\Phi$, and are thus intrinsic to the LQG surface.

\begin{remark} \label{remark-subsurface}
If $(D,\Phi)$ is an embedding of an LQG surface and $U\subset D$ is open, we will often slightly abuse notation by writing $(U , \Phi)/{\sim_{\ccL}}$ instead of $(U,\Phi|_U)/{\sim_{\ccL}}$ for the LQG surface obtained by restricting $\Phi$ to $U$. Moreover, if instead $D\subset \C$ is a bounded set, not necessarily open, then $(D, \Phi)/{\sim_{\ccL}}$ refers to the LQG surface $(\mathrm{int}(D), \Phi)/{\sim_{\ccL}}$. These conventions also apply for imaginary geometry surfaces.
\end{remark}

\subsubsection{Imaginary geometry}

\textbf{Imaginary geometry (IG)} is, heuristically speaking, the random geometry described by the vector field
\eqb \label{eqn-ig-vector-field}
e^{i \Psi / \chi} 
\eqe 
where $\chi > 0$ and $\Psi$ is a variant of the Gaussian free field on a domain $D\subset\BB C$~\cite{ig1,ig2,ig3,ig4}. 
We associated IG with the central charge $\cc_{\mathrm{M}} < 1 $ which is related to $\chi$ by 
\eqb \label{eq-c-chi}
\chi(\cc_\mathrm{M}) =  \sqrt{\frac{1-\cc_\mathrm{M}}6} \quad \text{ for }\cc_\mathrm{M} \leq 1, \qquad\qquad\qquad \cc_\mathrm{M}(\chi) = 1 -  6 \chi^2 \quad \text{ for } \chi \geq 0.
\eqe
IG surfaces  with central charge $\cc_{\mathrm{M}}$ play a central role in the study of SLE$_\kappa$ when $\cc_\mathrm{SLE}(\kappa) = \cc_{\mathrm{M}}$. Roughly speaking, the flow lines of the vector field~\eqref{eqn-ig-vector-field}, i.e., the solutions to the formal differential equation $\frac{d}{dt} \eta(t) = e^{i \Phi(\eta(t))/\chi}$, are SLE$_\kappa(\rho)$ curves where $\kappa \in (0,4)$ satisfies $\cc_{\mathrm{SLE}}(\kappa) = \ccM$. 
If $\kappa' = 16/\kappa > 4$ is the other solution to $\cc_{\mathrm{SLE}}(\kappa) = \ccM$, then SLE$_{\kappa'}(\rho')$ curves can instead be viewed as ``counterflow lines" of this same vector field.

We will also consider IG surfaces in the case when $\chi = 0$, i.e., $\cc_{\mathrm{M}} = 1$ and $\kappa = 4$. In this case, we view SLE$_4$ curves as level lines of the field $\Psi$ (as in~\cite{ss-contour}), instead of as flow lines of the vector field~\eqref{eqn-ig-vector-field}. See Section~\ref{sec-sle-prelim} for background on $\SLE_\kappa(\rho)$, flow lines, counterflow lines, and level lines.

Similarly to the case of LQG, we define imaginary geometry surfaces rigorously in terms of a conformal coordinate change rule. Previous works using imaginary geometry have considered IG surfaces described by a single field. In this paper we will need to consider IG surfaces described by multiple fields, which are associated with a vector of central charge values.

\begin{definition}\label{def-IG}
	Let $\cc_1, \dots, \cc_n \leq 1$ and write $\ul \cc = (\cc_1, \dots, \cc_n)$. 
	Consider tuples  $(D,\Psi_1, \dots, \Psi_n)$ where $D\subset \C$ is a bounded open set and $\Psi_1, \dots, \Psi_n$ are distributions defined on $D$.  A \textbf{$\ul \cc$-IG surface} is an equivalence class of such tuples under the equivalence relation $\sim_{\ul \cc}$ where $(D, \Psi_1, \dots, \Psi_n) \sim_{\ul \cc} (\wt D, \wt \Psi_1, \dots, \wt \Psi_n)$ if there exists an integer $m$ and a homeomorphism $f:  \C \to \C$ with $f(D) = \wt D$  that is conformal on $D$ such that $\Psi_i = \wt \Psi_i \circ f - \chi(\cc_i) \arg f'  + 2\pi \chi(\cc_i) m$ for all $i$. 
\end{definition}

Here, when $D$ is simply connected, $\arg f'$ is continuous version of the argument, so $\arg f'$ is uniquely specified up to an additive constant in $2\pi \Z$. For general $D$,   $\arg f': D \to \R$ can still be uniquely defined up to an element of $2\pi \Z$ because $f$ is a homeomorphism of $ \C$, see  Section~\ref{subsec-arg} for details.  
Definition~\ref{def-IG} does not depend on the choice of $\arg f'$ since $m$ can take any integer value.

\begin{remark}\label{rem-multiple-IG-fields}
	The original definition of IG surface  in \cite{shef-zipper, ig1} only considers $n=1$ and makes the integer $m$ implicit in the choice of the $\arg$ function. We make $m$ explicit to clarify the $n > 1$ case.
	A single IG field can be understood as being defined modulo $2\pi \chi(\cc_\mathrm{M})$, but this is no longer the case when there are multiple IG fields. For instance, the difference in boundary values for two IG fields with $\cc = \cc_\mathrm{M}$ is well defined as a real number and not just modulo $2\pi \chi(\cc_\mathrm{M})$. See Section~\ref{subsec-coord-change} for more discussion on differences between our Definitions~\ref{def-QS} and~\ref{def-IG} and other definitions in the literature. 
\end{remark}

\subsubsection{LQG surfaces decorated by imaginary geometry fields}

We now extend Definition~\ref{def-QS} to the setting where a $\ccL$-LQG surface is decorated by {IG fields}, that is,  distributions that transform according to the IG coordinate change rule. 

\begin{definition}\label{def-QS-IG}
	Let $\ul{ \mathbf c} = (\ccL, \mathbf c_1, \dots, \mathbf c_n)$ where $\ccL \geq 25$ and $\mathbf c_i \leq 1$ for each $i=1,\dots,n$. Consider tuples $(D, \Phi, \Psi_1, \dots, \Psi_n)$ where $D \subset \C$ is a bounded open set and $\Phi, \Psi_1, \dots, \Psi_n$ are distributions defined on $D$. We define an equivalence relation where $(D, \Phi, \Psi_1, \dots, \Psi_n) \sim_{\ul {\mathbf c}} (\wt D, \wt \Phi, \wt \Psi_1, \dots, \wt \Psi_n)$ if there exists an integer $m$ and a homeomorphism $f:  \C \to  \C$ with $f(D) = \wt D$ that is conformal on $D$ such that 
\eqb \label{eqn-qs-ig-coord}
\Phi = \wt \Phi \circ f + Q(\ccL) \log |f'| \quad \text{and} \quad \Psi_i = \wt \Psi_i \circ f - \chi(\mathbf c_i) \arg f' + 2\pi \chi(\mathbf c_i) m ,\quad \forall i=1,\dots,n .
\eqe  
We call an equivalence class under $\sim_{\ul{\mathbf c}}$ a \textbf{(generalized) $\ccL$-LQG surface decorated by IG fields with central charges $\mathbf c_1, \dots, \mathbf c_n$}. 
\end{definition}

An LQG surface decorated by $n$ IG fields can be thought of as a continuum analog of a random planar map decorated by $n$ statistical physics models. Indeed, for certain decorated random planar maps of this type, the LQG surface should describe the scaling limit of the underlying random planar map and the IG fields should describe the scaling limit of some notion of ``height function" associated with the statistical physics models. See, e.g.,~\cite{blr-dimer-universality} for an example of a situation where an imaginary geometry field arises as the scaling limit of a height function.   
 
We can extend Definition~\ref{def-QS-IG} to additionally keep track of sets and curves, by requiring that the homeomorphism identifies corresponding objects. Precisely, consider $n+4$-tuples \\$(D, \Phi, \Psi_1, \dots, \Psi_n, (K_i)_{i\in I} , (\eta_j)_{j\in J})$ where $(D,\Phi,\Psi_1,\dots,\Psi_n)$ is as in Definition~\ref{def-QS-IG}, $(K_i)_{i\in I}$ is a collection of subsets of $\ol D$, and $(\eta_j)_{j\in J}$ is a collection of parametrized curves in $\ol D$. For two such $n+4$-tuples, we say that 
\eqbn
(D, \Phi, \Psi_1, \dots, \Psi_n, (K_i)_{i\in I} , (\eta_j)_{j\in J} ) \sim_{\ul{\mathbf c}}(\wt D, \wt \Phi, \wt \Psi_1, \dots, \wt \Psi_n, (\wt K_i)_{i\in I} , (\wt \eta_j)_{j\in J} )
\eqen if for some $f$ and $m$ as in Definition~\ref{def-QS-IG}, we have that~\eqref{eqn-qs-ig-coord} holds and also
\eqb \label{eqn-qs-ig-coord-decorated}
f(K_i) = \wt K_i ,\quad \forall i  \in I \quad \text{and} \quad f \circ \eta_j = \wt \eta_j ,\quad \forall j \in J .
\eqe
In the case where $K_i = \{ z_i\}$ is a single point, we simply write $z_i$ rather than $\{ z_i\}$. 

We will often omit the subscript in $\sim_{\ccL}$ and $\sim_{\ul{\mathbf c}}$ and just write $\sim$ when the choice of $\ccL$ or $\underline{\mathbf c}$ is clear from context. 

The main LQG and IG surfaces we will work with in this paper are as follows.

\begin{defn} \label{def-lqg-disk}
For $\ccL > 25$, let $\mathrm{UQD}_{\ccL}$ be the law of the $\ccL$-LQG disk $(\D, \Phi , -\bbi)/{\sim}_{\ccL}$ conditioned to have unit boundary length and having one marked boundary point sampled from its boundary length measure. This LQG surface describes the  scaling limit of planar maps with the disk topology. It was introduced in  \cite[Section 4.5]{wedges}, and can alternatively be defined via Liouville CFT~\cite{hrv-disk,cercle-quantum-disk}; see Definition~\ref{def-disk} for a precise description. Similarly, let $\mathrm{UQD}_{\ccL}^\bullet$ be the law of the LQG surface $(\D,\Phi,-\bbi,0)/{\sim}_{\ccL}$ where $(\D,\Phi,-\bbi)/{\sim}_{\ccL}$ is sampled from the re-weighted probability measure $\frac{\mu_\Phi(\D)}{\mathbb E[\mu_\Phi(\D)]}\cdot \mathrm{UQD}_{\ccL}$ (where $\mu_\Phi$ is the LQG area measure) and 0 corresponds to a point sampled from $\mu_\Phi$, normalized to be a probability measure.
\end{defn}
In this paper we will not need to work with the precise descriptions of $\mathrm{UQD}_{\ccL}$ or $\mathrm{UQD}_{\ccL}^\bullet$.

	Let $\mathfrak h : \D \to \R$ be the harmonic function whose boundary value at $z \in \partial \D$ is $\arg(\bbi z)$, where arg takes values in $[0,2\pi)$. On the boundary, $\mathfrak h(z)$ equals the  counterclockwise tangent direction at $z$, with a discontinuity at $-\bbi$.
\begin{defn} \label{def-ig-disk}
For $\cc_{\mathrm{M}} \leq 1$, let $\mathrm{IG}_{\cc_\mathrm{M}}$ denote the law of the $\cc_\mathrm{M}$-IG surface $(\D, h + \chi(\cc_{\mathrm{M}})\mathfrak h, -\bbi)$ on the one-pointed domain $(\D, -\bbi)$ where $h$ is a zero boundary GFF on $\D$. Similarly, for $\ul \cc = (\cc_1, \dots, \cc_n)$, let $h_1, \dots, h_n$ be $n$ independent zero boundary GFFs on $\D$, and let $\mathrm{IG}_{\ul \cc}$ be the law of the $\ul \cc$-IG surface $(\D, h_1 + \chi(\cc_1) \mathfrak h, \dots, h_n + \chi(\cc_n) \mathfrak h, -\bbi)/{\sim_{\ul \cc}}$. 
\end{defn}

	See Lemma~\ref{lem-ig-bdy-cond} for a description of $\mathrm{IG}_{\cc_\mathrm{M}}$ in other one-pointed domains. In particular, for the unbounded one-pointed domain $(\bbH, \infty)$ the IG surface corresponding to  $\mathrm{IG}_{\ul\cc}$ would be $(\bbH, h_1, \dots, h_n, \infty)/{\sim_{\ul\cc}}$ where the $h_i$ are independent zero boundary GFFs, see Section~\ref{subsec-coord-change} for details.

 \begin{remark}
 	The zero boundary GFF, viewed modulo conformal maps, is the special case of $\mathrm{IG}_\ccM$ where $\ccM = 1$. Indeed,  $\chi (1) = 0$ so the $(\ccM = 1)$-IG coordinate change rule is $(D, \Psi)\sim_{\ccM=1} (\wt D, \wt \Psi)$ if $\Psi = \wt \Psi \circ f$ for some homeomorphism $f: \C \to \C$ which is conformal on $D$. Thus, for a Jordan domain $D$ and $x \in \partial D$, if $\Psi$ is a zero boundary GFF on $D$ then the law of $(D, \Psi, x)/{\sim_{\ccM = 1}}$ is $\mathrm{IG}_{\ccM= 1}$.
 \end{remark}

\subsection{Main results}
\label{subsec-intro-results}

The main results of this paper take the following form. Suppose we have an LQG disk decorated by a vector of imaginary geometry fields and a set $K$ (which could be an SLE curve, an LQG metric ball, etc.), as in Definition~\ref{def-QS-IG} and the discussion just after. 
If the sum of the central charges of all of these objects is 26, then the IG-decorated LQG surfaces obtained by restricting our fields to the complementary connected components of $K$ are conditionally independent given, roughly speaking, the IG-decorated LQG surface obtained by restricting the fields to an infinitesimal neighborhood of $K\cup \partial \D$.

For concreteness, we only state our results for the LQG disk, but we expect that similar statements hold for LQG surfaces with other topologies; see Problem~\ref{prob-other-surfaces}.
 
Our first main result gives a conditional independence statement of the above type when we cut by a chordal SLE curve. See Figures~\ref{fig-disk-welding} and~\ref{fig-thms} (a). We first discuss the setup. 

Suppose $\ccL > 25$ and $\cc_\mathrm{SLE}, \cc_1, \dots, \cc_n \leq 1$ satisfy 
\eqb \label{eqn-cc-sum}
\ccL + \cc_\mathrm{SLE}+ \cc_1 + \cdots + \cc_n = 26 .
\eqe 
Let $(\D, \wt \Phi, -\bbi)$ be an embedding of a central charge-$\ccL$ LQG disk with unit boundary length (i.e., a sample from $\mathrm{UQD}_{\ccL}$; see Definition~\ref{def-lqg-disk}), let $x \in \partial \D$ be the point such that the two boundary arcs from $-\bbi$ to $x$ each have LQG length $\frac12$, and apply a conformal map fixing $-\bbi$ and sending $x \mapsto \bbi$ to get an embedding $(\D, \Phi, -\bbi)$ of the same LQG surface. For this embedding, the LQG lengths of the boundary arcs from $-\bbi$ to $\bbi$ are equal. Equivalently, $(\D,\Phi , -\bbi , \bbi)$ is an embedding of a doubly marked LQG disk with left and right boundary lengths each equal to $1/2$.

Independently from $\Phi$, let $(\D, \Psi_1, \dots, \Psi_n, -\bbi)$ be an embedding of an imaginary geometry surface sampled from $\mathrm{IG}_{(\cc_1, \dots, \cc_n)}$ (Definition~\ref{def-IG}).
Also let $\kappa  > 0$ satisfy $\cc_{\mathrm{SLE}}(\kappa) = \cc_{\mathrm{SLE}}$ and let $K$ be the trace of an independent $\SLE_\kappa(\rho_L; \rho_R)$ curve in $(\D, -\bbi, \bbi)$ (with force points immediately to the left and right of its starting point) such that 
\eqb \label{eqn-rho-values}
\min(\rho_L, \rho_R) > -2 \quad \text{and} \quad \rho_L+ \rho_R \in \{-2,\kappa - 6\} .
\eqe 
This constraint on $\rho_L,\rho_R$ ensures that $K$ can be realized as a flow line or counterflow line of an imaginary geometry surface as in Definition~\ref{def-ig-disk} embedded into $(\BB D , -\bbi)$, see~\cite[Theorem 1.1]{ig1} or Section~\ref{sec-sle-prelim}.

Recall that we want to condition on the IG-decorated LQG surface obtained by restricting $\Phi ,\Psi_1,\dots,\Psi_n$ to an infinitesimal neighborhood of $K\cup \partial \D$. 
To make sense of this, we will define the $\sigma$-algebra we want to condition on as the intersection of $\sigma$-algebras obtained by restricting our fields to neighborhoods of $K\cup \partial \D$. 
This is similar to how one defines the $\sigma$-algebra generated by the restriction of the GFF to a closed set when talking about local sets (see~\cite[Section 3.3]{ss-contour}).

The LQG metric $\mathfrak d_\Phi$ provides a convenient way of defining neighborhoods of $K\cup \partial \D$ in a way which is intrinsic to the LQG surface.
We write $\cB_r(z ; {\mathfrak d_\Phi})$ to denote the ${\mathfrak d_\Phi}$-metric ball centered at $z$ having radius $r$, and more generally, for a set $S\subset\D$, we write 
\eqb \label{eqn-lqg-nbd}
\cB_r(S; {\mathfrak d_\Phi}) = \{ z \: : \: {\mathfrak d_\Phi}(z, S) < r\} . 
\eqe 
We then define (using the notational convention of Remark~\ref{remark-subsurface}) the $\sigma$-algebra 
\eqb \label{eqn-restriction-sigma-algebra}
\cF 
= \sigma\left( \left(  \cB_{0+}(K \cup \bdy\D ; \mathfrak d_\Phi) , \Phi   , \Psi_1, \dots, \Psi_n\right)/{\sim} \right) 
:= \bigcap_{\eps > 0} \sigma((\cB_\eps(K \cup \partial \D; {\mathfrak d_\Phi}), \Phi, \Psi_1, \dots, \Psi_n, K,  -\bbi)/{\sim}).
\eqe 
See Section~\ref{subsec-sigma} for a precise definition of $\sigma$-algebras generated by decorated LQG surfaces.
 We use the LQG metric $\mathfrak d_\Phi$ to give a concise definition of $\cF$ in~\eqref{eqn-restriction-sigma-algebra}, but there are also other equivalent definitions. For example, $\cF$ can equivalently be defined using space-filling SLE segments rather than the LQG metric, see Lemma~\ref{lem-equiv-F}.

We emphasize that $\cF$ contains much less information than $\sigma( K \cup \bdy\D , \Phi|_K   , \Psi_1|_K , \dots, \Psi_n|_K )$ since we are viewing everything modulo conformal maps.  
For example, in the closely related case where $\kappa = \gamma^2 \in (2,4)$,  $\Phi$ is an embedding of the so-called weight 4 $\gamma$-LQG disk with left and right boundary lengths both equal to 1, $K$ is an independent $\SLE_\kappa$ curve, and we have no auxiliary imaginary geometry fields (i.e., $n=0$), then it is possible to show\footnote{This follows from two inputs. Firstly, conditioned on $\nu_\Phi(K)$, cutting by $K$ gives two independent LQG disks with two marked boundary points and boundary arcs of lengths 1 and $\nu_\Phi(K)$ \cite{ahs-welding}. Secondly, if $(\D, \wt \Phi, -\bbi)/{\sim}$ is an LQG disk, then $\sigma((\partial \D, \wt \Phi)/{\sim}) = \sigma(\nu_{\wt\Phi}(\partial \D))$; this can be obtained from \cite[Theorem 5.1]{msw-non-simple-cle} by exploring a small neighborhood of $\partial \D$ using conformal percolation interfaces.} that $\cF = \sigma(( \cB_{0+}(K \cup \partial \D; \mathfrak d_\Phi) , \Phi)/{\sim})$ equals $\sigma(\nu_\Phi(K))$, where $\nu_\Phi$ denotes $\gamma$-LQG length.
We expect that similar statements also hold in other situations where $\kappa=\gamma^2$ and there are no auxiliary imaginary geometry fields. See also Problem~\ref{prob-exact}. 

\begin{figure}[ht!]
	\begin{center}
		\includegraphics[scale=0.45]{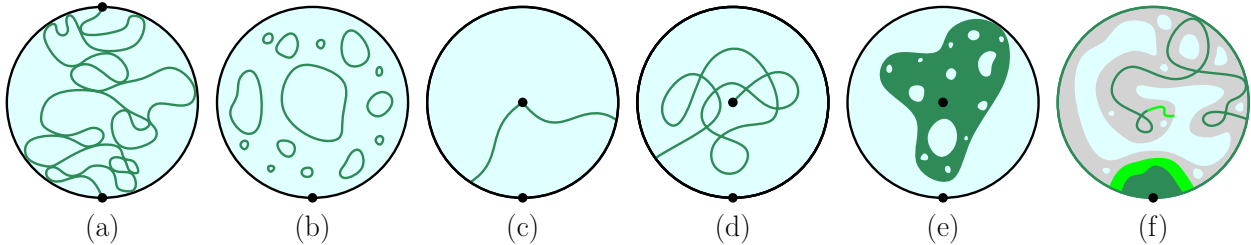}%
		\caption{\label{fig-thms} For (a)--(e), $K$ is the green set. Our main results state that conditioned on an infinitesimal LQG neighborhood of $K\cup \partial \D$ viewed as an LQG surface decorated by IG fields $\Psi_1, \dots, \Psi_n$, the complementary LQG surfaces decorated by IG fields are conditionally independent. \textbf{(a)} $K$ is an $\SLE_\kappa$-type curve for $\kappa > 0$ (Theorem~\ref{thm-sle-chordal}); in the figure we have $\kappa > 4$. \textbf{(b)} $K$ is the gasket of a $\mathrm{CLE}_\kappa$ for $\kappa \in [4,8)$ (we have drawn simple loops for clarity, but in actuality CLE$_\kappa$ loops intersect themselves and each other for $\kappa \in (4,8)$).  \textbf{(c)} $K$ is the union of a pair of $\SLE_\kappa$-type flow lines of an IG field started at a bulk point where $\kappa <  4$ (Theorem~\ref{thm-sle-interior}). \textbf{(d)} $K$ is the image of a Brownian motion run until it hits the boundary (Theorem~\ref{thm-bm}). \textbf{(e)} $K$ is an LQG metric ball grown until it hits $\partial \D$ (Theorem~\ref{thm-ball}). \textbf{(f)} A conditional independence statement (Theorem~\ref{thm-peeling}) also holds for the LQG-IG peeling processes $(K_\cdot)_{[0,T]}$ of  Definition~\ref{def-quantum-peeling}. 
			For each $t, \delta > 0$, conditioned on $T \geq t$, let $t'$ be the earlier of $T$ and the time the process hits $\partial \cB_\delta(K_t; {\mathfrak d_\Phi})$. Here $K_t$ is dark green, $K_{t'} \backslash K_t$ is light green, and $\cB_\eps(K_t; {\mathfrak d_\Phi})$ is gray and green.
			For LQG-IG peeling processes,  
			$(\cB_\delta(K_t; {\mathfrak d_\Phi}), \Phi, (K_\cdot)_{[0,t']})/{\sim}$ is conditionally independent of $(\Phi, \Psi_1, \dots, \Psi_n)$ given $(\cB_\delta(K_t; {\mathfrak d_\Phi}), \Phi, (K_\cdot)_{[0,t]})/{\sim}$. 
		}
	\end{center}
\end{figure}

\begin{thm}\label{thm-sle-chordal} 
In the setting described just above, the IG-decorated LQG surfaces $(U, \Phi, \Psi_1, \dots, \Psi_n)/{\sim}$ parametrized by the connected components\footnote{To be precise, one assigns indices to the connected components $U_i$ in some way measurable with respect to $\cF$, for instance in decreasing order of their central charge-$\ccL$ LQG boundary lengths with respect to $\Phi$. Then Theorem~\ref{thm-sle-chordal} states that the IG-decorated LQG surfaces $(U_i, \Phi, \Psi_1, \dots, \Psi_n)/{\sim}$ are conditionally independent given $\cF$.} $U$ of $\D \backslash K$ (Definition~\ref{def-QS-IG}) are conditionally independent given the $\sigma$-algebra $\cF$ of~\eqref{eqn-restriction-sigma-algebra}. 
\end{thm}

\begin{remark} \label{remark-total-central-charge}
Conditional independence statements of the type proven in this paper have been expected to be true at least for several years, e.g., based on Markovian properties of decorated random planar maps (see Appendix~\ref{sec-rpm}). 
We first heard about such properties from Sheffield in private communication. 

However, proofs of such statements turned out to be rather elusive, in large part because it is not clear where the condition that the total central charge is 26 appears in the continuum setting (except in the matched case~\eqref{eqn-matched}). 
In this paper, the total central charge condition arises by combining the relationship between SLE and LQG in the matched case; a certain locality property for imaginary geometry fields with central charge $\cc_{\mathrm{M}} = 0$ (Theorem \ref{thm-loc-dom}), which allows us to ignore such fields when proving conditional independence statements; and a ``rotational invariance" property for the total central charge associated with a vector of IG fields (Proposition~\ref{prop-ig-rotate}). See Section~\ref{sec-outline} for details.
 
Another way to see the central charge in the continuum, which may provide an alternative route to results of the type proven in this paper, is to weight the law of LQG surfaces by the Brownian loop soup partition function, in an appropriate regularized sense. See~\cite{apps-central-charge} for details.  
\end{remark}

Instead of using $\SLE_\kappa$, we can cut by a \textbf{conformal loop ensemble} (CLE$_\kappa$), a random countable collection of non-nested loops in a simply connected domain $D\subset\mathbb C$ which locally look like SLE$_\kappa$ curves \cite{shef-cle, shef-werner-cle}. See Figure~\ref{fig-thms} (b).
The \textbf{gasket} of the CLE$_\kappa$ is the closure of the union of the loops.

\begin{thm}\label{thm-cle}
	Theorem~\ref{thm-sle-chordal} holds if we replace $K$ by the gasket of a $\mathrm{CLE}_\kappa$ in $\D$ for $\kappa \in [4,8)$, sampled independently from everything else.
\end{thm}

The full parameter range for $\mathrm{CLE}_\kappa$ is $\frac83 < \kappa < 8$. We expect that Theorem~\ref{thm-cle} can also be proven for the remaining case $\kappa \in (\frac83, 4)$ by similar arguments, see Remark~\ref{rem-cle}.

Our next theorem applies to SLE started at the bulk marked point of $\mathrm{UQD}_{\ccL}^\bullet$ which we sample via imaginary geometry, see Figure~\ref{fig-thms} (c). This setup is particularly interesting since the outer boundary of space-filling SLE run until it hits the marked point is a pair of flow lines started at that point~\cite[Section 1.2.3]{ig4}, see Section~\ref{sec-sle-prelim} for IG background.

\begin{thm}\label{thm-sle-interior}
Suppose $\kappa < 4$.
Let $\ccL > 25$ and $\cc_1, \dots, \cc_n \leq 1$ satisfy $\ccL + \cc_\mathrm{SLE}(\kappa) + \cc_1 + \cdots + \cc_n = 26$. Let $(\D, \Phi, 0, -\bbi)$ be an embedding of a unit boundary length LQG disk with a marked interior point, i.e., a sample from $\mathrm{UQD}_{\ccL}^\bullet$ (Definition~\ref{def-lqg-disk}).  Independently, let $(\D, \Psi_1, \dots, \Psi_n, \Psi,  -\bbi)$ be an embedding of a sample from $\mathrm{IG}_{(\cc_1, \dots, \cc_n,  \cc_\mathrm{SLE}(\kappa) )}$ (Definition~\ref{def-ig-disk}). Let $K$ be the union of a pair of flow lines of $\Psi$ started at $0$ in $\D$ and run until they hit the boundary, with two distinct specified angles~\cite[Theorem 1.1]{ig4}. Let $\cF =  \sigma(( \cB_{0+}(K \cup \partial \D ; \mathfrak d_\Phi) , \Phi, \Psi_1, \dots, \Psi_n,   0)/{\sim})$ be defined as in~\eqref{eqn-restriction-sigma-algebra} (with an extra marked point at 0). Then the decorated LQG surfaces $(U, \Phi, \Psi_1, \dots, \Psi_n)/{\sim}$ parametrized by the connected components $U$ of $\D \backslash K$ are conditionally independent given $\cF$. 
\end{thm}
In Theorem~\ref{thm-sle-interior} we exclude the $\kappa = 4$ case since the natural analog of $K$ in this case is not a pair of curves, but rather a collection of nested ``level loops". This is related to the fact that the $\kappa = 4$ analog of space-filling $\SLE_\kappa$ is not a curve \cite{ahps-critical-mating}.
Next, in the case when $\cc_{\mathrm{SLE}}=0$, instead of cutting using an $\SLE_\kappa$-type curve, we can use Brownian motion, see Figure~\ref{fig-thms}  (d). We understand Brownian motion to have central charge 0, see Remark~\ref{rem-discrete-cc} for details.

\begin{thm}\label{thm-bm}
	Suppose $\ccL> 25$ and $\cc_1, \dots, \cc_n \leq 1$ satisfy $\ccL+ \cc_1 + \cdots + \cc_n = 26$. Let $(\D, \Phi, 0, -\bbi)$ be an embedding of a sample from $\mathrm{UQD}_{\ccL}^\bullet$ (Definition~\ref{def-lqg-disk}). Independently, let $(\D, \Psi_1, \dots, \Psi_n,  -\bbi)$ be an embedding of a sample from $\mathrm{IG}_{(\cc_1, \dots, \cc_n)}$ (Definition~\ref{def-ig-disk}). Let $K$ be the trace of a Brownian motion started at $0$ and stopped when it hits $\partial \D$, sampled independently from everything else. Let $\cF =  \sigma(( \cB_{0+}(K \cup \partial \D ; \mathfrak d_\Phi) , \Phi, \Psi_1, \dots, \Psi_n, K, 0)/{\sim})$ be defined as in~\eqref{eqn-restriction-sigma-algebra}. Then the decorated LQG surfaces $(U, \Phi, \Psi_1, \dots, \Psi_n)/{\sim}$ parametrized by the connected components $U$ of $\D \backslash K$ are conditionally independent given $\cF$. 
\end{thm}
 
Our next result applies when we cut by an LQG metric ball, and is illustrated in Figure~\ref{fig-thms} (e). 

\begin{thm}\label{thm-ball}
	Theorem~\ref{thm-bm} holds if we instead let $K$ be the ${\mathfrak d_\Phi}$-metric ball centered at $0$ grown until it hits $\partial \D$, i.e., $K = \ol{\cB_{{\mathfrak d_\Phi}(0, \partial \D)}(0; {\mathfrak d_\Phi})}$. 
\end{thm}

There are also analogs of Theorem~\ref{thm-bm} and~\ref{thm-ball} for other growth processes started from 0 and stopped upon hitting $\bdy\D$ which depend on the LQG surface in a local manner. Some examples include \textbf{LQG harmonic balls} defined in~\cite{bg-harmonic-ball}, which describe the scaling limit of internal DLA on random planar maps~\cite{bg-idla}; and finite unions of Brownian motion paths, LQG metric balls, and harmonic balls. See Theorem~\ref{thm-main-pt} for a general statement.

Finally, we can define exploration processes that depend only on an infinitesimal neighborhood of the explored region; see Figure~\ref{fig-thms}(f). Conditional independence results also hold for such processes. 

\begin{definition}\label{def-quantum-peeling}
	Suppose $\ccL > 25$ and $\cc_1, \dots, \cc_n \leq 1$ satisfy $\ccL + \cc_1 + \cdots + \cc_n = 26$. Let $(\D, \Phi, -\bbi)$ be an embedding of a sample from $\mathrm{UQD}_{\ccL}$ (Definition~\ref{def-disk}) and independently let $(\D, \Psi_1, \dots, \Psi_n, -\bbi)$ be an embedding of a sample from $\mathrm{IG}_{(\cc_1, \dots, \cc_n)}$ (Definition~\ref{def-ig-disk}). Let $(K_t)_{t \leq T}$ be an increasing process of compact connected subsets of $\ol \D$ with random duration $T < \infty$, such that $K_0 = \partial \D$, and the process is continuous with respect to the Hausdorff metric. We call $(K_t)_{t \leq T}$ an \textbf{LQG-IG peeling process} if the following holds. 
	
	Let $\delta > 0$. For $t \leq T$ define $\cF_t^\delta = \sigma(({\cB_\delta(K_t\cup \partial \D; {\mathfrak d_\Phi})}, \Phi, \Psi_1, \dots, \Psi_n, (K_s)_{s \leq t})/{\sim})$. Let $\tau$ be any stopping time for $\cF_t^\delta$. Let $\tau' = T \wedge \inf\{t\geq \tau\: : \: K_t \cap  \partial \cB_\delta(K_\tau; {\mathfrak d_\Phi}) \cap \D \neq \emptyset\}$. Conditioned on $\cF_\tau^\delta$, the decorated LQG surface $({\cB_\delta(K_\tau \cup \partial \D; {\mathfrak d_\Phi})}, \Phi, \Psi_1, \dots, \Psi_n, (K_t)_{t \leq \tau'})/{\sim}$ is conditionally independent from $(\Phi, \Psi_1, \dots, \Psi_n)$.  
\end{definition}

Roughly speaking, $(K_t)_{t\leq T}$ is an LQG-IG peeling process if the growth of $K_t$ depends locally on $\Phi,\Psi_1,\dots,\Psi_n$, modulo LQG / IG coordinate change. Examples of LQG-IG peeling processes include LQG metric balls started at boundary points (due to the locality property of the LQG metric, see Section~\ref{sec-metric-prelim}), chordal SLE$_6$ curves sampled independently from everything else (due to the locality property of SLE$_6$~\cite{lawler-book}), and flow lines or counterflow lines of $\Psi_1,\dots,\Psi_n$ started from boundary points (since flow and counterflow lines depend on the IG field in a local way). The reason for the name ``LQG-IG peeling process" is that these processes are in some sense a continuum analog of ``peeling processes" for random planar maps~\cite{angel-peeling}, except that the growth is allowed to depend on the IG fields $\Psi_1,\dots,\Psi_n$, instead of just on the LQG field $\Phi$.

\begin{thm}\label{thm-peeling}
	In the setting of Definition~\ref{def-quantum-peeling}, conditioned on $\cF := \bigcap_{\delta > 0} \cF_T^\delta$, the decorated LQG surfaces $(U, \Phi, \Psi_1, \dots, \Psi_n)/{\sim}$ parametrized by the connected components $U$ of $\D \backslash K_T$ are conditionally independent.
\end{thm}

\subsection{Towards a Markovian string theory}
\label{sec-string-theory}

A potential application of the present work is to construct a \textbf{Markovian string theory} corresponding to bosonic string theory, a problem posed by Sheffield in private communication. To motivate this, we begin with a lower-dimensional analogy, where a particle is moving in $d$-dimensional space. In classical mechanics, the trajectory $\phi:[0,1] \to \R^d$ of the particle is deterministic and minimizes the action $S(\phi)$. In quantum mechanics, however, one has to consider a formal ``sum over all possible trajectories'' called a path integral, given by $e^{\bbi S(\phi)} D\phi$ where $D\phi$ is the formal ``uniform measure on the space of functions $[0,1] \to \R^d$''. Performing Wick rotation corresponds to considering the real-valued ``measure'' $e^{-S_\mathrm{Euc}(\phi)}D\phi$ where $S_\mathrm{Euc}$ is the Euclidean action corresponding to $S$. When $S_\mathrm{Euc}(\phi) = \int_0^1 \frac12 (\frac{d \phi}{dt})^2 \, dt$, this measure can be interpreted as the law of Brownian motion, a canonical model for the trajectory of a particle in $\R^d$. 
	
We now turn to the setting of bosonic string theory in $\R^d$ for integer $d \geq 1$. A string is a map from the circle $\BB S^1$ to $\R^d$, and its trajectory is described by a map $\phi: \BB S^1 \times [0,1] \to \R^d$. The resulting surface $\phi(\BB S^1 \times [0,1])$ is called the \textbf{worldsheet}. Polyakov \cite{polyakov-qg1} proposed that after Wick rotation, the worldsheet is a random surface described by LQG with central charge $\ccL = 26-d$ decorated by $d$ independent Gaussian free fields. Namely, the worldsheet is the embedding of the LQG surface in $\R^d$ using the $d$ fields as coordinates. The physics and mathematics literature largely considers this construction at the level of partition functions (see e.g.\ \cite{nakayama-lqg, grv-higher-genus}): rather than studying an actual embedding in $\R^d$, one simply weights the law of the LQG surface by $(\det \Delta)^{-d/2}$, the partition function of $d$ free fields. On the other hand, in analogy with the previous paragraph, one could attempt to construct a random string trajectory in $\R^d$ from the worldsheet, and show that it satisfies a natural Markov property, analogous to the Markov property of Brownian motion. 

The present work represents substantial progress since it gives Markov properties when LQG is coupled with multiple random fields. 
The main additional step needed to get from the results of this paper to the construction of a Markovian string theory in $\R^d$ for $d\in \{1,\dots,25\}$ is to extend our results to the case of LQG with $\ccL \in [1,25)$ (see Problems~\ref{prob-critical} and~\ref{prob-supercritical}). 
Indeed, for $d\in \{1,\dots,25\}$, we also have $\ccL = 26-d \in \{1,\dots,25\}$. 
This means that the LQG involved in Polyakov's bosonic string theory is in the supercritical (a.k.a.\ strongly coupled) phase. This phase is much more mysterious than the subcritical or critical cases when $\ccL \geq 25$, even at a physics level of rigor. However, recent mathematical progress on supercritical LQG has been made, e.g., in~\cite{ghpr-central-charge,dg-supercritical-lfpp,pfeffer-supercritical-lqg,dg-uniqueness,ag-supercritical-cle4}.

Defining Markovian string trajectories in $\R^d$ may also be of interest in the probabilistic study of Yang-Mills theory, due to various manifestations of gauge-string duality and formulas expressing Wilson loop observables in terms of sums over discrete surfaces in $\BB R^d$. See~\cite{chatterjee-lattice-gauge,mp-matrix-group,cps-random-surface-ym} for some recent rigorous results along these lines.

\subsection{Main ideas of the proof}
\label{sec-outline}

Theorems~\ref{thm-sle-chordal}--\ref{thm-peeling}
are all special cases of Theorem~\ref{thm-main} for $\mathrm{UQD}_{\ccL}$ and Theorem~\ref{thm-main-pt} for $\mathrm{UQD}_{\ccL}^\bullet$.
The proof of Theorem~\ref{thm-main} can be broken down into three steps, each generalizing the result of the previous. Theorem~\ref{thm-main-pt} is similarly shown. 
These steps are carried out in Sections~\ref{subsec-n=1},~\ref{subsec-locality} and~\ref{subsec-main-proof} respectively. 

\medskip \noindent \textbf{Step 1.}
We prove a discretized version of Theorem~\ref{thm-main} with $n = 1$ (a single IG field with central charge $26 - \ccL$) using the seminal \textbf{mating-of-trees} theorem  which describes a $\gamma$-LQG surface decorated by an independent space-filling $\SLE_{\kappa' = 16/\gamma^2}$ curve in terms of planar Brownian motion \cite{wedges}. The desired conditional independence is a consequence of the Markov property of Brownian motion and a combinatorial independence property for submaps of a planar map arising from discretizing LQG decorated by SLE (a minor variant of the \textbf{mated-CRT map}).
The details of this part of the argument are given in Section~\ref{sec-n=1}.  
The combinatorial independence property is a generalization of a certain conditional independence property for uniform meanders, which is of independent interest and which we prove in Appendix \ref{sec-meander}.

\medskip \noindent \textbf{Step 2.} We use the \textbf{locality} of $\mathrm{IG}_{\mathbf c= 0}$ to extend to the setting of $n \geq 1$ IG fields with central charges $(\mathbf c_1, \dots, \mathbf c_n) = (26 - \ccL, 0, \dots, 0)$. This locality property, stated as Theorem~\ref{thm-loc-dom}, is due to \cite{dubedat-coupling} and can be viewed as an IG field variant of the locality of $\SLE_6$.

\medskip \noindent \textbf{Step 3.} 
Suppose we have a vector of $n$ independent IG fields with central charges $\cc_1,\dots,\cc_n \leq 1$ satisfying $\cc_1 + \cdots + \cc_n = 26 - \ccL$. We \textbf{rotate} this vector of IG fields to get another vector of $n$ independent IG fields with central charges $(26-\ccL,0,\dots,0)$. We then apply the case treated in Step 2 to conclude. 

The rotation is accomplished by means of the following elementary but extremely useful observation, which states that $n$ independent IG fields can be linearly transformed into another collection of $n$ independent IG fields with the same total central charge.

\begin{proposition}[Rotation of IG fields]\label{prop-ig-rotate}
	Fix $n \geq 1$ and let $A$ be an $n \times n$ orthogonal matrix. Let $\cc_1, \dots, \cc_n \leq 1$ and let $\chi_i = \chi(\cc_i)$. Suppose $(\D, \Psi_1, \dots, \Psi_n, -\bbi)$ is an embedding of a sample from $\mathrm{IG}_{(\cc_1, \dots, \cc_n)}$ (Definition~\ref{def-ig-disk}). Define
		\begin{equation}\label{eq-rotate}
		\left( \begin{array}{c}
			\wh\Psi_1 \\ 
			\vdots \\ 
			\wh\Psi_n
		\end{array}  \right) 
		= A \left( \begin{array}{c}
			\Psi_1 \\ 
			\vdots \\ 
			\Psi_n
		\end{array}  \right), \qquad 	\left( \begin{array}{c}
		\wh\chi_1 \\ 
		\vdots \\ 
		\wh\chi_n
	\end{array}  \right) 
= A \left( \begin{array}{c}
\chi_1 \\ 
\vdots \\ 
\chi_n
\end{array} \right)
	\end{equation}
and suppose $\wh \chi_i \geq 0$ for all $i$. 
Let $\wh \cc_i = 1 - 6 \wh \chi_i^2$. Then $\sum_i \wh \cc_i = \sum_i \cc_i$, the IG surface \\$(\D, \wh \Psi_1, \dots, \wh \Psi_n, -\bbi)/{\sim_{(\wh \cc_1, \dots, \wh \cc_n)}}$ is well defined (independently of the choice of embedding  $\Psi_1,\dots,\Psi_n$)  and its law is 
$\mathrm{IG}_{(\wh \cc_1, \dots, \wh \cc_n)}$. 

\end{proposition}
\begin{proof}
	Since orthogonal matrices preserve the Euclidean norm, we have 
	\eqbn
	\sum_{i=1}^n \wh \cc_i = n - 6\sum_{i=1}^n \wh \chi_i^2 = n - 6 \sum_{i=1}^n \chi_i^2 = \sum_{i=1}^n \cc_i .
	\eqen
	 Next, we verify that multiplication by $A$ as in~\eqref{eq-rotate} transforms a $(\cc_1, \dots, \cc_n)$-IG surface into a $(\wh \cc_1, \dots, \wh \cc_n)$-IG surface; that is, multiplication by $A$ is compatible with the coordinate change rule in Definition~\ref{def-IG}. If $(D, \Psi_1, \dots, \Psi_n) \sim_{(\cc_1, \dots, \cc_n)} (\wt D, \wt \Psi_1, \dots, \wt \Psi_n)$ and $f $ is a homeomorphism witnessing this equivalence with $f(D)  = \wt D$, then 
	\[A \left( \begin{array}{c}
		\Psi_1 \\ 
		\vdots \\ 
		\Psi_n
	\end{array}  \right) = A \left( \begin{array}{c}
		\wt \Psi_1 \circ f - \chi_1 (\arg f' - 2\pi m)\\ 
		\vdots \\ 
		\wt \Psi_n \circ f - \chi_n (\arg f' - 2\pi m)
	\end{array}  \right) = A \left( \begin{array}{c}
		\wt \Psi_1 \circ f\\ 
		\vdots \\ 
		\wt \Psi_n \circ f
	\end{array}  \right) - \left( \begin{array}{c}
		\wh \chi_1 (\arg f' - 2\pi m) \\
		\vdots\\
		\wh \chi_n (\arg f' - 2\pi m)
	\end{array} \right).\]
	
	Finally, we identify the law of $(\D, \wh \Psi_1, \dots, \wh \Psi_n, -\bbi)/{\sim_{(\wh \cc_1, \dots, \wh \cc_n)}}$. Without loss of generality we may assume that $\Psi_i = h_i + \chi(\cc_i) \mathfrak h$ for $i = 1, \dots, n$ where $h_1, \dots, h_n$ are independent zero boundary GFFs  on $\D$ and $\mathfrak h$ is the harmonic function on $\D$ with boundary data given by $\arg(\bbi \cdot)$, as in Definition~\ref{def-ig-disk}. 
	
	Let $(\wh h_1,\dots,\wh h_n)$ be the vector of fields obtained by multiplying $(h_1,\dots,h_n)$ by $A$, as in~\eqref{eq-rotate}. We claim that $(\wh h_1,\dots,\wh h_n)$ also has the law of $n$ independent zero boundary GFFs. Indeed, a zero boundary GFF can be sampled as $\sum_{i=1}^\infty g_i f_i$ where $(f_i)$ is an orthonormal basis of the Hilbert space completion of compactly supported functions on $\D$ with finite Dirichlet energy and $g_i$ are independent standard Gaussians \cite{shef-gff}. If $\ul g$ is a standard Gaussian in $\R^n$ then $\ul g \stackrel d= A\ul g$. Applying this to the coefficients in the orthonormal basis expansions of $h_1,\dots,h_n$ gives the claim.
	
	From~\eqref{eq-rotate}  we have $\wh \Psi_i = \wh h_i + \wh \chi_i \mathfrak h$ for all $i$, so the law of  $(\D, \wh \Psi_1, \dots, \wh \Psi_n)/{\sim_{(\wh \cc_1, \dots, \wh \cc_n)}}$ is $\mathrm{IG}_{(\wh \cc_1, \dots, \wh \cc_n)}$. 
\end{proof}

\begin{remark}
The idea of rotating a vector of fields via an orthogonal matrix to get a new vector of fields with the same total charge is also used in~\cite{ag-supercritical-cle4}. In that paper, a key idea is to rotate the field describing a critical ($\gamma=2$) LQG disk and a zero-boundary GFF to get two LQG disks of central charges $\ccL , 26-\ccL \in (1,25)$. See also~\cite[Appendix A]{ag-supercritical-cle4} for a rotational invariance property analogous to Proposition~\ref{prop-ig-rotate} for vectors of fields sampled from the infinite measure on LQG disks or spheres. 
\end{remark}

\section{Preliminaries}
\label{sec-prelim}

\subsection{Comments on conformal coordinate change}\label{subsec-coord-change}

Our definitions of LQG and IG surfaces (Definitions~\ref{def-QS},~\ref{def-IG} and~\ref{def-QS-IG}) differ slightly from definitions elsewhere in the literature \cite{shef-zipper, ig1,  wedges}.  In this subsection we will explain the extent of the difference and the reason why we use a different definition. 
For concreteness we will focus on the case of IG surfaces, but a similar discussion also applies for LQG surfaces. 

In most other works on imaginary geometry, one only considers simply connected IG surfaces, and defines an IG surface as an equivalence class of pairs $(D, \Psi)$ where $D \subset \C$ is a \emph{simply connected} and \emph{possibly unbounded} open set and $\Psi$ is a distribution on $D$, and $(D, \Psi) \sim_{\cc}' (\wt D, \wt \Psi)$ if there is a conformal map $f: D \to \wt D$ such that $\Psi = \wt \Psi \circ f - \chi(\cc) \arg f' + 2\pi \chi(\cc) m$ for some integer $m$. Note that \emph{$f$ need not extend to a homeomorphism of $\C$}.

For the set of pairs $(D, \Psi)$ where $D$ is a Jordan domain, the relations $\sim_\cc$ and $\sim_\cc'$ are equivalent: $(D, \Psi) \sim_\cc (\wt D, \wt \Psi)$ if and only if $(D, \Psi) \sim_\cc' (\wt D, \wt \Psi)$. On the larger set of pairs $(D, \Psi)$ where $D$ is bounded and simply connected, $(D, \Psi) \sim_\cc (\wt D , \wt \Psi)$ implies $(D, \Psi) \sim_\cc'(\wt D, \wt \Psi)$ (but the reverse implication does not hold, since $\bdy D$ can have ``exterior" self-intersections). Consequently, 
for any bounded simply connected $D$, any object determined by $(D, \Psi)/{\sim_\cc'}$ is also determined by $(D, \Psi)/{\sim_\cc}$. 
In particular, flow and counterflow lines of a simply connected IG surface $(D, \Psi)/{\sim_\cc}$ are determined by  $(D, \Psi)/{\sim_\cc}$ (see Section~\ref{sec-sle-prelim}). To discuss these objects, we will sometimes want to embed in the upper half-plane $\bbH$ via $\sim_\cc'$ to match the convention of \cite{ig1}.

We now explain why we use $\sim_\cc$ rather than $\sim_\cc'$. 
For non-simply connected domains, 
in order to define an IG surface as an equivalence class of pairs $(D, \Psi)$, we need $\arg f'$ to be well-defined modulo a single global additive multiple of $2\pi \chi$. This is not necessarily true if we only require $f:D \to \wt D$ to be conformal, but as we see in Proposition~\ref{prop-arg-f'} there is a canonical way to define $\arg f'$ (modulo additive multiple of $2\pi$) when $f$ extends to a homeomorphism of $\C$. This explains the condition on $f$ in Definition~\ref{def-IG}. The homeomorphism condition distinguishes the point $\infty$ on the Riemann sphere, and we choose to work with bounded domains to avoid interacting with this point.

We similarly work with a modified definition of LQG surface, so as to make them compatible with IG surfaces (as in Definition~\ref{def-QS-IG}).

\begin{remark} \label{remark-slitted}
An important configuration that arises in \cite{ig1,shef-zipper, wedges} is the slitted domain $\bbH\backslash \eta$ where $\eta$ is a simple curve from $0$ to the interior of $\bbH$ arising as a segment of a flow line of an IG field $\Psi$. Crucially, one can embed the slitted surface in $\bbH$ using $\sim_\cc'$, i.e., $(\bbH \backslash \eta, \Psi) \sim_\cc' (\bbH, \wt \Psi)$ for some $\wt \Psi$. This is not possible using our equivalence relation $\sim_\cc$. We will not need to consider slitted domains in this paper since we will always cut our domains by curves which disconnect the domain (usually into connected components which are each Jordan domains), and we view the connected components as parametrizing separate LQG/IG surfaces.
\end{remark}

\subsection{$\sigma$-algebras generated by LQG surfaces}
\label{subsec-sigma}
One can define a topology (and hence a $\sigma$-algebra) for LQG surfaces with specific conformal structures (e.g., simply connected LQG surfaces) by conformally mapping to a canonical reference domain, see, e.g.,~\cite[Section 2.2.5]{gwynne-miller-char}. It is not obvious how to define a topology for, say, LQG surfaces parametrized by domains with infinitely many complementary connected components. However, it is straightforward to define the $\sigma$-algebras generated by such LQG surfaces, as we now explain.

Consider the set $\mathcal X$ of pairs $(D, \phi)$ where $D$ is a bounded open set and $\phi$ is a generalized function belonging to the local Sobolev space $H^{-1}_\mathrm{loc}(D)$. We equip $\mcl X$ with some reasonable topology. For concreteness, we use the topology whereby $(D_n,\phi_n)$ converges to $(D,\phi)$ if and only if $(\C\cup\{\infty\})\setminus D_n \to (\C\cup\{\infty\})\setminus D$ with respect to the Hausdorff distance for the spherical metric on $\C \cup \{\infty\}$; and, for each bounded open set $U$ with $\ol U\subset D$, we have $\phi_n \to \phi$ with respect to the metric on $H^{-1}(U)$  (i.e., the one induced by the operator norm on the dual of $H^1(U)$). One can check that this topology is separable and metrizable. 
 
Let $\mathcal B$ be the Borel $\sigma$-algebra for this topology on $\mathcal X$, and let $\mathcal B'$ be the sub-$\sigma$-algebra generated by Borel measurable functions $F: \mathcal X \to \R$ such that $F(D, \phi) = F(\wt D, \wt \phi)$ whenever $(D, \phi)\sim_\ccL (\wt D, \wt \phi)$,  where $\sim_\ccL$ is as in Definition~\ref{def-QS}.  For a random LQG surface $\cS$, we define $\sigma(\cS)$ to be the collection of events of the form $\{ (D, \Phi) \in E\}$ where $E \in \mathcal B'$ and $(D, \Phi)$ is any embedding of $\cS$. Informally, the information carried by $\sigma(\cS)$ is precisely the values of the functions of $(D, \Phi)$ which are invariant under LQG coordinate change. 

More generally, this same approach allows us to define the $\sigma$-algebras generated by LQG surfaces decorated by IG fields, sets, and/or curves as in~\eqref{eqn-qs-ig-coord-decorated}.

\subsection{Unit boundary length LQG disk}
We will not need to use the precise definition of the law $\mathrm{UQD}_{\ccL}$ of the unit boundary length LQG disk, but we include one possible definition (which comes from~\cite[Corollary 1.2]{cercle-quantum-disk}) for completeness. 
See~\cite{wedges,hrv-disk,ahs-integrability} for alternative definitions. 

\begin{definition}\label{def-disk}
	Let $\ccL> 25$ and let $\gamma \in (0,2)$ satisfy $\ccL= 1 + 6(\frac\gamma2+\frac2\gamma)^2$. 
  Let $h$ be a Neumann GFF on $\D$ normalized to have average zero on $\partial \D$ and let 
  \eqbn
  \wh\Phi_0 = h - \gamma \log |\cdot + \bbi| - \gamma \log |\cdot -1| - \gamma \log |\cdot +1| . 
  \eqen
 Let $\wh\Phi$ be sampled from the law of $\wh\Phi_0$ weighted by $\nu_{\wh \Phi_0}(\partial \D)^{\frac{4}{\gamma^2}-2}/\BB E[\nu_{\wh \Phi_0}(\partial \D)^{\frac{4}{\gamma^2}-2}]$ where $\nu_{\wh \Phi_0}$ is the $\gamma$-LQG boundary length measure, and let $\Phi = \wh \Phi - \frac2\gamma \log \nu_{\wh \Phi}(\partial \D)$. Then $\mathrm{UQD}_{\ccL}$ is the law of the $\gamma$-LQG surface $(\D, \Phi, -\bbi)/{\sim_{\ccL}}$. 
\end{definition}

\subsection{LQG metric}
\label{sec-metric-prelim}

Let $D\subset\BB C$ be an open domain such that Brownian motion started at any point of $D$ a.s.\ exits $D$ in finite time. We say that a random generalized function $\Phi$ on $D$ is a \textbf{GFF plus a continuous function on $D$} if there is a coupling of $\Phi$ with a zero-boundary GFF $\Phi^0$ on $D$ such that a.s.\ $\Phi - \Phi^0$ is a continuous function on $D$ (with no assumption about the behavior of $\Phi-\Phi^0$ near $\bdy D$). We say that $\Phi$ is a \textbf{free-boundary GFF plus a continuous function on $\ol D$} if there is a coupling of $\Phi$ with a free-boundary GFF $\Phi^{\mathrm{Free}}$ on $D$ such that a.s.\ $\Phi -  \Phi^{\mathrm{Free}}$ is a continuous function on $\ol D$.  It is clear from Definition~\ref{def-disk} that if $\Phi$ is an embedding of the LQG disk, then $\Phi$ is a free-boundary GFF plus a continuous function plus finitely many functions of the form $\gamma \log |\cdot-z|$, for $z\in \bdy D$.
 
Let $D\subset\BB C$, and let $\Phi$ be a GFF plus a continuous function on $D$. For each $\ccL > 25$, there is a unique (up to deterministic multiplicative constant) metric $\mathfrak d_\Phi: D \times D \to [0,\infty)$ (i.e., a measurable assignment $\Phi \mapsto \mathfrak d_\Phi$) satisfying a list of natural axioms \cite{dddf-lfpp,gm-uniqueness}, called the \textbf{LQG metric}. Heuristically, $\mathfrak d_\Phi(z,w)$ corresponds to the infimum over paths from $z$ to $w$ of ``the integral of $e^{\xi(\ccL)\Phi}$ along the path'', where $\xi(\ccL) >0$ is a constant defined in terms of the fractal dimension of $\ccL$-LQG. See~\cite{ddg-metric-survey} for a survey of the LQG metric and its properties. The LQG metric can also be defined for $\ccL \in (1,25]$~\cite{dg-supercritical-lfpp,dg-uniqueness}, but we will not need this case here.

We will now review the basic properties of $\mathfrak d_\Phi$ which we will need in this paper.
\medskip

\noindent
\textbf{Topology.} Almost surely, the LQG metric $\mathfrak d_\Phi$ induces the Euclidean topology on $D$. Furthermore, a.s.\ $\frk d_\Phi$ is a length metric, i.e., $\frk d_\Phi(z,w)$ is the infimum of the $\frk d_\Phi$-lengths of paths joining $z$ and $w$~\cite[Axiom I]{gm-uniqueness}.
\medskip

\noindent
\textbf{Coordinate change.} If $f : D \to \wt D$ is a conformal map and $\Phi = \wt \Phi \circ f + Q \log |f'|$, then~\cite{gm-coord-change} almost surely
\eqb  \label{eqn-lqg-metric-coord}
\frk d_{\wt \Phi}(f(z),f(w))= \frk d_\Phi(z,w) ,\quad\forall z,w\in D. 
\eqe 
\medskip

\noindent
\textbf{Locality.} For any open set $U \subset D$, the internal metric $\mathfrak d_\Phi(\cdot,\cdot; U)$ is the metric on $U$ defined by 
\eqb \label{eqn-internal-metric}
\mathfrak d_\Phi(z,w; U) := \inf \mathrm{len}(P; \mathfrak d_\phi) ,\quad \forall z,w\in U
\eqe 
 where the infimum is taken over paths $P$ from $z$ to $w$ in $U$ and $\mathrm{len}(P; \mathfrak d_\Phi)$ denotes the $\mathfrak d_\Phi$-length of $P$. Note that $\mathfrak d_\Phi(z,w;U) \geq \mathfrak d_\Phi(z,w)$.
We will frequently use the locality property of $\mathfrak d_\Phi$~\cite[Axiom II]{gm-uniqueness}: for any deterministic open set $U \subset D$, the internal metric $\mathfrak d_\Phi(\cdot,\cdot; U)$ is a.s.\ given by a measurable function of $\Phi|_U$. 
In particular, for $z \in U$ and $r >0$ the event $\{ \cB_r(z; \mathfrak d_\Phi) 
\subset U\}$ is measurable with respect to $\sigma(\Phi|_U)$, and 
on this event the set $\cB_r (z; \mathfrak d_\Phi)$ is measurable with respect to $\sigma(\Phi|_U)$. 
\medskip

\noindent
\textbf{Boundary extension.}
When $\Phi$ is a free-boundary GFF plus a continuous function on $\ol D$, the LQG metric $\mathfrak d_\Phi$ extends by continuity to a metric $\mathfrak d_\Phi: \ol D \times \ol D \to [0,\infty)$ which induces the Euclidean topology on $\ol D$ \cite[Proposition 1.6]{hm-metric-gluing}. The same is true if we add finitely many functions of the form $\alpha \log |\cdot-z|$ for $\alpha < Q$ and distinct deterministic points $z\in\bdy\BB D$. In particular, the LQG metric associated with an embedding of the quantum disk extends to a metric on $\ol D$ which induces the Euclidean topology. It is immediate that for this extended definition of $\mathfrak d_\Phi$, the locality property stated above still holds for relatively open $U \subset \ol D$.

\subsection{SLE and imaginary geometry}
\label{sec-sle-prelim}
For $\rho_L, \rho_R > -2$, there is a natural variant of $\SLE_\kappa$ called $\SLE_\kappa(\rho_L; \rho_R)$ \cite{lsw-restriction, dubedat-duality, ig1} which is locally absolutely continuous with respect to $\SLE_\kappa$ away from the boundaries. It keeps track of two force points, which we will always assume are immediately to the left and right of the starting point of the curve.  $\SLE_\kappa(\rho_L; \rho_R)$ hits the left boundary if and only if $\rho_L < \frac\kappa2 - 2$, and similarly for the right boundary.  

Imaginary geometry \cite{ig1, ig2, ig3, ig4} studies SLE via its coupling with IG fields, building on~\cite{dubedat-coupling}. We only discuss a few special cases of the general theory here. Let $\kappa\neq 4, \cc = \cc_\mathrm{SLE}(\kappa)$ and let $\rho_L, \rho_R> -2$. Let $\Psi$ be a zero boundary GFF in $\bbH$ plus the harmonic function with boundary values $-\frac{\pi}{\sqrt\kappa}(1 + \rho_L)$ on $(-\infty, 0)$ and $\frac{\pi}{\sqrt\kappa}(1+\rho_R)$ on $(0, \infty)$. Then there is a random curve $\eta$ in $\ol \bbH$ measurable with respect to $\Psi$ whose marginal law is $\SLE_\kappa(\rho_L; \rho_R)$, such that conditioned on $\eta$, the restrictions of $\Psi$ to the connected components $U$ of $\bbH \backslash \eta$ are conditionally independent. The conditional law of $\Psi|_U$ is that of a zero boundary GFF plus a harmonic function with boundary conditions on $\partial U \cap \R$ as above and boundary conditions on $\partial U \cap \eta$ given by a constant plus $-\chi(\cc) \arg f'$ where $f: U \to \bbH$ is a suitable conformal map.
The boundary conditions on $\eta$ are measurable with respect to $\eta$ and are compatible with $\cc$-IG coordinate change (Definition~\ref{def-IG}).

In this coupling, $\eta$ is a \textbf{local set} \cite[Section 3.3]{ss-contour} of $\Psi$: for any relatively open set $U \subset \ol \bbH$ containing neighborhoods of $0$ and $\infty$, the event $\{ \eta \subset U\}$ is measurable with respect to $\Psi|_U$, and on this event $\eta$ is measurable with respect to $\Psi|_U$. 

When $\kappa < 4$, the curve $\eta$ is called a \textbf{flow line} of $\Psi$, and similarly we say an $\SLE_\kappa(\rho_L; \rho_R)$ curve is a flow line of angle $\theta$ if it is a flow line of $\Psi + \theta \chi(\cc)$. When $\kappa > 4$ these curves are instead called \textbf{counterflow lines}\footnote{By convention, in the above coupling $\eta$ is a counterflow line of $-\Psi$ rather than $\Psi$.}. Flow lines can also be started at interior points $z \in \bbH$ and run until they hit $\partial \bbH$, and satisfy properties analogous to those stated above \cite{ig4}. The case $\kappa = 4$ ($\chi = 0$) falls outside the scope of the imaginary geometry framework, but the above statements for chordal $\eta$ still hold, and $\eta$ is called a \textbf{level line} \cite{ss-contour,wang-wu-level-lines}. 

For $\kappa > 4$, given a zero boundary GFF $\Psi$ in $\bbH$ there is a  counterclockwise space-filling $\SLE_{\kappa}$ loop based at $\infty$ measurable with respect to $\Psi$ defined as follows. For each point with rational coordinates $z \in \BB Q^2 \cap \bbH$, let $\eta^{z,L}$ and $\eta^{z, R}$ be the flow lines of $\Psi - \frac\pi{\sqrt\kappa}$ started at $z$ with angles $\theta^L = -\frac\pi2$ and $\theta^R = \frac\pi2$ respectively. We define a total order $\preceq$ on $\BB Q^2 \cap \bbH$ by saying $z \preceq w$ if $z$ lies in a connected component of $\BB H \backslash (\eta^{w,L} \cup \eta^{w, R})$ whose boundary is traced by the left side of $\eta^{w, L}$ and the right side of $\eta^{w, R}$. 
This ordering is well defined due to properties of interacting flow lines, and gives rise to a space-filling loop in $\BB H$. See \cite[Appendix A.3]{bg-LBM} for details.
Since this coupling is constructed via the  flow lines of $\Psi$ viewed as a $\cc_\mathrm{SLE}(\kappa)$-IG field, in any simply connected domain  with one boundary point  we have a coupling of the counterclockwise space-filling $\SLE_{\kappa}$ loop with $\mathrm{IG}_{\mathbf c_\mathrm{SLE}(\kappa)}$. 

Let $\eta$ be a counterclockwise space-filling $\SLE_\kappa$ loop in $\ol \D$, and let $z, z' \in \ol \D$ be distinct points chosen independently of $\eta$. Let $a < b$ be the two times when $\eta$ hits $\{ z, z'\}$. When $\kappa \geq 8$,  a.s.\ the interior of $\eta([a,b])$ is simply connected. On the other hand, when $\kappa \in (4,8)$, a.s.\ the interior of $\eta([a,b])$ has countably many connected components.
See Figure~\ref{fig-discretization}.

\subsection{Total curvature and winding}
\label{subsec-winding}

We recall the notion of total curvature, also known as winding in some probability literature such as \cite{ig1}. This will be used in Section~\ref{subsec-arg} to make sense of LQG / IG coordinate change as in Definitions~\ref{def-IG} and~\ref{def-QS-IG} for domains which are not necessarily connected, and in Section~\ref{subsec-topo} for the main combinatorial argument of the paper (see also Appendix~\ref{sec-meander}).

Suppose $\eta:[0,1] \to \C$ is a regular curve, meaning it is continuously differentiable with nonvanishing derivative. Let $\theta:[0,1] \to \R$ be  the unique continuous function such that $\theta(0) \in [0,2\pi)$ and $e^{-\bbi\theta(t)}\eta'(t) >0 $ for each $t \in [0,1]$. The \emph{total curvature} of $\eta$ is defined to be $\theta(1) - \theta(0)$, and more generally we define the total curvature of $\eta|_{[t,t']}$ to be  $\theta(t') - \theta(t)$. This name comes from the fact that when $\eta$ is twice-differentiable, if $\wh \eta: [0,T] \to \C$ is the parametrization of $\eta$ according to arc length and $k(t) = \frac{\wh \eta''(t)}{\bbi \wh \eta'(t)}\in \R$ is its signed curvature, then 
\[\theta(1) - \theta(0) = \int_0^T k(t) \, dt. \]

		A regular simple closed curve $\eta:[0,T] \to \C$ is a simple loop ($\eta(0) = \eta(T)$) with a continuous nonvanishing derivative such that $\eta'(0) = \eta'(T)$. The following classical fact is often called \emph{Hopf's Umlaufsatz}, see e.g.\ \cite[Theorem 17.4]{tu-differential-geometry}.
		\begin{lemma}\label{lem-hopf}
			The total curvature of a regular simple closed curve is $\pm 2\pi$, with the $+$ sign if the loop is oriented in the counterclockwise direction and $-$ sign otherwise.
		\end{lemma}
	\begin{lemma}\label{lem-total-winding}
		Suppose $\eta_0$ and $\eta_1$ are regular simple curves such that for some $\eps > 0$ we have  $\eta_0|_{[0,\eps]} = \eta_1|_{[0,\eps]}$ and $\eta_0|_{[1-\eps, 1]} = \eta_1|_{[1-\eps, 1]}$, and $\eta_0|_{[\eps, 1-\eps]}$ and $\eta_1|_{[\eps, 1-\eps]}$ are homotopic in the twice-slitted domain $\C \backslash (\eta_0([0,\eps]) \cup \eta_0([1-\eps, 1]))$.    Then the total curvatures of $\eta_0$ and $\eta_1$ agree.
		\end{lemma}
	\begin{proof}
		Let $\eta$ be a simple regular curve from $\eta_0(1)$ to $\eta_0(0)$ which concatenates with $\eta_0$ to form a   regular simple closed curve. By Lemma~\ref{lem-hopf}, the concatenation has total curvature $\pm 2\pi$ with sign depending on the orientation of the loop. Thus, if such an $\eta$ can be chosen such that the loop obtained by concatenating $\eta$ with $\eta_1$ is also simple, then $\eta_0$ and $\eta_1$ have the same total curvature, as needed. 
		
		We now address the general case where no such $\eta$ can be chosen. Our homotopy hypothesis implies that there exists $\ep  >0$ and a  homotopy of \emph{simple} curves $(\eta_t)_{[0,1]}$ from $\eta_0$ to $\eta_1$ such that $\eta_t|_{[0,\eps]}$ and $\eta_t|_{[1-\eps, 1]}$ are continuously differentiable and do not depend on $t \in [0,1]$, see e.g.\ \cite[Section 1.2.7]{fm-primer-mcg} for details on why we can take the curves to be simple. 
		 Pick a finite set of times $0 = t_0 < \cdots < t_k = 1$ such that for each $j = 1, \dots, k$ the points $\eta_0(0)$ and $\eta_0(1)$ lie on the boundary of the same connected component of $\C  \backslash (\eta_{t_{j-1}}([0,1]) \cup \eta_{t_j}([0,1]))$. Let $\wt \eta_0 = \eta_0$ and $\wt \eta_k = \eta_1$.  For $j = 1, \dots, k-1$ let $\wt \eta_j$ be a simple regular curve such that $\wt \eta_j|_{[0,\eps]} = \eta_0|_{[0,\eps]}$ and $\wt \eta_j|_{[1-\eps, 1]} = \eta_0|_{[1-\eps, 1]}$ for all $j$, and $\wt \eta_j|_{[\eps, 1-\eps]}$ is homotopic to $\eta_0|_{[\eps, 1-\eps]}$ in $\C \backslash (\eta_0([0,\eps]) \cup \eta_0([1-\eps, 1]))$. Moreover we assume $\eta_0(0)$ and $\eta_0(1)$ lie on the boundary of the same connected component of $\C \backslash (\wt \eta_{j-1}([0,1]) \cup \wt \eta_j([0,1]))$ for $j = 1, \dots, k$. Such curves $\wt \eta_j$ can be chosen by picking simple regular curves that stay sufficiently close to $\eta_{t_j}$ in the Hausdorff topology for each $j = 1, \dots, k-1$. By the first paragraph, the total curvatures of $\wt \eta_{j-1}$ and $\wt \eta_j$ agree for $j = 1, \dots, k$, and hence the total curvatures of $\wt \eta_0 = \eta_0$ and $\wt \eta_k = \eta_1$ agree, as needed. 
	\end{proof}
	
	By Lemma~\ref{lem-total-winding}, the following definition makes sense.
	
	\begin{definition}\label{def-winding-nonsmooth}
		Suppose $\eta: [0,1] \to \C$ is a (possibly non-smooth) simple curve such that $\eta|_{[0,\eps]}$ and $\eta|_{[1-\eps, 1]}$ are regular for some $\eps>0$. Define the total curvature of $\eta$ to be that of any regular simple curve $\eta^\mathrm{reg}:[0,1] \to \C$ such that  $\eta|_{[0,\eps]} = \eta^\mathrm{reg}|_{[0,\eps]}$ and $\eta|_{[1-\eps, 1]} = \eta^\mathrm{reg}|_{[1-\eps, 1]}$, and $\eta|_{[\eps, 1-\eps]}$ and $\eta^\mathrm{reg}|_{[\eps, 1-\eps]}$ are homotopic in the twice-slitted domain $\C \backslash (\eta([0,\eps]) \cup \eta([1-\eps, 1]))$. 
	\end{definition}

\subsection{IG coordinate change for general domains}\label{subsec-arg}

Suppose $D$ and $\wt D$ are bounded open subsets of $\C$ and $f:  \C \to \C$ is a homeomorphism with $f(D) = \wt D$ which is conformal on $D$. When $D$ is simply connected (which in particular means that it is connected), one can define $\arg f':D \to \R$ uniquely up to an additive constant in $2\pi \Z$.  In this section we extend the definition of $\arg f'$ to general $D$. Roughly speaking the condition that $f$ is a homeomorphism rules out the possibility of ``spinning'' a connected component an arbitrary number of times. This ensures that the definition of an IG surface (Definition~\ref{def-IG}) makes sense even if $D$ is not simply connected, or even connected.

 Here is the extended definition of $\arg f': D \to \R$. First, fix any point $z_0 \in D$ and define  $(\arg f')(z_0)$ to be any real $\theta$ satisfying $f'(z_0) e^{-\bbi \theta} > 0$; this specifies it up to an element of $2\pi \Z$. For $z \in D \backslash \{ z_0\}$, pick a curve $\eta:[0,1] \to \C$ with $\eta(0) = z_0$ and $\eta(1) = z$ and which is regular in neighborhoods of its endpoints. Let $\wt \eta = f \circ \eta$, then we define
\eqb \label{eq-def-arg}
(\arg f')(z) = (\arg f')(z_0) + w(\wt \eta) - w(\eta),
\eqe
where $w$ denotes  total curvature as in  Definition~\ref{def-winding-nonsmooth}. The following result states that this definition is valid. It uses Lemmas~\ref{lem-usual-arg} and~\ref{lem-winding-well-defined} which we state and prove at the end of this section. 
\begin{proposition}\label{prop-arg-f'}
	The above definition of $\arg f': D \to \R$ does not depend on the choices of $z_0$ and $\eta$, and agrees with the usual definition of $\arg f'$ when $D$ is  simply connected.
\end{proposition}
\begin{proof}
	Agreement with the usual definition in the simply connected setting follows from Lemma~\ref{lem-usual-arg}. The non-dependence on $\eta$ is stated as Lemma~\ref{lem-winding-well-defined}. For different choices of base point $z_0$ and $z_0'$, the consistency of the above definition can be shown by concatenating with a curve from $z_0'$ to $z_0$. 
\end{proof}

We now use~\eqref{eq-def-arg} to  explain how total curvature relates to IG boundary conditions. 

\begin{lemma}\label{lem-ig-bdy-cond}
	Let $\cc_\mathrm{M} \leq 1$. 
	Suppose $D \subset \C$ is a simply connected domain whose boundary is a regular simple closed curve, and let $x \in \partial D$. Let $\theta_0 \in [0, 2\pi)$ be the angle such that the tangent to $D$ at $x$ in the counterclockwise direction is parallel to $e^{\bbi \theta_0}$. Let $\mathfrak h_D : D \to \R$ be the harmonic function whose boundary value at  $y \neq x$  is the total curvature of the counterclockwise boundary arc from  $x$ to $y$. If $h$ is a zero boundary GFF in $D$, then the law of $(D, h + \chi(\cc_\mathrm{M})(\mathfrak h_D + \theta_0), x)/{\sim_\cc}$ is $\mathrm{IG}_{\cc_\mathrm{M}}$. 
\end{lemma}
\begin{proof}
	Recall the harmonic function $\mathfrak h: \D \to \R$ in Definition~\ref{def-ig-disk} whose boundary value at $y \in \partial \D$ is the total curvature of the counterclockwise boundary arc from $-\bbi$ to $y$. 
	Let $f: \D \to D$ be a conformal map sending $-\bbi \mapsto x$. By the definition of an IG surface (Definition~\ref{def-IG}), we must show that 
\eqbn
	\chi(\cc_\mathrm{M})\mathfrak h_D \circ f + \chi(\cc_\mathrm{M})\theta_0 - \chi(\cc_\mathrm{M}) \arg f'= \chi(\cc_\mathrm{M})(\mathfrak h + 2 \pi m)
\eqen
for some integer $m$. Since both functions are harmonic and $\arg f'(-\bbi) - \theta_0 \in 2\pi \Z$,  it suffices to show $\arg f'(y) = \arg f'(-\bbi) + \mathfrak h_D(f(y)) - \mathfrak h(y)$ for all  $y \in \partial \D$ with $y \neq -\bbi$. Let $\eta$ (resp.\ $\wt \eta$) be the counterclockwise boundary arc of $\D$ from $-\bbi$ to $y$ (resp.\ arc of $D$ from $x$ to $f(y)$). Since $\mathfrak h(y)$ (resp.\ $\mathfrak h_D(f(y))$) is the total curvature of $\eta$ (resp.\ $\wt \eta$), the claim would follow from~\eqref{eq-def-arg} if that equation applied to the curves $\eta$ and $\wt \eta$. Instead, we approximate $\eta$ and $\wt \eta$ by smooth curves in $\D$ and $D$, where the approximation is with respect to the supremum norm for the curve and its first derivative,  apply~\eqref{eq-def-arg}, and take a limit to conclude. 
\end{proof}

In the rest of this section, we state and prove the lemmas needed for the proof of Proposition~\ref{prop-arg-f'} above. Lemma~\ref{lem-usual-arg} is also stated in \cite[Remark 2.5]{blr-dimer-universality}; we give its proof here for completeness.

\begin{lemma}\label{lem-usual-arg}
	Suppose $D$ is a bounded simply connected open set and let $\arg_0 f': D \to \R$ be a continuous function such that $f'(z) e^{-\bbi (\arg_0 f')(z)} > 0$ for all $z \in D$. Let $\eta:[0,1] \to D$ be a regular curve (not necessarily simple) and $\wt \eta = f \circ \eta$. Then $(\arg_0 f')(\eta(1)) =  (\arg_0 f')(\eta(0)) + w(\wt \eta) - w(\eta)$. 
\end{lemma}
\begin{proof}
		Let $\theta:[0,1] \to \R$ be a continuous function such that $\eta'(t) e^{-\bbi \theta(t)} > 0$ for all $t$, and similarly define $\wt \theta$ for $\wt\eta$. Since $\wt \eta = f \circ \eta$ we have $\wt \eta'(t) = f'(\eta(t)) \eta'(t)$, so $\wt \theta(t) $ and $ (\arg_0 f') (\eta(t)) + \theta(t) $ agree up to an additive integer multiple of $2\pi$. The  function $\wt \theta(t) - (\arg_0 f')(\eta(t)) - \theta(t)$ is continuous in $t$ and takes values in $2\pi \Z$, thus is constant. Identifying its values at $0$ and $1$ gives the result.
\end{proof}

In Lemma~\ref{lem-winding-well-defined} below we do not assume $D$ is simply connected. 
\begin{lemma}\label{lem-winding-well-defined}
For $j = 1,2$ let $\eta_j:[0,1] \to \C$  be a simple curve which is regular in neighborhoods of its endpoints. Suppose  $\eta_1(0) = \eta_2(0)$ and $\eta_2(1) = \eta_2(1)$, and these endpoints lie in $D$. Let $\wt \eta_j = f \circ \eta_j$ for $j = 1,2$. Then $w(\wt \eta_1) - w(\eta_1) = w(\wt \eta_2) - w(\eta_2)$. 
\end{lemma}
\begin{proof}
	Call an intersection point of two curve segments \emph{generic} if at that point a curve crosses from one side to the other side of the other curve.
		For $\eta_1$ and $\eta_2$ as in Lemma~\ref{lem-winding-well-defined}, we use the notation $\eta_1 \perp \eta_2$ to mean that $\eta_1$ and $\eta_2$ intersect only finitely many times, each intersection is generic, the tangent vectors of $\eta_1$ and $\eta_2$ at their starting point are not parallel, and similarly their tangent vectors at their ending point are not parallel.
		
	We first prove the result under the assumption that $\eta_1 \perp \eta_2$. Let $p = \eta_j(0)$ and $q = \eta_j(1)$, and let $\eta_{2,r}(t) = \eta_2(1-t)$ be the time-reversal of $\eta_2$.
 	Let $\eta_{p}$ be a regular curve which intersects itself only at its endpoints  $\eta_{p}(0) = \eta_{p}(1) = p$, is disjoint from $\eta_1$ and $\eta_2$, lies in a simply connected neighborhood of $p$ in $D$, and satisfies $\eta_{2,r}'(1) = \eta_{p}'(0)$ and $\eta_{p}'(1) = \eta_1'(0)$. 
	Let $\eta_q:[0,1] \to \C$ be a regular curve which intersects itself only at its endpoints  $\eta_q(0) = \eta_q(1) = q$, is disjoint from $\eta_1$ and $\eta_2$, lies in a simply connected neighborhood of $q$ in $D$, and satisfies $ \eta_1'(1) = \eta_q'(0)$ and $\eta_q'(1) = \eta_{2,r}'(0)$. See Figure~\ref{fig-arg}. 
	
	 Then $\eta:[0,4] \to \C$ obtained by concatenating $\eta_p, \eta_1, \eta_q, \eta_{2,r}$ is a loop which intersects itself finitely many times, all generically. By standard topological arguments, there exists a homeomorphism $g: \C \to \C$ fixing simply connected neighborhoods of $\eta_p$ and $\eta_q$ in $D$, which sends $\eta$ to a regular loop $\eta^\mathrm{reg} = g\circ \eta$ which intersects itself finitely many times, each intersection being  transversal. 
	 Decompose the loop $\eta^{\mathrm{reg}}$ into four segments $\eta_p, \eta_1^\mathrm{reg}, \eta_q, \eta_{2,r}^\mathrm{reg}$. 
	 
	 Let $\wt \eta_p = f \circ \eta_p, \wt \eta_q = f\circ \eta_q,$ and $  \wt \eta = f \circ \eta$. Let $\wt g$ be a  homeomorphism of $\C$ fixing simply connected neighborhoods of $\wt \eta_p$ and $\wt \eta_q$ in $\wt D$, and sending $\wt \eta$ to a regular simple closed curve $\wt \eta^\mathrm{reg}$ with finitely many self-intersections, all of which are  transversal. Decompose this closed curve into four segments 
	$\wt \eta_p, \wt\eta_1^\mathrm{reg}, \wt\eta_q, \wt\eta_{2,r}^\mathrm{reg} $. Since the orientation-preserving homeomorphism $\wt g \circ f \circ g^{-1}: \C \to \C$ sends $\eta^\mathrm{reg}$ to $\wt \eta^\mathrm{reg}$, and these regular loops only intersect themselves finitely many times, each intersection being transversal and a double-point, 
 by  \cite{whitney-regular} the total curvatures of $\wt\eta^\mathrm{reg}$ and $\eta^\mathrm{reg}$ agree\footnote{\cite[Theorem 2]{whitney-regular} gives a formula for the total curvature of a loop which is invariant under orientation-preserving homeomorphism, see e.g.\ \cite[Theorem 2]{bp-whitney} for the particular formulation we use.
}. 
In other words, 
	\[w(\wt \eta_p) +  w(\wt \eta_1^\mathrm{reg}) + w(\wt \eta_q) + w(\wt \eta_{2, r}^\mathrm{reg})  = w(\eta_p) + w(\eta_1^\mathrm{reg}) + w(\eta_q) + w(\eta_{2, r}^\mathrm{reg}).\]
	By Lemma~\ref{lem-usual-arg} we have $w(\wt \eta_p) = w(\eta_p) $ and $w(\wt \eta_q) = w(\eta_q)$, so 
	$w(\wt\eta_1^\mathrm{reg}) - w(\eta_1^\mathrm{reg}) = w( \eta_{2, r}^\mathrm{reg}) - w(\wt \eta_{2, r}^\mathrm{reg})$. Applying Definition~\ref{def-winding-nonsmooth} gives $w(\wt \eta_1) - w(\eta_1) = w(\eta_{2, r}) - w(\wt \eta_{2, r}) = w(\wt \eta_2) - w(\eta_2)$ as desired. 
	
	Now we address the case where we do not assume $\eta_1 \perp \eta_2$. By  Lemma~\ref{lem-third-curve} just below, we can find a third curve $\eta:[0,1] \to \C$  such that $\eta_1 \perp \eta$ and $\eta_2 \perp \eta$. Then the previous paragraph implies $w(\wt \eta_1) - w(\eta_1) = w(f \circ  \eta) - w(\eta) = w(\wt\eta_2) - w(\eta_2)$ as needed. 
\end{proof}

 \begin{figure}[ht!]
	\begin{center}
		\includegraphics[scale=0.5]{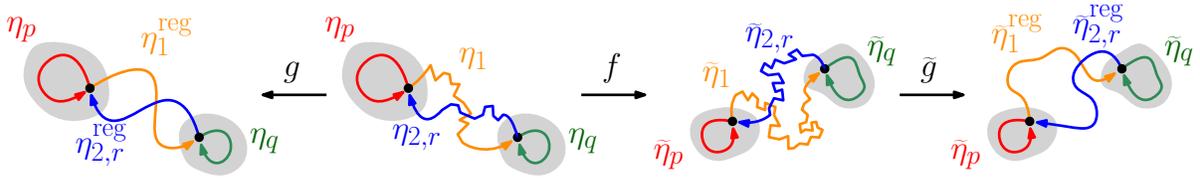}%
		\caption{\label{fig-arg} 
			Figure for the proof of Lemma~\ref{lem-winding-well-defined}. We introduce curves $\eta_p$ and $\eta_q$ in simply connected neighborhoods of $p$ and $q$ (gray) which join with $\eta_1$ and $\eta_{2,r}$ to yield a loop $\eta$ whose restriction to these neighborhoods is regular. Let $g:\C \to \C$ be a homeomorphism fixing these neighborhoods such that $g \circ \eta$ is regular. One can similarly choose $\wt g$ for the loop $f \circ \eta$. 
		}
	\end{center}
\end{figure}

\begin{lemma}\label{lem-third-curve}
	Let $\perp$ be as defined in the proof of Lemma~\ref{lem-winding-well-defined}.	For any curves $\eta_1, \eta_2$ as in Lemma~\ref{lem-winding-well-defined}, there exists a third curve $\eta$ such that $\eta_1 \perp \eta$ and $\eta_2 \perp \eta$. 
\end{lemma}
\begin{proof}
	By applying a suitable homeomorphism if necessary, we may assume that $\eta_1$ is a line segment.

	We first choose the initial and ending segments of $\eta$. Let $\eta|_{[0,1/3]}$ be a line segment with $\eta(0) = \eta_1(0)$ such that  $\eta_1'(0)$ and $\eta_2'(0)$ are not parallel to $\eta'(0)$, and $\eta((0,1/3])$ is disjoint from $\eta_1$ and $\eta_2$. This is possible since $\eta_1$ and $\eta_2$ are regular in neighborhoods of their starting points.
 Similarly, let $\eta|_{[2/3, 1]}$ be a line segment ending at $\eta_1(1)$ such that $\eta_1'(1)$ and $\eta_2'(1)$ are not parallel to $\eta'(1)$, and $\eta([2/3, 1))$ is disjoint from $\eta_1$ and $\eta_2$. 
	
	Next, we say an interval $[a,b]$ is \emph{separating} if $\eta_2(a)$ and $\eta_2(b)$ each lie on $\eta_1$, $\eta_2((a,b))$ is disjoint from $\eta_1$, and the loop obtained as the union of $\eta_2([a,b])$ and the segment of $\eta_1$ joining $\eta_2(a)$ and $\eta_2(b)$ separates $\eta((0,1/3])$ from $\eta([2/3, 1))$. 
	If $[a,b]$ is separating, then  $\mathrm{diam}(\eta_2([a,b])) > \min(\mathrm{diam}(\eta((0,1/3])), \mathrm{diam}(\eta([2/3, 1))))$.  Since $[0,1]$ is compact and $\eta_2$ is continuous, the curve $\eta_2$ is uniformly continuous, thus the diameter lower bound implies there are at most finitely many separating intervals. The initial and ending segments of $\eta$ can then be extended to a simple curve $\eta:[0,1]\to \C$ such that $\eta((0,1))$ is disjoint from $\eta_1$ and crosses $\eta_2$ once for each separating interval. This is the desired curve. 
\end{proof}

\section{Proofs of main theorems}
\label{sec-proofs}

In this section, we prove our most general result Theorem~\ref{thm-main} on the LQG-IG local sets of Definition~\ref{def-eps-local}. The main theorems stated in Section~\ref{subsec-intro-results} are all straightforward consequences of Theorem~\ref{thm-main} and its minor variant Theorem~\ref{thm-main-pt}.

 First, we give the definition of an LQG-IG local set.
 This is an analog of the definition of a local set of the GFF~\cite[Section 3]{ss-contour} in the setting of LQG surfaces decorated by imaginary geometry fields.
  Roughly speaking, a random closed set $K$ is an LQG-IG local set of an IG-decorated LQG surface $\cS$ if, for ``any'' IG-decorated LQG subsurface $\cS'$ of $\cS$ chosen conditionally independently of $K$ given $\cS$, the event \{$K$ lies in $\cS'$\} is conditionally independent of $\cS$ given $\cS'$, and further conditioned on this event, the decorated LQG surface $\cS'$ further decorated by $K$ is conditionally independent of $\cS$. To make sense of a ``random LQG sub-surface'' $\cS'$, we will discretize our setup at scale $\eps$ and define $\eps$-LQG-IG local sets.
  
 Let $\mathbf c_{L} >25$, $n \geq 1$ and $\mathbf c_1, \dots, \mathbf c_n \leq 1$ satisfy $\ccL + \mathbf c_1 + \cdots + \mathbf c_n = 26$.  Let $(\D, \Phi, - \mathbbm i)$ be an embedding of a sample from $\mathrm{UQD}_{\ccL}$ (Definition~\ref{def-lqg-disk}), and independently let $(\D, \Psi_1, \dots, \Psi_{n}, -\bbi)$ be an embedding of a sample from $\mathrm{IG}_{(\cc_1, \dots, \cc_n)}$ (Definition~\ref{def-ig-disk}). Let $\eta'$ be a counterclockwise space-filling $\SLE$ with central charge $(26 - \ccL)$ measurable with respect to $\Psi_1 \dots, \Psi_n$ in the following way. Let $\wh \cc_1 = 26 - \ccL$ and $\wh \cc_2 = \cdots = \wh \cc_n = 0$, and let $\chi_i = \chi( \cc_i) $ and $\wh \chi_i = \chi(\wh \cc_i)$ for $i = 1, \dots, n$. Let $A$ be an orthogonal matrix which maps $(\chi_1, \dots, \chi_n)$ to $( \wh \chi_1, \dots, \wh \chi_n)$, and let $(\wh \Psi_1, \dots, \wh \Psi_n)$ be the image of $(\Psi_1, \dots, \Psi_{n})$ under $A$ as in~\eqref{eq-rotate}, so by Proposition~\ref{prop-ig-rotate} the law of $(\wh \Psi_1, \dots, \wh \Psi_n)$ is $\mathrm{IG}_{(\wh \cc_1, \dots, \wh \cc_n)}$. Let $\eta'$ be the counterclockwise space-filling SLE curve coupled with $\wh \Psi_{1}$ as in Section~\ref{sec-prelim}. 
 
In Definition~\ref{def-eps-local}, we will discretize the above setup 
 using a finite set of points $\Lambda \subset \D$ independent of $\eta'$, chosen using the LQG area measure $\mu_\Phi$. The curve $\eta'$ is split by the points of $\Lambda$  into $|\Lambda| + 1$ curve segments which we call $\eta'_0, \dots, \eta'_{|\Lambda|}$; each $\eta'_i$ is defined on a time interval $[0, t_i]$. For $i = 0, \dots, |\Lambda|$, let $D_i = \eta_i'([0,t_i])$ be the region traced out by $\eta_i'$, let $a_i  = \eta'_i(0)$ and $d_i = \eta_i'(t_i)$ be the start and end points of $\eta_i'$ respectively, and let $b_i$ (resp.\ $c_i$) be the furthest point on the clockwise (resp.\ counterclockwise) arc\footnote{For $\kappa' \geq 8$ the boundary of $\partial D_i$ is simple. For $\kappa' \in (4,8)$, the clockwise arc of $\partial D_i$ from $a_i$ to $d_i$ is a nonsimple curve, constructed by ordering the points of $\partial D_i$ hit by the left side of $\eta'_i$ according to the last time they are hit.} of $\partial D_i$ from $a_i$ to $d_i$ such that the boundary arc from $a_i$ to $b_i$ (resp.\ $c_i$) is a subset of $\ol{(\partial \D\cup \bigcup_{j \leq i} D_i)}$. 
See Figure~\ref{fig-discretization}. 

 \begin{figure}[ht!]
	\begin{center}
		\includegraphics[scale=0.49]{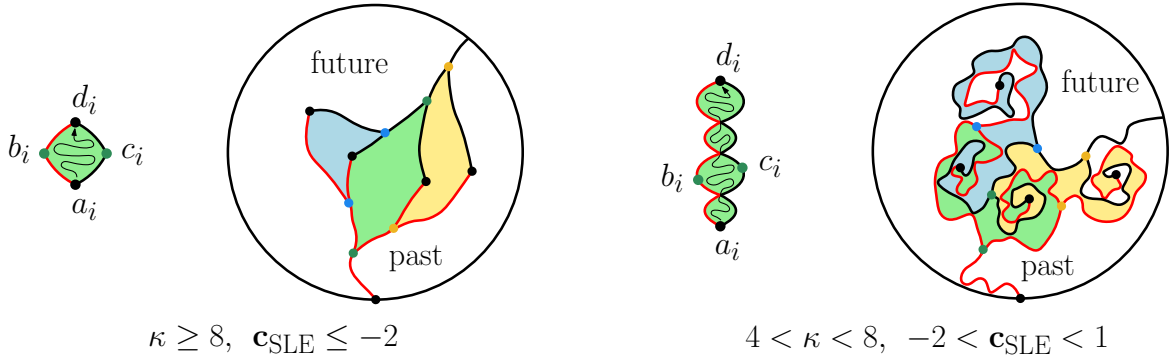}%
		\caption{\label{fig-discretization} 
			Let $\kappa > 4$ satisfy $\cc_\mathrm{SLE}(\kappa) = 26 - \ccL$. In each figure the space-filling $\SLE_\kappa$ loop  $\eta'$ traces three $D_i$ regions in the order yellow, green, blue. The points of $\Lambda$ are shown in black. In particular $a_i$ and $d_i$ are black, whereas $b_i$ and $c_i$ are shown in color. The left boundaries of the $D_i$ (clockwise boundary arcs from $a_i$ to $d_i$) are red, while the right boundaries are black. 
			\textbf{Left:} When $\kappa \geq 8$, each of the $D_i$ is simply connected. \textbf{Right:} When $\kappa \in (4,8)$, each $D_i$ has countably many components, and arises from a sequence of simply-connected regions ordered by the time they are traced by $\eta'$. 
		}
	\end{center}
\end{figure}

 \begin{definition}\label{def-eps-local}
 	Let $\eps > 0$. 
 	An \textbf{$\eps$-LQG-IG local set} $K \subset \ol \D$ is a random compact set coupled with $(\Phi,  \Psi_1, \dots, \Psi_{n}, \eta')$ such that $K \cup \partial \D$ is a.s.\ connected, and 
 	the following holds. See Figure~\ref{fig-thm-indep}  (left). 
 	
 	Independently of $(\Psi_1, \dots, \Psi_n, K, \eta')$, 
 	sample a Poisson point process $\Lambda \subset \D$ with intensity measure $\eps^{-1} \mu_\Phi$ and define $D_i, \eta_i', a_i, b_i, c_i, d_i$ as above. Independently of everything else sample a nonnegative integer $T \geq 0$ with $\P[T = t] = 2^{-t-1}$. Let $I_0$ be the set of indices $i$ such that  $D_i \cap \partial \D \neq \emptyset$, and for $j \geq 0$  inductively define $I_{j+1}$ from $I_j$ by independently sampling a point $p_{j+1} \in (\partial \bigcup_{i \in I_j} D_i) \backslash \partial \D$ from the quantum length measure, then setting $I_{j+1} = I_j \cup \{ i \: : \: p_{j+1} \in \partial D_i\}$. We stop the process either at time $T$ or at the time $t$ when $\bigcup_{i \in I_t} D_i = \ol \D$, whichever is earlier; call this time $\tau$. 
 	
 	Conditioned on $( \bigcup_{i \in I_\tau} D_i, \Phi, \Psi_1, \dots, \Psi_n, \{ (\eta_i', a_i, b_i, c_i, d_i) : i \in I_{\tau}\}, \Lambda\cap \bigcup_{i \in I_\tau} D_i)/{\sim}$, the event $\{K \subset \bigcup_{i \in I_\tau} D_i\}$ is conditionally independent of $(\Phi, \Psi_1, \dots, \Psi_n, \eta',  \Lambda)$. 
 	Moreover,  further conditioning on $\{K \subset \bigcup_{i \in I_\tau} D_i\}$, the decorated LQG surface $( \bigcup_{i \in I_\tau} D_i, \Phi, K)/{\sim}$  is conditionally independent of $(\Phi,  \Psi_1, \dots, \Psi_n, \eta', \Lambda)$. 
 \end{definition}
The procedure of Definition~\ref{def-eps-local} can be viewed as a semi-continuous analog of \emph{peeling} for random planar maps (see, e.g., \cite{angel-uihpq-perc}). Indeed, let $\cM_t$ be the decorated LQG sub-surface parametrized by $\bigcup_{i \in I_t} D_t$, then if $t < \tau$, we explore from a point on $\partial \cM_t$ chosen in a way only depending on $\cM_t$ to obtain $\cM_{t+1}$. 

In Definition~\ref{def-eps-local}, we could equivalently replace $(\bigcup_{i \in I_\tau} D_i, \Phi, K)/{\sim}$ with 
\eqbn
( \bigcup_{i \in I_\tau} D_i, \Phi, \Psi_1, \dots, \Psi_n, \{ (\eta_i', a_i, b_i, c_i, d_i) : i \in I_{\tau}\}, \Lambda\cap \bigcup_{i \in I_\tau} D_i, K)/{\sim} ,
\eqen
since the latter decorated LQG surface is determined by $(\bigcup_{i \in I_\tau} D_i, \Phi, K)/{\sim}$ and 
\eqbn
( \bigcup_{i \in I_\tau} D_i, \Phi, \Psi_1, \dots, \Psi_n, \{ (\eta_i', a_i, b_i, c_i, d_i) : i \in I_{\tau}\}, \Lambda\cap \bigcup_{i \in I_\tau} D_i)/{\sim} . 
\eqen

Definition~\ref{def-eps-local} is closely related to the definition of local sets of the GFF given in  \cite[Lemma 3.9]{ss-contour} but more complicated since we consider LQG sub-surfaces rather than planar domains. Roughly speaking, replacing the deterministic open sets in their definition of GFF local sets with the randomly chosen LQG surface $( \bigcup_{i \in I_\tau} D_i, \Phi, \Psi_1, \dots, \Psi_n, \{ (\eta_i', a_i, b_i, c_i, d_i) : i \in I_{\tau}\}, \Lambda)/{\sim}$ gives Definition~\ref{def-eps-local}.

 We emphasize that in the above definition, for the decorated LQG surface we condition on, $ \{ (\eta_i', a_i, b_i, c_i, d_i) : i \in I_{\tau}\}$ is an \emph{unordered} set of tuples, and each tuple is not labelled by its index $i$.

 \begin{figure}[ht!]
 	\begin{center}
 		\includegraphics[scale=0.6]{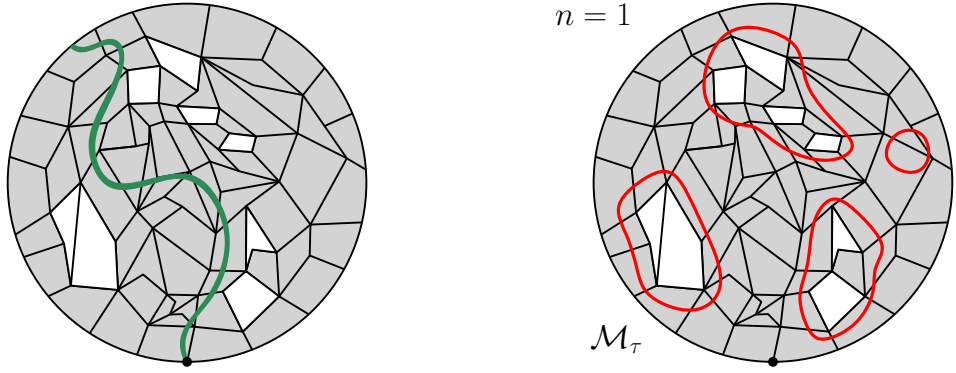}%
 		\caption{\label{fig-thm-indep} 
 			Combinatorially, the planar map structure depicted here cannot arise in our setup, but see Definition~\ref{def-T} for a description of the possible planar maps and  Figures~\ref{fig-BM} and~\ref{fig-template} (left) for examples.
 			\textbf{Left.}
 			Figure for Definition~\ref{def-eps-local}. Using $\eta'$ and $\Lambda$, the disk $\D$ is divided into regions $D_i$. A random domain $\bigcup_{i \in I_\tau} D_i$ (gray) is explored independently of $K$ (green).  We say $K$ is an $\eps$-LQG-IG local set if, roughly speaking, the event $\{ K \subset \bigcup_{i \in I_\tau} D_i\}$ is conditionally independent of the unexplored decorated LQG surfaces (white) given the explored decorated LQG surface (gray, not decorated by $K$), and further conditioning on  $\{ K \subset \bigcup_{i \in I_\tau} D_i\}$, the explored decorated LQG surface further decorated by $K$ is conditionally independent of the unexplored decorated LQG surfaces. 
 			\textbf{Right.} 
 			Figure for Proposition~\ref{prop-eps-indep}. 
 			The explored LQG surface $\cM_\tau$ is depicted in gray. Assuming there is only one IG field, and the loops $L_i$ (red) on $\cM_\tau$ are chosen  in a way measurable with respect to $\cM_\tau$, Proposition~\ref{prop-eps-indep} says the encircled IG field-decorated LQG surfaces are conditionally independent given $\cM_\tau$. 
 		}
 	\end{center}
 \end{figure}

 \begin{definition}\label{def-quantum-local}
 	A random compact set $K \subset \ol \D$ coupled with $(\Phi, \Psi_1, \dots, \Psi_n, \eta')$  is called an \textbf{LQG-IG local set} if it is an $\eps$-LQG-IG local set for all $\eps > 0$. 
 \end{definition}

Examples of LQG-IG local sets include flow lines of $\Psi_1$, $\SLE_6$, Brownian motion, and $\mathfrak d_\Phi$-metric balls run until stopping times intrinsic to the decorated LQG surface. See the proofs of Theorems~\ref{thm-sle-chordal},~\ref{thm-bm} and~\ref{thm-ball} in Section~\ref{subsec-other} for details. 

The following conditional independence result is the main result of this section. 
 
 \begin{thm}\label{thm-main}
 	Suppose $(\Phi, \Psi_1, \dots, \Psi_n, \eta', K)$ are as in Definition~\ref{def-quantum-local}. Let $\cF = \bigcap_{\eps > 0} \sigma((\cB_\eps(K \cup \partial \D; {\mathfrak d_\Phi}), \Phi, \Psi_1, \dots, \Psi_n, K, -\bbi)/{\sim})$. Then the decorated LQG surfaces $(U, \Phi, \Psi_1, \dots, \Psi_n)/{\sim}$ parametrized by the connected components $U$ of $\D \backslash K$ are conditionally independent given $\cF$. 
 \end{thm}

 There is a variant of Theorem~\ref{thm-main} for $\mathrm{UQD}_{\ccL}^\bullet$ (Theorem~\ref{thm-main-pt}) whose statement we defer to Section~\ref{subsec-other} to avoid notational confusion. Its proof is essentially the same as that of Theorem~\ref{thm-main}. 
 
 \begin{remark}\label{rem-mated-crt}
 	Definition~\ref{def-eps-local} gives a discretization of SLE-decorated LQG via a Poisson point process of intensity $\eps^{-1} \mu_\Phi$. 
 	A related discretization where $\eta'$ is segmented such that each $D_i$ has fixed LQG area is used to motivate and analyze the \emph{mated-CRT map} \cite{wedges, ghs-dist-exponent, gms-tutte}. The Poissonian formulation simplifies our arguments due to its locality property: the restrictions of a Poisson point process to fixed subdomains are independent. 
 \end{remark}

In Sections~\ref{subsec-n=1}--\ref{subsec-main-proof} we carry out the three steps described in the introduction to prove Theorem~\ref{thm-main}. 
In Section~\ref{subsec-n=1} we state Proposition~\ref{prop-eps-indep}  which is a special case of Theorem~\ref{thm-main} where $n=1$ and $K$ is a particular set defined in terms of mating of trees; its proof is deferred to Section~\ref{sec-n=1}. 
In Section~\ref{subsec-locality} we use the locality of $\mathrm{IG}_{\mathbf c=0}$ fields to extend to the $n \geq 1$ setting where all but one IG field has central charge 0. Finally, in Section~\ref{subsec-main-proof} we use Proposition~\ref{prop-ig-rotate} to generalize to the setting of arbitrary central charges, completing the proof of Theorem~\ref{thm-main}. In Section~\ref{subsec-other} we prove Theorem~\ref{thm-main-pt}, and deduce the results stated in the introduction from Theorems~\ref{thm-main} and~\ref{thm-main-pt}.

\subsection{An $n=1$ special case via mating-of-trees}\label{subsec-n=1}

In Definition~\ref{def-eps-local} with $n=1$, we independently sample an LQG surface from $\mathrm{UQD}_{\ccL}$ and an IG field from $\mathrm{IG}_{26 - \ccL}$, discretize this IG-decorated LQG surface in a Poissonian way, and discover a random decorated LQG sub-surface via a peeling process. 
Proposition~\ref{prop-eps-indep} below states that given this explored LQG sub-surface, its complementary LQG surfaces are conditionally independent, where ``complementary LQG surface'' is appropriately defined. See Figure~\ref{fig-thm-indep} (right).

\begin{prop}\label{prop-eps-indep}
	Consider the setting of Definition~\ref{def-eps-local} with $n=1$ (so $\eta'$ is the counterclockwise space-filling SLE curve coupled with $\Psi_1$). Let $L_1, \dots, L_m$ be simple loops in the interior of $\bigcup_{i \in I_\tau} D_i$ and for $k=1,\dots,m$ let $O_k$ be the bounded connected component of $\C \backslash L_k$. Suppose $(\bigcup_{i \in I_\tau} D_i, \Phi, L_1, \dots, L_m)/{\sim}$ is measurable with respect to $\cM_\tau:=( \bigcup_{i \in I_\tau} D_i, \Phi, \Psi_1, \{ \eta_i', a_i, b_i, c_i, d_i) : i \in I_{\tau}\}, \Lambda)/{\sim}$, and the $O_k$ are pairwise disjoint. Then conditioned on $\cM_\tau$, the IG-decorated LQG surfaces $(O_1, \Phi, \Psi)/{\sim}, \dots, (O_m, \Phi, \Psi)/{\sim}$ are conditionally independent.
\end{prop}

\begin{remark}
	For $\gamma \in (0,\sqrt2]$, the proposition still holds if we simply define the $O_i$ to be the complementary connected components of $\D\backslash M_\tau$. 
	For $\gamma \in (\sqrt2, 2)$, this simpler formulation runs into topological issues because the interior of each $D_i$ has multiple components, hence $D_i$ might intersect multiple connected components of $\D \backslash M_\tau$. While a suitable modification of this formulation is true, we prefer to avoid this issue altogether by using the loops $L_i$.
\end{remark}

We defer the proof of Proposition~\ref{prop-eps-indep} to Section~\ref{sec-n=1}. The starting point is the mating-of-trees theorem for the LQG disk \cite{wedges, ag-disk}, which describes the independent coupling of $\mathrm{UQD}_{\ccL}$ and $\mathrm{IG}_{26 - \ccL}$ in terms of a pair of correlated one-dimensional Brownian motions. See Section~\ref{subsec-mot}. A statement like Proposition~\ref{prop-eps-indep} should then seem plausible due to the independent increments of Brownian motion; see Proposition~\ref{prop-decomp-quilt-cell} for a factorized description of the LQG surfaces $(D_i, \Phi, \eta_i', a_i, b_i, c_i, d_i)/{\sim}$ coming from Brownian motion.

 On the other hand, consider the planar map whose faces correspond to the $D_i$ from Definition~\ref{def-eps-local}; its precise definition is given in  Section~\ref{subsec-quilt-def} (see Figure~\ref{fig-template} (left)). This random planar map has the global condition that the faces must ``chain together'' to form a \emph{single} path, corresponding to the order they are traced by the space-filling SLE loop. Further, adjacent faces (not necessarily consecutive) must intersect in a way which is compatible with their relative ordering in this path. See Definition~\ref{def-T} for details. As such,  one could a priori expect complicated dependencies between disjoint regions of the map.  
To deal with this difficulty, we show that conditioned on the random submap corresponding to $\cM_\tau$ in Proposition~\ref{prop-eps-indep}, the set of possible complementary submaps factorizes as a product  (Proposition~\ref{prop-product-templates}), i.e.\ the possible realizations of each submap do not depend on the other complementary submaps. We prove this in Section~\ref{subsec-topo} via a  topological/combinatorial argument involving \emph{total curvature}, a.k.a.\ \emph{winding} (Section~\ref{subsec-winding}). In the context of imaginary geometry, winding is natural since it is related to the boundary conditions of flow and counterflow lines. We prove that the global conditions described above are equivalent to local conditions related to winding, and hence obtain the combinatorial factorization Proposition~\ref{prop-product-templates}. 

Proposition~\ref{prop-eps-indep} is proved by combining the probabilistic input of Proposition~\ref{prop-decomp-quilt-cell} with the combinatorial input of Proposition~\ref{prop-product-templates}.

A simpler version of the argument of Section~\ref{subsec-topo} leads to a conditional independence statement for the ``upper" and ``lower" halves of a uniform meander given its associated winding function. We record this result and its proof in Appendix~\ref{sec-meander}, both to motivate the argument of Section~\ref{subsec-topo} and because the result is of independent interest.

\subsection{Extending to $(\mathbf c_1, \dots, \mathbf c_n) = (26 - \ccL, 0, \dots, 0)$ via locality of $\mathbf c=0$ IG fields}\label{subsec-locality}

In this section we build on Proposition~\ref{prop-eps-indep} to prove Proposition~\ref{prop-main-0s}, a special case of Theorem~\ref{thm-main} where all but one of the IG fields has central charge 0 and  we approximate $K$ by the union of the space-filling SLE segments $D_i$ which it intersects rather than by $\cB_\eps(K; {\mathfrak d_\Phi})$. We first state the more general  Proposition~\ref{prop-main-Keps}  whose proof we defer to Section~\ref{subsec-main-proof}.

\begin{proposition}\label{prop-main-Keps}
	Let $\eps > 0$. In the setting of Definition~\ref{def-eps-local}, let $K$ be an $\eps$-LQG-IG local set coupled with $(\Phi, \Psi_1, \dots, \Psi_n, \eta')$. Sample a Poisson point process $\Lambda \subset \D$ with intensity measure $\eps^{-1} \mu_\Phi$ independently of $(\Psi_1, \dots, \Psi_n, \eta', K)$ and define $D_i, \eta_i', a_i, b_i, c_i, d_i$ as in Definition~\ref{def-eps-local}. 
	Then the decorated LQG surfaces $(U, \Phi, \Psi_1, \dots, \Psi_{n})/{\sim}$ parametrized by the connected components $U$ of $\D \backslash K$ are conditionally independent  given $\cF_\eps^\mathrm{MOT}$, where
	\begin{equation} \label{eqn-local-set-sigma}
		\cF_\eps^\mathrm{MOT} := \sigma\left( (K_\eps,\Phi, \Psi_1, \dots, \Psi_{n}, \{ (\eta'_i, a_i, b_i, c_i, d_i)\::\: i \in I\}, K)/{\sim} \right) ,
	\end{equation}
 $I = \{ i \:: \: D_i \cap K \neq \emptyset\}$, and $K_\eps = \bigcup_{i \in I} D_i$.
\end{proposition}

In this section we only prove the following special case of Proposition~\ref{prop-main-Keps}. 
The general case is proven in Section~\ref{subsec-main-proof} using Proposition~\ref{prop-ig-rotate}.

\begin{proposition}\label{prop-main-0s}
	Proposition~\ref{prop-main-Keps} holds when $(\mathbf c_1, \mathbf c_2, \dots, \mathbf c_n) = (26 - \ccL, 0, \dots, 0)$ and $\eta'$ is the counterclockwise space-filling SLE coupled with $\Psi_1$. 
\end{proposition}

Proposition~\ref{prop-main-0s} follows from Proposition~\ref{prop-eps-indep}  and a \emph{locality} property for $\mathbf c=0$ IG fields. In Section~\ref{subsec-loc-planar} we state the locality property for planar domains due to Dub\'edat \cite{dubedat-coupling}. In Section~\ref{subsec-loc-ig} we extend locality to the setting of IG surfaces. We use this locality property to prove Proposition~\ref{prop-main-0s} in Section~\ref{subsec-main-ig0-proof}. 

\subsubsection{Locality in planar domains}\label{subsec-loc-planar}
In this section we state Theorem~\ref{thm-loc-dom} on the locality of $\mathrm{IG}_{\mathbf c= 0}$ in planar domains, due to Dub\'edat \cite{dubedat-coupling}. For the reader's convenience we include the argument in Appendix~\ref{appendix-locality}.

For $i,j \in \{1,2\}$ let $D_{ij}$ be a bounded  simply-connected domain with smooth boundary. Suppose that for the $D_{ij}$, there exists a pair of common open boundary segments and a pair of common smooth disjoint cross-cuts $\eta_\ell, \eta_r$ joining them; call the region they bound $B\subset D_{ij}$, and let $x \in \partial B \cap D_{ij}$ be a common boundary point. We call the connected component of $D_{ij} \backslash B$ immediately clockwise of $x$ the \emph{left subdomain} of $D_{ij}$, and similarly define the \emph{right subdomain}.  For $i,j \in \{1,2\}$ we suppose the left subdomains of $D_{i1}$ and $D_{i2}$ agree and the right subdomains of $D_{1j}$ and $D_{2j}$ agree. See Figure~\ref{fig-Dij}.

\begin{figure}[ht!]
	\begin{center}
		\includegraphics[scale=0.4]{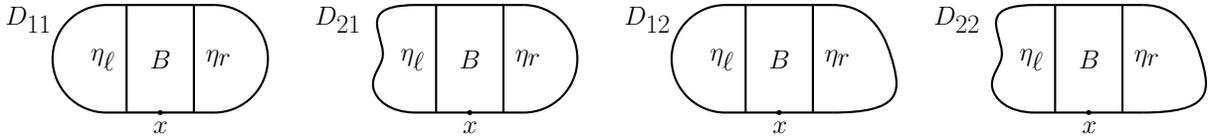}%
		\caption{\label{fig-Dij} Setup for Theorem~\ref{thm-loc-dom}.}
	\end{center}
\end{figure}

For an IG surface $(D_{ij}, \Psi, x)/{\sim_{\cc}}$ sampled from $\mathrm{IG}_\cc$ on the one pointed domain (Definition~\ref{def-ig-disk}), let $\mathrm{IG}_\cc^{B, ij}$ denote the law of the decorated IG surface $(B, \Psi, \eta_\ell, \eta_r,  x)/{\sim_\cc}$.  

\begin{lemma}
The measures $\mathrm{IG}^{B,ij}_{\mathbf c}$ are mutually absolutely continuous for all $i, j \in \{1,2\}$.	
\end{lemma}
\begin{proof}
	Let $\theta_0 \in [0,2\pi)$  be the angle such that the counterclockwise tangent to $B$ at $x$ is parallel to $e^{\bbi \theta_0}$, and let $\mathfrak h_{ij}$ be the harmonic function on $D_{ij}$ whose boundary value at $y \neq x$ equals the total curvature of the counterclockwise boundary arc from $x$ to $y$. By Lemma~\ref{lem-ig-bdy-cond}, if $h_{ij}$ is a zero boundary GFF on $D_{ij}$, then the law of $(D_{ij}, h_{ij} + \chi(\cc) (\mathfrak h_{ij} + \theta_0), x)/{\sim_\cc}$ is $\mathrm{IG}_\cc$. By Lemma~\ref{lem-total-winding}, the values of $\mathfrak h_{ij}$ on $\partial B \cap \partial D_{ij}$ (and hence on a neighborhood of $\partial B \cap \partial D_{ij}$ in $\partial D_{ij}$) agree for all $i,j$. Therefore the laws of $(h_{ij} + \chi(\cc) (\mathfrak h_{ij} + \theta_0))|_{B}$ are mutually absolutely continuous for all $i,j$~\cite[Proposition 3.4]{ig1}. 
\end{proof}

\begin{thm}[{\cite{dubedat-coupling}}]\label{thm-loc-dom}
	\eqb\label{eq-locality-planar}
	\frac{d\mathrm{IG}_{\mathbf c=0}^{B,11}}{d\mathrm{IG}_{\mathbf c=0}^{B,21}} (\cS)   = \frac{d\mathrm{IG}_{\mathbf c=0}^{B,12}}{ d\mathrm{IG}_{\mathbf c=0}^{B,22}} (\cS)   \qquad\text{ for } \mathrm{IG}^{B,ij}_{\mathbf c=0}\text{-a.e. IG surface }\cS.
	\eqe
\end{thm}
The left hand side of~\eqref{eq-locality-planar} describes the relative likelihood of observing $\cS$  when we vary the left subdomain. 
Equation~\eqref{eq-locality-planar} states that this relative likelihood does not depend on the right subdomain. In this way Theorem~\ref{thm-loc-dom} is a locality result for $\mathrm{IG}_{\mathbf c =0}$.

\subsubsection{Locality for $\ul \cc = (0, \dots, 0)$ IG surfaces}\label{subsec-loc-ig}

In this section we adapt Theorem~\ref{thm-loc-dom} to the setting of IG surfaces. 

We work in the domain $(\D, x)$ where $x = -\mathbbm i$.  See Figure~\ref{fig-Dij-16} (left).
For $i,j \in \{1,2\}$, consider two disjoint cross-cuts $\eta_\ell^{ij}, \eta_r^{ij}$ which do not contain $x$. Suppose the connected component $B^{ij}$ of $\D \backslash (\eta_\ell^{ij} \cup  \eta_r^{ij})$ having $x$ on its boundary also has $\eta_\ell^{ij}$ and $\eta_r^{ij}$ on its boundary, and $x, \eta_\ell^{ij}, \eta_r^{ij}$ lie on $\partial B^{ij}$ in clockwise order.   Let $L^{ij}$ be the connected component of $\D\backslash B^{ij}$ having $ \eta_\ell^{ij}$ on its boundary, let $R^{ij}$ be the other connected component, and let $V^{ij}$ be a simply connected subdomain of $\D$ such that $B^{ij} \subset V^{ij}$ and $\partial V^{ij} \setminus \partial\D$ lies at positive distance from $B^{ij}$.

Suppose for each $i = 1,2$ the curve-decorated domains $(L^{ij} \cup V^{ij}, \eta_\ell^{ij}, \eta_r^{ij})$ for $j=1,2$ are conformally equivalent, and suppose for each $j = 1,2$ that $(V^{ij} \cup R^{ij}, \eta_\ell^{ij}, \eta_r^{ij})$ for $i=1,2$ are conformally equivalent. One should think of this as saying that \emph{modulo conformal equivalence}, the four curve-decorated domains arise from choosing among two possible left subdomains and two possible right subdomains. In particular, unlike in Theorem~\ref{thm-loc-dom}, we do not assume $L^{1j} = L^{2j}$. 

For $(\Psi_1, \dots, \Psi_n)$ sampled from $\mathrm{IG}_{\ul \cc}$ on $(\D, x)$, let $\mathrm{IG}_{\ul \cc}^{ij}$ denote the law of the curve-decorated $\ul \cc$-IG surface $(B^{ij}, \Psi_1, \dots, \Psi_n, \eta_\ell^{ij}, \eta_r^{ij})/{\sim}$.
\begin{figure}[ht!]
	\begin{center}
		\includegraphics[scale=0.5]{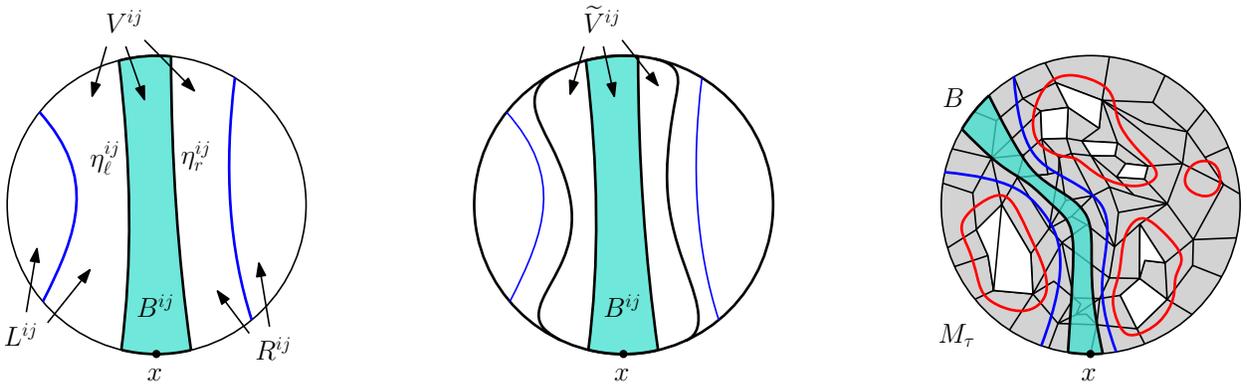}%
		\caption{\label{fig-Dij-16} \textbf{Left.}  Proposition~\ref{prop-locality-qs}  gives an identity for Radon-Nikodym derivatives of the turquoise $\cc$-IG surface when  $i$ and $j$ take values in $\{ 1,2\}$. \textbf{Middle.} In the proof of Proposition~\ref{prop-locality-qs}, we define a domain $\wt V^{ij}$ with smooth boundary. Four possible domains (boundaries shown in bold) can be obtained by gluing to $B^{ij}$ a ``left piece'' and ``right piece'' chosen from the connected components of $\wt V^{ij} \backslash B^{ij}$ and $\D \backslash B^{ij}$. We apply Theorem~\ref{thm-loc-dom} to these four domains.  \textbf{Right.} Setup for the proof of Proposition~\ref{prop-local-indep}. 
			The shaded region is $M_\tau = \bigcup_{i \in I_\tau} D_i \subset \D$, and $L_1, \dots, L_m$ are the red loops. We choose a region $B$ (turquoise) bounded by a pair of cross-cuts separating  $L_1$ from the other loops such that $x \in \partial B$. Let $L$ and $R$ be the two connected components of $\D\backslash B$ (each colored gray and white), and let $V \subset M_\tau$ be a simply connected neighborhood of $B$ (between blue curves).}
	\end{center}
\end{figure}

\begin{proposition}\label{prop-locality-qs}
	Let $\ul \cc$ be a vector with all entries equal to 0. 
	The four laws $\mathrm{IG}_{\ul \cc}^{ij}$ are mutually absolutely continuous, and 
	\[ \frac{d\mathrm{IG}_{\ul \cc}^{11}}{d\mathrm{IG}_{\ul \cc}^{21}}(\cS)  = \frac{d\mathrm{IG}_{\ul \cc}^{12}}{d\mathrm{IG}_{\ul \cc}^{22}}(\cS) \qquad \text{ for }\mathrm{IG}_{\ul \cc}^{ij}\text{-a.e.\ IG surface }\cS.\]
\end{proposition}
\begin{proof}
	From our conformal equivalence assumptions, it is immediate that $(V^{ij}, \eta_\ell^{ij}, \eta_r^{ij})$ for $i,j \in \{1,2\}$ are all conformally equivalent. 
	Pick  a simply connected domain $\wt V^{11}$ with smooth boundary such that $B^{11} \subset \wt V^{11} \subset V^{11}$, and $B^{11}$ (resp.\ $\wt V^{11}$) lies at positive distance from $\partial \wt V^{11} \backslash \partial \D$ (resp.\ $\partial V^{11} \backslash \partial \D$). Define $\wt V^{ij}$ for $i, j \in \{1,2\}$ using the above conformal equivalence. See Figure~\ref{fig-Dij-16}  (middle). 
	
	For this paragraph we fix the choice $i = j = 1$. Let $L_0$ (resp.\  $R_0$) be the connected component of $\wt V^{11} \backslash B^{11}$ immediately clockwise (resp.\ counterclockwise) of $x$, and write $L_1 = L^{11}$ and $R_1 = R^{11}$. For $i', j' \in \{ 0,1\}$ let $D_{i'j'} = \ol{L_{i'} \cup B^{11} \cup R_{j'}}$, see Figure~\ref{fig-Dij-16}  (middle). Let $n$  be the number of entries of  $\ul \cc$. 
	Recall $\mathrm{IG}_{\ul \cc}^{11}$ is the law of the ${\ul \cc}$-IG surface $(B^{11},  \Psi_1|_{B^{11}}, \dots, \Psi_n|_{B^{11}},  \eta^{11}_\ell, \eta^{11}_r)/{\sim}$ where $\Psi_1, \dots, \Psi_n$ are independently sampled from $\mathrm{IG}_{{\cc=0}}$ on $(D_{11}, x)$ with additive constant fixed such that the boundary value infinitesimally counterclockwise of $x = -\bbi$ in $\partial \D$ is $0$. Similarly, we define $\mathrm{IG}_{\ul \cc}^{00}$, ${ \mathrm{IG}}_{\ul \cc}^{10}$ and $\mathrm{IG}_{\ul\cc}^{01}$ by replacing $D_{11}$ with $D_{00}, D_{10}$ and $D_{01}$ respectively. 
	This give four measures, which are mutually absolutely continuous by the  sentence above  Theorem~\ref{thm-loc-dom}. Applying  Theorem~\ref{thm-loc-dom} $n$ times gives
	\[ \frac{d{ \mathrm{IG}}_{{\ul\cc}}^{00}}{d{ \mathrm{IG}}_{{\ul\cc}}^{10}} 
	\frac{d{ \mathrm{IG}}_{{\ul\cc}}^{11}}
	{d{ \mathrm{IG}}_{{\ul\cc}}^{01}}  =1 \qquad \text{ for } { \mathrm{IG}}^{00}_{{\ul\cc}}\text{-a.e. IG surface}.\]
	
	For each $i,j \in \{1,2\}$ we repeat the preceding paragraph to define three measures ${ \mathrm{IG}}_{\ul\cc}^{00}, { \mathrm{IG}}_{{\ul\cc}}^{i0}, { \mathrm{IG}}_{{\ul\cc}}^{0j}$. Some measures are multiply defined (for instance ${\mathrm{IG}}_{{\ul\cc}}^{10}$ is defined twice, for $i=1, j=1,2$), but these definitions are consistent by conformal equivalence of the domains. In total this gives nine mutually absolutely continuous measures satisfying
	\eqb\label{eq-RN-i0}
	\frac{d{ \mathrm{IG}}_{{\ul\cc}}^{00}}{d{ \mathrm{IG}}_{{\ul\cc}}^{i0}} 
	\frac{d{ \mathrm{IG}}_{{\ul\cc}}^{ij}}
	{d{ \mathrm{IG}}_{{\ul\cc}}^{0j}}  =1  \qquad \text{ for } { \mathrm{IG}}^{00}_{{\ul\cc}}\text{-a.e. IG surface,  for }i,j \in \{1,2\}.
	\eqe
	Multiplying the left hand sides of~\eqref{eq-RN-i0} for $(i,j) \in \{ (1,1),(2,2)\}$, and dividing by the left hand sides for $(i,j) \in \{(1,2),(2,1)\}$, we obtain $\frac{d{ \mathrm{IG}}_{{\ul\cc}}^{11}}{d{ \mathrm{IG}}_{{\ul\cc}}^{21}} 
	\frac{d{ \mathrm{IG}}_{{\ul\cc}}^{22}}
	{d{ \mathrm{IG}}_{{\ul\cc}}^{12}}  =1$ for ${\mathrm{IG}}_{{\ul\cc}}^{ij}$-a.e.\ IG surface, as desired. 
\end{proof}

\subsubsection{LQG-IG local sets of $(\mathbf c_1, \dots, \mathbf c_n) = (26 - \ccL, 0, \dots, 0)$  IG-decorated LQG surfaces}\label{subsec-main-ig0-proof}

In this section we prove Proposition~\ref{prop-main-0s}.
We work in the setting of Proposition~\ref{prop-main-0s}, so  $(\mathbf c_1, \mathbf c_2, \dots, \mathbf c_n) = (26 - \ccL, 0, \dots, 0)$,  $K$ is an $\eps$-LQG-IG local set coupled with $(\Phi, \Psi_1, \dots, \Psi_n, \eta')$, and $\Lambda \subset \D$ is a Poisson point process with intensity measure $\eps^{-1} \mu_\Phi$ independent of $(\Psi_1, \dots, \Psi_n, \eta', K)$. As before the curve $\eta'$ is segmented by the points $\Lambda$ into curve segments $\eta_0', \dots, \eta_{|\Lambda|}'$, and each $\eta_i'$ traces out a domain $D_i$ with boundary points $a_i, b_i, c_i, d_i$. 

Recall that in Definition~\ref{def-eps-local}, we sample a random set of indices $I_\tau$ conditionally independent of $K$ given $(\Phi,  \Psi_1, \dots, \Psi_n, \eta', \Lambda)$. We first show an $n \geq 1$ generalization of Proposition~\ref{prop-eps-indep} via locality of $\cc=0$ IG fields. 

\begin{proposition}\label{prop-local-indep}
	Assume $\ccL > 25$ and $(\mathbf c_1, \mathbf c_2, \dots, \mathbf c_n) = (26 - \ccL, 0, \dots, 0)$. 
	In the setting of Definition~\ref{def-eps-local}, let $L_1, \dots, L_m$ be disjoint loops in $M_\tau$ which bound disjoint regions $O_1, \dots, O_m$ and let
	$\wh \cM_\tau  = (M_\tau, \Phi, \Psi_1, \dots, \Psi_n, \{ \eta_i', a_i, b_i, c_i, d_i\: : \: i \in I_\tau\})/{\sim}$. Suppose $(M_\tau, \Phi, L_1, \dots, L_m)/{\sim}$ is measurable with respect to $\sigma(\wh \cM_\tau)$, i.e., the loops on $\wh \cM_\tau$ are chosen in a way depending only on $\wh \cM_\tau$.  Then almost surely, given $\wh \cM_\tau$  the IG-decorated LQG surfaces $(O_i, \Phi, \Psi_1,  \dots, \Psi_n)/{\sim}$ for $i = 1, \dots, m$ are conditionally independent.
\end{proposition}
\begin{proof}

We begin with a proof outline.  Let $\cM_\tau := (M_\tau, \Phi, \Psi_1,  \{(\eta_i', a_i, b_i, c_i, d_i): i \in I_{\tau} \})/{\sim}$ be $\wh \cM_\tau$ with the $\cc=0$ fields forgotten. We split $\D$ into three regions $L, B, R$ with $B \subset M_\tau$ such that $(M_\tau, \Phi, B)/{\sim}$ is measurable with respect to $\cM_\tau$ (see Figure~\ref{fig-Dij-16}, right).  By our $n=1$ result (Proposition~\ref{prop-eps-indep}) the IG-decorated LQG surfaces $(L \cup B, \Phi, \Psi_1)/{\sim}$ and $(R\cup B, \Phi, \Psi_1)/{\sim}$ are conditionally independent given $\cM_\tau$. Using the $\cc=0$ IG locality property (Proposition~\ref{prop-locality-qs}), we can show this conditional independence still holds when we  instead condition on $\cM_\tau$ further decorated by $\Psi_2|_B, \dots, \Psi_n|_B$. The Markov property of the GFF then implies the conditional independence of $(L\cup B, \Phi, \Psi_1,  \Psi_2, \dots, \Psi_n)/{\sim}$ and $(R \cup B, \Phi, \Psi_1, \Psi_2, \dots, \Psi_n)/{\sim}$ under this conditioning; an easy argument then shows this still holds when conditioning on $\wh \cM_\tau$ instead. Varying our choice of $B$ gives the desired result.

We now turn to the proof proper. 
	Draw a pair of cross-cuts $\eta_\ell, \eta_r$ in $\D$ which stay in $M_\tau$ and do not hit $x := -\bbi$, such that the connected component $B$ of $\D \backslash (\eta_\ell \cup  \eta_r)$ having $x$ on its boundary also has $\eta_\ell$ and $\eta_r$ on its boundary, and $x, \eta_\ell, \eta_r$ lie on $\partial B$ in clockwise order. Let $L$ be the connected component of $\D \backslash B$ having $\eta_\ell$ on its boundary, and let $R$ be the other connected component. Choose a simply connected domain $V$ such that $B \subset V \subset M_\tau$, and $B$ (resp.\ $V$) lies at positive distance from $\partial V \backslash \partial \D$ (resp.\ $\partial M_\tau \backslash \partial \D$). We do this in such a way that $(M_\tau, \Phi, B, V)/{\sim}$ is measurable with respect to $\cM_\tau$, and $O_1 \subset L$ while $O_2, \dots, O_m \subset R$.  See Figure~\ref{fig-Dij-16} (right).

		Let  $\cS_\ell = (L \cup B, \Phi, \Psi_1)/{\sim}$ and $\cS_r = (R \cup B, \Phi, \Psi_1)/{\sim}$. 
		Note that $\cS_\ell$ is measurable with respect to $\cM_\tau$ and $(O_1, \Phi, \Psi_1)/{\sim}$, and $\cS_r$ is measurable with respect to $\cM_\tau$ and $(O_i, \Phi, \Psi_1)/{\sim}$ for $i = 2, \dots, m$. Thus, by Proposition~\ref{prop-eps-indep}, conditioned on $\cM_\tau$, the decorated LQG surfaces $\cS_\ell, \cS_r$  are conditionally independent. Let 
		\eqbn
		\cM_\tau^* = (M_\tau, \Phi, \Psi_1,  \{(\eta_i', a_i, b_i, c_i, d_i): i \in I_{\tau} \}, \Psi_2|_B, \dots, \Psi_n|_B)/{\sim} .
		\eqen
		 Using the locality of $\mathbf c=0$ IG fields, we will prove that conditioned on $\cM_\tau^*$ the decorated LQG surfaces $\cS_\ell, \cS_r$ are still conditionally independent.
		 
		 Let $P_{\cM_\tau}$ be the conditional law of of $\cM_\tau^*$ given $\cM_\tau$. Similarly let  $P_{\cS_\ell, \cS_r}$ be the conditional law of $\cM_\tau^*$ given $(\cM_\tau, \cS_\ell, \cS_r)$. Let $\cL_{\cM_\tau}$ (resp.\ $\cL_{\cM_\tau^*}$) be the conditional law of $(\cS_\ell, \cS_r)$ given $\cM_\tau$ (resp.\ $\cM_\tau^*$). By the aforementioned {conditional} independence we have $\cL_{\cM_\tau}(d\cS_\ell, d\cS_r) = \cL^\ell_{\cM_\tau}(d\cS_\ell) \cL^r_{\cM_\tau}(d\cS_r)$ for some probability measures $\cL^\ell_{\cM_\tau}, \cL^r_{\cM_\tau}$ which are measurable functions of $\cM_\tau$. We will need the following lemma.

		 \begin{lemma}\label{lem-rd-M}
		 	Condition on $\cM_\tau$ and fix $\cS_\ell^1, \cS_\ell^2, \cS_r^1, \cS_r^2$ (where the $\cS_\ell^i$ and $\cS_r^j$ denote possible realizations of $\cS_\ell$ and $\cS_r$). Then the probability measures $P_{\cS_\ell^i, \cS_r^j}$ are mutually absolutely continuous for $i, j \in \{1,2\}$. Furthermore, for a.e.\ $\cM_\tau^*$ we have $\frac{dP_{\cS_\ell^1, \cS_r^1}}{dP_{\cS_\ell^2, \cS_r^1}} (\cM_\tau^*) = \frac{dP_{\cS_\ell^1, \cS_r^2}}{dP_{\cS_\ell^2, \cS_r^2}} (\cM_\tau^*)$. 
		 \end{lemma}
		 \begin{proof}
		 	Throughout this proof we condition on $\cM_\tau$.
		 	For each $i, j \in \{1, 2\}$, we have a copy of the setup described above, where we have fields $\Phi^{ij}, \Psi_1^{ij}, \dots, \Psi_n^{ij}$, a Poisson point process $\Lambda^{ij}$, and domains $B^{ij} \subset V^{ij} \subset M^{ij}$. Fixing $\cS_\ell^i$ and $\cS_r^j$ corresponds to conditioning on all of the above (modulo conformal coordinate change) except $\Psi_2^{ij}, \dots, \Psi_n^{ij}$. Proposition~\ref{prop-locality-qs} gives mutual absolute continuity of the conditional laws of the IG surface $\cB = (B^{ij}, \Psi_2, \dots, \Psi_n, \eta^{ij}_\ell, \eta^{ij}_r)/{\sim}$ for $i, j \in \{1,2\}$, and further gives 
		 	an identity for  Radon-Nikodym derivatives of these conditional laws.  Since $\cM_\tau^*$ determines and is determined by $(\cB, \cM_\tau)$, we obtain the desired Lemma~\ref{lem-rd-M}.
		 \end{proof}
		 
Condition on $\cM_\tau$. 		 
	By definition we have 
		 \eqb\label{eq-marginal-PM}
		 P_{\cM_\tau} = \iint P_{\cS_\ell', \cS_r'} \cL_{\cM_\tau}^\ell(d\cS_\ell') \cL_{\cM_\tau}^r(d\cS_r')
		 \eqe
		 where $\cS_\ell', \cS_r'$ denote possible realizations of $\cS_\ell$ and $\cS_r$. The mutual absolute continuity claim of Lemma~\ref{lem-rd-M} applied to  \eqref{eq-marginal-PM} implies that a.s. $P_{\cS_\ell, \cS_r}$ and $P_{\cM_\tau}$ are mutually absolutely continuous.
		  We can thus take  Radon-Nikodym derivatives in~\eqref{eq-marginal-PM} to get 
 \eqb\label{eq-pre-fac}
	\frac{dP_{\cM_\tau}}{dP_{\cS_\ell, \cS_r}}(\cM_\tau^*) 
	= \iint \frac {dP_{\cS_\ell', \cS_r'}}{dP_{\cS_\ell, \cS_r'}}(\cM_\tau^*)
	\frac {dP_{\cS_\ell, \cS_r'}}{dP_{\cS_\ell, \cS_r}}(\cM_\tau^*) \cL_{\cM_\tau}^\ell(d\cS_\ell')\cL_{\cM_\tau}^r(d\cS_r').
\eqe
Lemma~\ref{lem-rd-M} implies that  $\frac {dP_{\cS_\ell', \cS_r'}}{dP_{\cS_\ell, \cS_r'}}(\cM_\tau^*)$ does not depend on $\cS_r'$, i.e., it equals $f_\ell(\cS_\ell, \cS_\ell', \cM_\tau^*)$ for some function $f_\ell$. Likewise $\frac {dP_{\cS_\ell, \cS_r'}}{dP_{\cS_\ell, \cS_r}}(\cM_\tau^*)=f_r(\cS_r, \cS_r', \cM_\tau^*)$ for some function $f_r$. Thus~\eqref{eq-pre-fac} gives 
	\[ 	\frac{dP_{\cM_\tau}}{dP_{\cS_\ell, \cS_r}}(\cM_\tau^*)  = \int f_\ell(\cS_\ell, \cS_\ell', \cM_\tau^*) \cL^\ell_{\cM_\tau}(d\cS_\ell')  \int f_r(\cS_r, \cS_r', \cM_\tau^*) \cL^r_{\cM_\tau}(d\cS_r'),\]
	so $\frac{dP_{\cS_\ell, \cS_r}}{dP_{\cM_\tau}} = F_\ell(\cS_\ell, \cM_\tau^*) F_r(\cS_r, \cM_\tau^*)$ for some functions $F_\ell, F_r$. 
	
	Recall that $\cL_{\cM_\tau^*}$ denotes the conditional law of $(\cS_\ell,\cS_r)$ given $\cM_\tau^*$. By considering two ways of sampling from the joint law of $(\cM_\tau^*, \cS_\ell, \cS_r)$ given $\cM_\tau$, we have 
	\begin{align*}
	\cL_{\cM_\tau^*}(d\cS_\ell, d\cS_r) P_{\cM_\tau}(d\cM_\tau^*) &= P_{\cS_\ell, \cS_r}(d\cM_\tau^*) \cL_{\cM_\tau}^\ell (d\cS_\ell) \cL_{\cM_\tau}^r (d\cS_r) \\
	&= F_\ell(\cS_\ell, \cM_\tau^*) \cL_{\cM_\tau}^\ell(d\cS_\ell) F_r(\cS_r, \cM_\tau^*) \cL_{\cM_\tau}^r(d\cS_r) P_{\cM_\tau}(d\cM_\tau^*). 
	\end{align*}
Thus, for a.e.\ $\cM_\tau^*$ the measure $\cL_{\cM_\tau^*}$ factorizes as $F_\ell(\cS_\ell, \cM_\tau^*) \cL_{\cM_\tau}^\ell(d\cS_\ell) F_r(\cS_r, \cM_\tau^*) \cL_{\cM_\tau}^r(d\cS_r)$, or in other words, a.s.\ conditioned on $\cM_\tau^*$ the decorated LQG surfaces $\cS_\ell$ and $\cS_r$ are conditionally independent. 
	
	Now, let $\wh \cS_\ell = (L\cup B, \Phi, \Psi_1, \eta_\ell, \Psi_2, \dots, \Psi_n)/{\sim}$ be $\cS_\ell$ decorated by the additional $\cc =0$ IG fields, and likewise define $\wh \cS_r$. By the above conditional independence and the domain Markov property of the GFF, $\wh \cS_\ell$ and $\wh \cS_r$ are conditionally independent given $\cM_\tau^*$. Next, let $\cM_\tau^{*, \ell}$ be $\cM_\tau^*$ decorated by the fields $\Psi_i|_{(L\cup B) \cap M_\tau}$ rather than $\Psi_i |_{B}$ for $i= 2, \dots, n$. It is measurable with respect to  $\cM_\tau^*$ and $\wh \cS_\ell$, so $\wh \cS_\ell, \wh \cS_r$ are conditionally independent given $\cM_\tau^{*, \ell}$. Finally, $\wh \cM_\tau$ is $\cM_\tau^{*, \ell}$ decorated by the fields $\Psi_i|_{M_\tau} $ rather than $\Psi_i |_{(L \cup B) \cap M_\tau}$ for $i= 2, \dots, n$, hence is measurable with respect to $\cM_\tau^{*, \ell}$ and $\wh \cS_r$. We conclude $\wh \cS_\ell, \wh \cS_r$ are conditionally independent given $\wh \cM_\tau$.
	
	Thus, given  $\wh \cM_\tau$,  the decorated LQG surface $(O_1,\Phi, \Psi_1, \dots, \Psi_n)/{\sim}$ is conditionally independent of $(O_2, \Phi, \Psi_1, \dots, \Psi_n)/{\sim}, \dots, (O_m, \Phi, \Psi_1, \dots, \Psi_n)/{\sim}$. Repeating the argument for the other $m-1$ indices,  we obtain the mutual independence of the $(O_i, \Phi, \Psi_1, \dots, \Psi_n)/{\sim}$ given $\wh \cM_\tau$. 
\end{proof}

\begin{proof}[{Proof of Proposition~\ref{prop-main-0s}}]
	 Let $U_1, \dots, U_m$ be a finite collection of connected components of $\D \backslash K$ chosen such that $(K_\eps, \Phi, \partial U_1, \dots, \partial U_m)/{\sim}$ is measurable with respect to $\cF_\eps^\mathrm{MOT}$, and let $L_1, \dots, L_m$ be disjoint simple loops in $K_\eps$ such that $L_i$ is homotopic\footnote{More precisely, if $f: U_i \to \D$ is a conformal map then $f(L_i)$ is homotopic to $\partial \D$ in $f(U_i \cap K_\eps)$.} to $\partial U_i$ in $U_i \cap K_\eps$, and $(K_\eps, \Phi,  L_1, \dots, L_m)/{\sim}$ is measurable with respect to $\cF_\eps^\mathrm{MOT}$. Note that even if $\ol U_i \cap \ol U_j \neq \emptyset$ for some $i\neq j$ (e.g.,\ when $K$ has ``pinch points''), by construction $K_\eps$ contains a neighborhood of $K$, so the loops $L_1, \dots, L_m$ can always be chosen
	
	Let $\wh \cM_\tau$ be defined as in Proposition~\ref{prop-local-indep}. Let $\wh \cM = (K_\eps, \Phi, \Psi_1, \dots, \Psi_n, \{\eta_i', a_i, b_i, c_i, d_i \: : \: i \in I\} )/{\sim}$, so that $\cF_\eps^\mathrm{MOT} = \sigma(\wh \cM, (K_\eps, \Phi, K)/{\sim})$. By Definition~\ref{def-eps-local}, conditioned on $\wh \cM_\tau$, the event $\{\wh \cM_\tau = \wh \cM\}$ is conditionally independent of $(\Phi, \Psi_1, \dots, \Psi_n, \eta', \Lambda)$, and further conditioned on $\{\wh \cM_\tau = \wh \cM\}$, the decorated LQG surface $(K_\eps, \Phi, K)/{\sim}$ (and hence $\cF_\eps^\mathrm{MOT}$) is conditionally independent of $(\Phi, \Psi_1, \dots, \Psi_n, \eta', \Lambda)$. Our loops are chosen in a way which is determined by $\cF_\eps^\mathrm{MOT}$, so by Proposition~\ref{prop-local-indep}, conditioned on $\wh \cM_\tau$, $\{ \wh \cM_\tau = \wh \cM\}$ and $\cF_\eps^\mathrm{MOT}$, the decorated LQG surfaces bounded by the $L_i$ are conditionally independent. Consequently, with the same conditioning, the  $(U_i, \Phi,  \Psi_1, \dots, \Psi_n)/{\sim}$  are conditionally independent.
	
	Notice that on $\{ \wh \cM_\tau = \wh \cM\}$, the decorated LQG surface $\wh \cM_\tau$ is measurable with respect to $\cF_\eps^\mathrm{MOT}$. Therefore  the $(U_i, \Phi, \Psi_1, \dots, \Psi_n)/{\sim}$ are conditionally independent  given $\cF_\eps^\mathrm{MOT}$ and  $\{ \wh \cM_\tau = \wh \cM\}$. But conditioned on $\cF_\eps^\mathrm{MOT}$, the event $\{ \wh \cM_\tau = \wh \cM\}$ is conditionally independent of $(\Phi, \Lambda, \Psi, \Psi_1, \dots, \Psi_n)$, since the event corresponds, in the inductive definition of $\wh \cM_\tau$, to always selecting the point $p_j \in \partial (\bigcup_{i \in I_j} D_i)$ to lie in $\wh \cM$ at every step $j$ and stopping when the whole $\wh \cM$ is explored. Thus, conditioned just on $\cF_\eps^\mathrm{MOT}$, the $m$ complementary LQG surfaces $(U_i, \Phi, \Psi_1, \dots, \Psi_n)/{\sim}$ are conditionally independent. Since $m$ is arbitrary, this gives the result.
	\end{proof}

\subsection{General $(\mathbf c_1, \dots, \mathbf c_n)$ by rotation: proof of Theorem~\ref{thm-main}}\label{subsec-main-proof}

We now use Proposition~\ref{prop-ig-rotate} to deduce Proposition~\ref{prop-main-Keps} from Proposition~\ref{prop-main-0s}, after which Theorem~\ref{thm-main} is straightforward.

\begin{proof}[Proof of Proposition~\ref{prop-main-Keps}] 
Let $\ul \cc = (\cc_1, \dots, \cc_n)$ and $\hat{\ul\cc} = (26 - \ccL, 0, \dots, 0)$. Recall from the discussion at the beginning of Section~\ref{sec-proofs} the vector of IG fields $(\wh\Psi_1,\dots,\wh\Psi_n)$ with central charges $(26-\ccL,0,\dots,0)$ and the $n\times n$ orthogonal matrix $A$ taking $(\Psi_1,\dots,\Psi_n)$ to $(\wh\Psi_1,\dots,\hat\Psi_n)$. 
By Proposition~\ref{prop-ig-rotate}, for any $D$ the IG-decorated LQG surfaces $(D, \Phi, \Psi_1, \dots, \Psi_n)/{\sim_{\ul\cc}}$ and $(D, \Phi, \wh \Psi_1, \dots, \wh \Psi_n)/{\sim_{\wh{\ul\cc}}}$ determine each other (by applying $A$ or $A^{-1}$). Thus, applying Proposition~\ref{prop-main-0s} to $(\Phi, \wh \Psi_1, \dots, \wh \Psi_n, \eta', K)$,  the decorated LQG surfaces $(U_i, \Phi, \Psi_1, \dots, \Psi_{n})/{\sim_{\ul \cc}}$ are conditionally independent given $\cF_\eps^\mathrm{MOT}$. 
\end{proof}

\begin{proof}[Proof of Theorem~\ref{thm-main}]
	Let $\eps> 0$, and define 
	\eqbn
	\cF_\eps = \sigma((\cB_\eps(K \cup \partial \D; {\mathfrak d_\Phi}), \Phi, \Psi_1, \dots, \Psi_n, K)/{\sim}) ,
	\eqen
	 so $\cF= \bigcap_{\eps>0} \cF_\eps$. It suffices to show that the decorated LQG surfaces $(U, \Phi, \Psi_1, \dots, \Psi_n)/{\sim}$ parametrized by the connected components $U$ of $\mathbb D\setminus K$ are conditionally independent given $\cF_\eps$. Indeed, conditional independence is preserved under decreasing limits of $\sigma$-algebras by the backward martingale convergence theorem.
	
	Sample a Poisson point process $\Lambda^* \subset \D \times \R_+$ with intensity measure $\mu_\Phi \times \mathrm{Leb}_{\R_+}$ which, given $\Phi$, is conditionally independent of everything else. For $\eps' > 0$ define $\Lambda_{\eps'} := \{ z \: : \: (z,t) \in \Lambda^* \text{ and } t < (\eps')^{-1}\}$, so that $\Lambda_{\eps'}$ is a Poisson point process with intensity measure $(\eps')^{-1} \mu_\Phi$. Then, segment $\eta'$ using $\Lambda_{\eps'}$ and define $D_i(\eps')$, $\eta_i'(\eps')$, $a_i(\eps')$, $b_i(\eps')$, $c_i(\eps')$, $d_i(\eps')$ for $i = 0, \dots, |\Lambda_{\eps'}|$ as in Definition~\ref{def-eps-local}. Let 
	\eqbn
	I(\eps') := \{ i \: : \: K \cap D_i(\eps') \neq \emptyset\} \quad \text{and} \quad K_{\eps'} = \bigcup_{i \in I(\eps')} D_i(\eps') .
	\eqen
	For $S \subset \ol\D$, write $\Lambda^*|_{S} := \{ (z, t)  \in \Lambda^* \: : \: z \in S\}$. Define
	\[		\cF_{\eps'}^\mathrm{MOT} := \sigma\left( (K_{\eps'},\Phi, \Psi_1, \dots, \Psi_{n}, \{ (\eta'_i(\eps'), a_i(\eps'), b_i(\eps'), c_i(\eps'), d_i(\eps'))\::\: i \in I(\eps')\}, K, \Lambda^*|_{K_{\eps'}})/{\sim} \right).\]
	This is the same $\sigma$-algebra as defined in Proposition~\ref{prop-main-Keps} except that we further keep track of $\Lambda^*$. The same argument still applies so Proposition~\ref{prop-main-Keps} holds in this setting.
	
	For $\tilde \eps' < \eps'$, the set $K_{\tilde \eps'}\subset K_{\eps'}$ is measurable with respect to $K_{\eps'}$ and $\Psi_1|_{K_{\eps'}}, \dots, \Psi_n|_{K_{\eps'}}$. Indeed, the segments of $\eta'$ contained in $K_{\eps'}$ are locally determined by $\wh\Psi_1|_{K_{\eps'}}$~\cite[Lemma 2.4]{gms-harmonic} and hence by the $\Psi_i|_{K_{\eps'}}$.  Since $\eta'$ is defined via flow lines, which are intrinsic to the IG surface, we conclude that the $\sigma$-algebra $\cF_{\eps'}^\mathrm{MOT}$ is decreasing.
	Let $\cF_{0+}^\mathrm{MOT} = \bigcap_{\eps'>0} \cF_{\eps'}^\mathrm{MOT}$. As $K_{\eps'} \downarrow K$ as $\eps' \to 0$, and ${\mathfrak d_\Phi}$ induces the Euclidean topology \cite[Proposition 1.7]{hm-metric-gluing}, we see that a.s.\ $K_{\eps'} \subset \cB_\eps(K \cup \partial \D; {\mathfrak d_\Phi})$ for all sufficiently small $\eps'$. Since $\eta'$ is locally determined by $\wh \Psi_1$ and hence by $\Psi_1, \dots, \Psi_n$,  for all sufficiently small $\eps'$ the set $K_{\eps'}$ is determined by $\cB_\eps(K \cup \partial \D; {\mathfrak d_\Phi})$ and the restrictions of $\Psi_1, \dots, \Psi_n, \Lambda^*$ to this domain.
	 As before, it follows that 
	 $\cF_{0+}^\mathrm{MOT} \subset \wt\cF_\eps$ where
	\[ \wt \cF_\eps := \sigma((\cB_\eps(K \cup \partial \D; {\mathfrak d_\Phi}), \Phi, \Psi_1, \dots, \Psi_n, K, \Lambda^*|_{\cB_\eps(K \cup \partial \D; {\mathfrak d_\Phi})})/{\sim}).\]

	For each connected component $U$ of $\D\backslash K$ let $U_\eps = \cB_\eps(\partial U; {\mathfrak d_\Phi}) \cap U$. Let $\cU_\eps$ be the collection of all decorated LQG surfaces $(U_\eps, \Phi, \Psi_1, \dots, \Psi_n, \Lambda^*|_{U_\eps})/{\sim}$ as $U$ varies. We know $\wt \cF_\eps$ is measurable with respect to $\cF_{\eps'}^\mathrm{MOT}$ and $\sigma(\cU_\eps)$ for all $\eps' > 0$ because $\cB_\eps(K \cup \partial \D; \mathfrak d_\Phi) \subset K_{\eps'} \cup \bigcup_U U_\eps$. Hence $\wt \cF_\eps$ is measurable with respect to $\cF_{0+}^\mathrm{MOT}$ and $\sigma(\cU_\eps)$. On the other hand, $\cF_{0+}^\mathrm{MOT} \subset \wt \cF_\eps$, and  $\sigma(\cU_\eps) \subset \wt \cF_\eps$. We conclude $\wt \cF_\eps = \sigma(\cF_{0+}^\mathrm{MOT}, \cU_\eps)$. 
	
	By Proposition~\ref{prop-main-Keps} and the backward martingale convergence theorem, the collection of decorated LQG surfaces $(U, \Phi, \Psi_1, \dots, \Psi_n, \Lambda^*|_U)/{\sim}$ indexed by the connected components $U$ of $\D \backslash K$ are conditionally independent given  $\cF_{0+}^\mathrm{MOT}$. Since each $(U_\eps, \Phi, \Psi_1, \dots, \Psi_n, \Lambda^*|_{U_\eps})/{\sim}$ is measurable with respect to $(U, \Phi, \Psi_1, \dots, \Psi_n, \Lambda^*|_{U})/{\sim}$ by the locality and conformal covariance of the LQG metric (see Section~\ref{sec-metric-prelim}), the conditional independence still holds if we further condition on $\cU_\eps$. We conclude that the decorated LQG surfaces $(U, \Phi, \Psi_1, \dots, \Psi_n, \Lambda^*|_U)/{\sim}$ are conditionally independent given $\wt \cF_\eps = \sigma(\cF_{0+}^\mathrm{MOT}, \cU_\eps)$. Since the restrictions of a Poisson point process to disjoint domains are independent Poisson point processes in these domains, the decorated LQG surfaces $(U, \Phi, \Psi_1, \dots, \Psi_n)/{\sim}$ are conditionally independent from $\wt \cF_\eps$ given $\cF_\eps$, hence are conditionally independent given $\cF_\eps$. This completes the proof. 
\end{proof}

	In Lemma~\ref{lem-equiv-F} below we give a $\sigma$-algebra equivalent to $\cF$ in Theorem~\ref{thm-main} which does not use the LQG metric. We will not need to use  Lemma~\ref{lem-equiv-F} but state it to give context. Recall that for a measure space $(\Sigma, \cG, \mu)$, the \emph{completion} of $\cG$ is $\sigma(\cG, \cN_\mu)$ where $\cN_\mu$ is the collection of all subsets of $\mu$-null sets in $\cG$. 
	\begin{lemma}\label{lem-equiv-F}
	In the setting of Theorem~\ref{thm-main}, sample a Poisson point process $\Lambda^* \subset \D \times \R_+$ with intensity measure $\mu_\Phi \times \mathrm{Leb}_{\R_+}$ which, given $\Phi$, is conditionally independent of everything else. For $\eps > 0$ let $\Lambda_{\eps} = \{ z \in \D : (z, t) \in \Lambda^* \text{ and } t \leq \eps^{-1}\}$. The curve $\eta'$ is split by the points of $\Lambda_\eps \backslash K$ into a collection of curve segments; let $\{D_i'(\eps)\}$ be the traces of these curve segments. Define
	\[ \begin{gathered}
		I'(\eps) := \{ i \: : \: K  \cap D_i'(\eps) \neq \emptyset\}, \qquad K'_\eps = \bigcup_{i \in I'(\eps)} D_i'(\eps), \\
	\cF_\eps' = \sigma((K'_\eps, \Phi, \Psi_1, \dots, \Psi_n, K, \{(z, t) \in \Lambda^* \: : \: z \in K'_\eps \backslash K \})/{\sim}), \qquad \cF' = \bigcap_{\eps > 0} \cF_\eps'.
	\end{gathered}\]
	Then $\cF$ and $\cF'$ agree up to null sets, i.e., they have the same completion.
\end{lemma}
\begin{proof}
	By the arguments in the proof of  Theorem~\ref{thm-main} immediately above, $\cF'$ is equal to 
	\[\cF'' := \bigcap_{\eps>0} \sigma((\cB_\eps(K \cup \partial \D; {\mathfrak d_\Phi}), \Phi, \Psi_1, \dots, \Psi_n, K, \{(z, t) \in \Lambda^* \: : \: z \in \cB_\eps(K \cup \partial \D; {\mathfrak d_\Phi}) \backslash K \} )/{\sim}) . \]
	By definition $\cF \subset \cF''$, and by Kolmogorov's zero-one law, conditioned on $\cF$, every event in $\cF''$ has conditional probability either $0$ or $1$. Therefore $\cF$ and $\cF''$ have the same completion, as needed. 
\end{proof} 

\subsection{Proofs of remaining theorems}\label{subsec-other}

\begin{proof}[Proof of Theorem~\ref{thm-sle-chordal}]	
	We explain the $\kappa < 4$ case by constructing $\eta$ as a flow line of an $\mathrm{IG}_{\cc_\mathrm{SLE}}$ field; the $\kappa > 4$ and $\kappa = 4$ cases are the same except $\eta$ is a counterflow line or level line of the IG field, respectively. 
	
	Let $(\D, \Phi, -\bbi)$ be an embedding of a sample from $\mathrm{UQD}_{\ccL}$ and independently let $(\D, \Psi_1, \dots, \Psi_n, \Psi, -\bbi)$ be an embedding of a sample from $\mathrm{IG}_{(\cc_1, \dots, \cc_n, \cc_\mathrm{SLE})}$. For concreteness we use the embedding given right below Definition~\ref{def-ig-disk}.
	
	Suppose $\rho_L + \rho_R = \kappa - 6$. 	Note that if $f: \bbH \to \D$ is a conformal map with $f(0) = -\bbi$, then $\Psi \circ f - \chi \arg f'$ has the law of a zero boundary GFF plus the harmonic function $\mathfrak h$ such that $\mathfrak h(x) = 2\pi \chi$ for $x\in(-\infty, 0)$ and $\mathfrak h(x) = 0$ for $x \in (0, \infty)$, where $\chi = \chi(\cc_\mathrm{SLE})$ as defined in~\eqref{eq-c-chi}.
	Thus, the flow line of angle $\theta = \frac{\pi}{\chi\sqrt\kappa}(1+\rho_R)$ is an  $\SLE_\kappa(\rho_L; \rho_R)$ curve from $-\bbi$ to $\bbi$; see Section~\ref{sec-sle-prelim} for details. Let $K$ be this angle $\theta$ flow line of $\Psi$. Since $K$ is a local set of $\Psi$, in the setting of Definition~\ref{def-eps-local}, conditioned on $\bigcup_{i\in I_\tau} D_i$ the event $\{K \subset \bigcup_{i \in I_\tau} D_i\}$ is measurable with respect to $\Psi|_{\bigcup_{i \in I_\tau}D_i}$,  and on this event $K$ is measurable with respect to  $\Psi|_{\bigcup_{i \in I_\tau}D_i}$. Consequently, the event $\{ K \subset \bigcup_{i \in I_\tau} D_i\}$ is measurable with respect to $(\bigcup_{i \in I_\tau} D_i, \Psi)/{\sim}$, and on this event $(\bigcup_{i \in I_\tau} D_i, \Psi, K)/{\sim}$ is measurable with respect to $(\bigcup_{i \in I_\tau} D_i, \Psi)/{\sim}$. This follows from the fact that flow lines are intrinsic to IG surfaces, plus absolute continuity considerations; see the discussion after the statement of \cite[Proposition 3.4]{ig1} for details. 
	Therefore, $K$ is an $\eps$-LQG-IG local set for all $\eps > 0$, and hence an LQG-IG local set.
	
By Theorem~\ref{thm-main}, given $\cF' = \sigma(( \cB_{0+}(K \cup \partial D; \mathfrak d_\Phi) , \Phi, \Psi_1, \dots, \Psi_n, \Psi)/{\sim})$ defined as in~\eqref{eqn-restriction-sigma-algebra} (i.e., $\cF$ where we additionally keep track of $\Psi$ in an infinitesimal neighborhood of $K\cup \partial \D$), the complementary IG-decorated LQG surfaces are conditionally independent.  In the imaginary geometry coupling of $K$ and $\Psi$, the conditional law of $\Psi$ given $K$ is that of the the zero boundary GFF plus a harmonic function with boundary values measurable with respect to $K$ (see Section~\ref{sec-prelim}). The information carried by $\Psi$ in an infinitesimal neighborhood of $K \cup \partial \D$ is thus just the boundary values of the harmonic function, and so is measurable with respect to  $K$. We conclude $\cF' = \cF$, finishing the proof. 
	
	Now instead assume $\rho_L + \rho_R = -2$.  Let $K$ be an angle $\theta$ flow line of $\Psi$ in $\D$ from $\bbi$ to $-\bbi$; varying $\theta$ gives $\SLE_\kappa(\rho_L; \rho_R)$ for any $\rho_L, \rho_R > -2$ with sum $-2$. Proceeding as before completes the argument. 
\end{proof}

\begin{proof}[{Proof of Theorem~\ref{thm-cle}}]
	Suppose $\kappa \in (4,8)$. 
	Let $(\D, \Phi, -\bbi)$ be an embedding of a sample from $\mathrm{UQD}_{\ccL}$ and independently let $(\D, \Psi_1, \dots, \Psi_n, \Psi, -\bbi)$ be an embedding of a sample from $\mathrm{IG}_{\cc_1, \dots, \cc_n, \cc_\mathrm{SLE}}$.

As in the proof of Theorem~\ref{thm-sle-chordal}, the counterflow line of $\Psi$ started at $-\bbi$ targeted at $\bbi$ with an appropriately chosen angle $\theta$ is an $\SLE_\kappa(\kappa-6)$ curve.
 This counterflow line is a local set of $\Psi$. In fact, by choosing a countable dense set of target points, one can realize a \emph{branching $\SLE_\kappa (\kappa - 6)$ process} as a local set of $\Psi$, and thus obtain the CLE$_\kappa$ gasket  as a local set of $\Psi$ \cite[Section 1.2.3]{ig4}. Now the same argument as in the proof of Theorem~\ref{thm-sle-chordal} directly above applies.

	If $\kappa = 4$,  the same argument holds since $\mathrm{CLE}_4$ can be realized as the level loops of a zero boundary GFF, and in this coupling the CLE$_4$ gasket is a local set of this GFF \cite{shef-miller-cle4} (see \cite[Proposition 1]{asw-local-sets} for a published construction).
\end{proof}

\begin{remark}\label{rem-cle}
	When $\kappa \in (\frac83, 4)$, if we assume the $\mathrm{CLE}_\kappa$ gasket can be constructed via flow and counterflow lines of a sample from $\mathrm{IG}_{\cc_\mathrm{SLE}}$, then Theorem~\ref{thm-cle} holds for $\kappa \in (\frac83, 4)$ by the same proof. We expect that the desired imaginary geometry construction of CLE follows from the arguments of \cite{cle-percolations}. Indeed \cite[Theorem 7.8]{cle-percolations} iteratively constructs $\mathrm{CLE}_\kappa$ via the so-called \emph{boundary CLE}, which arises from branching flow and counterflow lines for $\Psi$ \cite[Table 1]{cle-percolations}. 
\end{remark}

We now turn to Theorem~\ref{thm-peeling}. It suffices to show that if $(K_t)_{t \leq T}$ is an LQG-IG peeling process (Definition~\ref{def-quantum-peeling}), then $K_T$ is an LQG-IG local set. 

\begin{lemma}\label{lem-peeling-implies-local}
	If $(K_t)_{t \leq T}$ is an LQG-IG peeling process for $(\Phi, \Psi_1, \dots, \Psi_n)$ and  $\eta'$ is coupled with $(\Phi, \Psi_1, \dots, \Psi_n)$ as in Definition~\ref{def-quantum-local}, then $K_T$ is an LQG-IG local set for $(\Phi, \Psi_1, \dots, \Psi_n, \eta')$.
\end{lemma}

\begin{proof}
	As in Definition~\ref{def-eps-local}, let $\eps > 0$, and let $\Lambda \subset \D$ be a Poisson point process with intensity measure $\eps^{-1}\mu_\Phi$ conditionally independent of everything else given $\Phi$, and define $D_i, \eta_i', a_i, b_i, c_i, d_i$ as explained above Definition~\ref{def-eps-local}. Recall that $I_\tau$ is the set of indices $i$ of the regions $D_i$ discovered by the random inductive peeling process of Definition~\ref{def-eps-local}. To lighten notation define 
	\eqbn
	M_\tau = \bigcup_{i \in I_\tau} D_i .
	\eqen

	 We first check that $\ol{\cB_\delta(K_T; \mathfrak d_\Phi)}$ is an $\eps$-LQG-IG local set. Let $t_0 = 0$ and inductively define the stopping times $t_{j+1} = T \wedge \inf\{ s \geq t_j \: : \: K_s \cap \partial \cB_\delta(K_{t_j}; \mathfrak d_\Phi) \cap \D \neq \emptyset\}$ for the filtration $\cF_t^\delta$ defined in  Definition~\ref{def-quantum-peeling}. 
	 For each $j \geq 0$ let Claim$_j$ be the following claim: ``Conditioned on the decorated LQG surface
	 	\eqb\label{eq-peeled-surface}
	 (M_\tau, \Phi, \Lambda, \Psi_1, \dots, \Psi_n, \{ (\eta'_i, a_i, b_i, c_i, d_i)\::\: i \in I_\tau\}, \Lambda \cap M_\tau)/{\sim},
	 \eqe the event $\{ \ol{\cB_\delta ( K_{t_j}; \mathfrak d_\Phi)} \subset M_\tau\}$ is conditionally independent of $(\Phi, \Psi_1, \dots, \Psi_n, \Lambda)$. Moreover, further conditioning on $\{ \ol{\cB_\delta ( K_{t_j}; \mathfrak d_\Phi)} \subset M_\tau\}$, the decorated LQG surface $(M_\tau, \Phi, (K_s)_{s \leq {t_j}})/{\sim}$ is conditionally independent of $(\Phi, \Psi_1, \dots, \Psi_n, \Lambda)$.'' We will inductively prove that Claim$_j$ holds for all $j$.

	 By definition $K_{t_0} = \partial \D$. By the locality of the LQG metric and the fact that it is invariant under LQG coordinate change (Section~\ref{sec-metric-prelim}) the event $\{ \ol{\cB_\delta(K_{t_0}; \mathfrak d_\Phi)} \subset M_\tau\}$ is measurable with respect to~\eqref{eq-peeled-surface}. Conditioning on this event and on~\eqref{eq-peeled-surface}, by the definition of LQG-IG peeling process, the decorated LQG surface $(M_\tau, \Phi, (K_s)_{ s \leq t_1})/{\sim}$ is conditionally independent of $(\Phi, \Psi_1, \dots, \Psi_n, \Lambda)$. Thus Claim$_0$ holds. Now, assuming Claim$_j$, we will show Claim$_{j+1}$. Combining Claim$_j$ with the fact that $(K_t)_{t \leq T}$ is an LQG-IG peeling process yields the following: 		 
	  Conditioned on~\eqref{eq-peeled-surface},  $\{ \ol {\cB_\delta(K_{t_j}; \mathfrak d_\Phi)} \subset M_\tau\}$ and  $(M_\tau, \Phi, (K_s)_{s \leq t_j})/{\sim}$, the decorated LQG surface $(M_\tau, \Phi, (K_s)_{s \leq t_{j+1}})/{\sim}$ is conditionally independent of $(\Phi, \Psi_1, \dots, \Psi_n, \Lambda)$. By the locality and conformal covariance of the LQG metric, with the above conditioning,  the event $\{ \ol{\cB_\delta(K_{t_{j+1}}; \mathfrak d_\Phi)} \subset M_\tau\}$ is conditionally independent of $(\Phi, \Psi_1, \dots, \Psi_n, \Lambda)$, and further conditioning on this event, $(M_\tau, \Phi, (K_s)_{s \leq t_{j+1}})/{\sim}$ is conditionally independent of $(\Phi, \Psi_1, \dots, \Psi_n, \Lambda)$. We have thus obtained a version of Claim$_{j+1}$ where we not only condition on~\eqref{eq-peeled-surface}, but additionally on $\{ \ol {\cB_\delta(K_{t_j}; \mathfrak d_\Phi)} \subset M_\tau\}$ and  $(M_\tau, \Phi, (K_s)_{s \leq t_j})/{\sim}$. Since $\{\ol{\cB_\delta(K_{t_j}; \mathfrak d_\Phi)} \subset M_\tau\} \subset \{\ol{\cB_\delta(K_{t_{j+1}}; \mathfrak d_\Phi)} \subset M_\tau\}$ and $(M_\tau, \Phi, (K_s)_{s \leq t_{j+1}})/{\sim}$ determines $(M_\tau, \Phi, (K_s)_{s \leq t_{j}})/{\sim}$, by Claim$_j$, the various independence statements still hold if we do not condition on the additional terms above. Therefore Claim$_{j+1}$ holds. 
	  
	  We claim that $t_j = T$ for all sufficiently large $j$. Suppose for the sake of contradiction this is false. Arbitrarily pick a point $p_j \in K_{t_j} \backslash \cB_\delta(K_{t_{j-1}} ; \mathfrak d_\Phi)$ for each $j > 0$. By definition $\mathfrak d_\Phi(p_i, p_j) > \delta$ for every $i < j$, contradicting the fact that $(\ol\D, \mathfrak d_\Phi)$ has the Euclidean topology and thus is compact. Thus $t_j = T$ for all large $j$. By Claim$_j$, we conclude $\ol {\cB_\delta(K_T; \mathfrak d_\Phi)}$ is an $\eps$-LQG-IG set.   
		
	Now, we need to justify that $K_T$ itself is an $\eps$-LQG-IG local set. 
	 Let $I^\delta = \{ i \:: \: D_i \cap \ol{\cB_\delta(K_T; \mathfrak d_\Phi)} \neq \emptyset\}$ and let $I = \{ i \::\: D_i \cap K_T \neq \emptyset\}$.  
	Given~\eqref{eq-peeled-surface}, the event $\{I^\delta = I_\tau\}$ is conditionally independent of $(\Phi, \Lambda, \Psi_1, \dots, \Psi_n)$. Indeed, $\ol {\cB_\delta(K_T; \mathfrak d_\Phi)}$ is $\eps$-LQG-IG local, and $\{ I^\delta = I_\tau \}$ is measurable with respect to~\eqref{eq-peeled-surface} and $(\ol {\cB_\delta(K_T; \mathfrak d_\Phi)}, \Phi, (K_t)_{s \leq T})/{\sim}$. We conclude that $1_{I = I_\tau}= \lim_{\delta \to 0} 1_{I^\delta = I_\tau}$ is conditionally independent of  $(\Phi, \Lambda, \Psi_1, \dots, \Psi_n)$ given~\eqref{eq-peeled-surface}, fulfilling the first condition of Definition~\ref{def-eps-local}.
	Next, since $\ol {\cB_\delta(K_T; \mathfrak d_\Phi)}$ is $\eps$-LQG-IG local, given~\eqref{eq-peeled-surface} and $\{I^\delta = I_\tau\}$ the decorated LQG surface $(M_\tau, \Phi, \ol{\cB_\delta(K_T; \mathfrak d_\Phi)})/{\sim}$ is conditionally independent of $(\Phi, \Lambda, \Psi, \Psi_1, \dots, \Psi_n)$. Since $\ol{\cB_\delta(K_T; \mathfrak d_\Phi)} \downarrow K_T$ and  $1_{I = I_\tau}= \lim_{\delta \to 0} 1_{I^\delta = I_\tau}$, we obtain the second condition of  Definition~\ref{def-eps-local}.
	Thus $K_T$ is an $\eps$-LQG-IG local set for all $\eps > 0$, and hence is an LQG-IG local set.
\end{proof}

\begin{proof}[{Proof of Theorem~\ref{thm-peeling}}]
	This follows from Theorem~\ref{thm-main} and Lemma~\ref{lem-peeling-implies-local}. 
\end{proof}

We  next address the unit boundary length LQG disk with a marked interior point, whose law is denoted $\mathrm{UQD}_{\ccL}^\bullet$ (Definition~\ref{def-lqg-disk}). The main result is Theorem~\ref{thm-main-pt} for the LQG-IG local sets  defined below, from which all the other theorems follow.

\begin{definition}\label{def-eps-local-pt}
	Consider the setting of Definition~\ref{def-eps-local} except that $(\D, \Phi, 0, -\bbi)$ is an embedding of a sample from $\mathrm{UQD}_{\ccL}^\bullet$ instead. 
	An \textbf{$\eps$-LQG-IG local set} $K \subset \ol \D$ is a random compact set coupled with $(\Phi,  \Psi_1, \dots, \Psi_{n}, \eta')$ such that $K \cup \partial \D$ is a.s.\ connected, almost surely $0 \in K$, and 
	the following holds. 
	
	Independently of $(\Psi_1, \dots,\Psi_n, K, \eta')$, sample a Poisson point process $\Lambda \subset \D$ with intensity measure $\eps^{-1} \mu_\Phi$ and define $D_i, \eta_i', a_i, b_i, c_i, d_i$ as in Definition~\ref{def-eps-local}. 
	Independently of everything else sample a nonnegative integer $T \geq 0$ with $\P[T = t] = 2^{-t-1}$, let $I_0$ be the set of indices $i$ such that $D_i(\eps) \cap \partial \D \neq \emptyset$, and for $j \geq 0$  inductively define $I_{j+1}$ from $I_j$ by independently sampling a point $p_{j+1} \in (\partial \bigcup_{i \in I_j} D_i) \backslash \partial \D$ from the LQG length measure, then setting $I_{j+1} = I_j \cup \{ i \: : \: p_{j+1} \in \partial D_i\}$. We stop the process either at time $T$ or at the time $t$ when $I_t = \{ 0,\dots, N+1\}$, whichever is earlier; call this time $\tau$. 
	
	Conditioned on $\{0 \in \bigcup_{i \in I_\tau} D_i \}$ and on  $( \bigcup_{i \in I_\tau} D_i, \Phi, \Psi_1, \dots, \Psi_n, \{ \eta_i', a_i, b_i, c_i, d_i) : i \in I_{\tau}\}, 0, \Lambda \cap \bigcup_{i \in I_\tau} D_i)/{\sim}$, the event $\{K \subset \bigcup_{i \in I_\tau} D_i\}$ is conditionally independent of $(\Phi, \Psi_1, \dots, \Psi_n,  \Lambda)$.
	Moreover,  further conditioning on $\{K \subset \bigcup_{i \in I_\tau} D_i\}$, the decorated LQG surface $( \bigcup_{i \in I_\tau} D_i, \Phi, K)/{\sim}$ is conditionally independent of $(\Phi, \Psi, \Psi_1, \dots, \Psi_n, \eta', \Lambda)$. 
	
	$K$ is called a \textbf{LQG-IG local set} if it is an $\eps$-LQG-IG local set for all  $\eps > 0$. 
\end{definition}
Definition~\ref{def-eps-local-pt} for $\mathrm{UQD}_{\ccL}^\bullet$ is the same as Definitions~\ref{def-eps-local} and~\ref{def-quantum-local} for $\mathrm{UQD}_{\ccL}$, except that we require that $K$ contains the marked bulk point $0$, and in our random exploration of the disk we condition on $0$ lying in the explored region. 
The following lemma is an analog of Proposition~\ref{prop-eps-indep} for $\mathrm{UQD}_{\ccL}^\bullet$.
\begin{lemma}\label{lem-eps-indep-pt}
	Consider the setting of Definition~\ref{def-eps-local} with $n=1$. Weight by $\mu_\Phi(\D)$, sample a  point $z\in \D$ from $\mu_\Phi/\mu_\Phi(\D)$ conditionally independently of everything except $\Phi$, and condition on $\{ z \in \bigcup_{i \in I_\tau} D_i\}$. 
	Let $L_1, \dots, L_m$ be simple loops in the interior of $\bigcup_{i \in I_\tau} D_i$ and let $O_k$ be the bounded connected component of $\C \backslash L_k$. Suppose $(\bigcup_{i \in I_\tau} D_i, \Phi, L_1, \dots, L_m)/{\sim}$ is measurable with respect to $\cM_\tau^\bullet:=( \bigcup_{i \in I_\tau} D_i, \Phi, \Psi_1, \{ \eta_i', a_i, b_i, c_i, d_i) : i \in I_{\tau}\}, \Lambda, z)/{\sim}$, and the $O_k$ are pairwise disjoint. Then conditioned on $\cM_\tau^\bullet$, the IG-decorated LQG surfaces $(O_1, \Phi, \Psi)/{\sim}, \dots, (O_m, \Phi, \Psi)/{\sim}$ are conditionally independent.
\end{lemma}
\begin{proof}
For a finite measure $M$ let $M^\# = M/|M|$ be the corresponding probability measure. 
In the  setting of Definition~\ref{def-eps-local}, weighting by $\mu_\Phi(\D)$, sampling $z \sim \mu_\Phi^\#$ and restricting to $\{ z \in \bigcup_{i \in I_\tau}D_i\}$ is equivalent to weighting by $\mu_\Phi(\bigcup_{i \in I_\tau} D_i)$ then sampling $z \sim (\mu_\Phi|_{\bigcup_{i \in I_\tau} D_i})^\#$. Thus, the present setup can be obtained starting with the setup of  Proposition~\ref{prop-eps-indep}, weighting $\cM_\tau$ by its LQG area and sampling a point in $\cM_\tau$ from the probability measure proportional to its area measure to get $\cM_\tau^\bullet$. This weighting and sampling does not affect the conditional law of the $(U_i, \Phi, \Psi)/{\sim}$ given $\cM_\tau$, so  Proposition~\ref{prop-eps-indep} gives the desired conditional independence. 
\end{proof}

Now we state and prove Theorem~\ref{thm-main-pt}, which is the $\mathrm{UQD}_{\ccL}^\bullet$ variant of Theorem~\ref{thm-main}.

\begin{thm}\label{thm-main-pt}
	Suppose $(\Phi, \Psi_1, \dots, \Psi_n, \eta', K)$ are as in Definition~\ref{def-eps-local-pt}. Let $\cF = \bigcap_{\eps > 0} \sigma((\cB_\eps(K \cup \partial \D; {\mathfrak d_\Phi}), \Phi, \Psi_1, \dots, \Psi_n, K, 0 )/{\sim})$. Then the decorated LQG surfaces $(U, \Phi, \Psi_1, \dots, \Psi_n)/{\sim}$ parametrized by the connected components $U$ of $\D \backslash K$ are conditionally independent given $\cF$. 
\end{thm}
\begin{proof}
	Using Lemma~\ref{lem-eps-indep-pt} in place of Proposition~\ref{prop-eps-indep}, we can prove the $\mathrm{UQD}_{\ccL}^\bullet$ analogs of Proposition~\ref{prop-main-0s}, Proposition~\ref{prop-main-Keps} and Theorem~\ref{thm-main} by the same arguments. The last of these analogs is the desired result. 
\end{proof}

\begin{proof}[Proof of Theorem~\ref{thm-sle-interior}]
By the same argument as in Theorem~\ref{thm-sle-chordal} we see that  $K$ is an LQG-IG local set so Theorem~\ref{thm-main-pt} applies. 
\end{proof}

\begin{proof}[Proof of Theorem~\ref{thm-bm}]
Brownian motion is conformally invariant, and local in the sense that for any domains $D_1 \subset D_2$, Brownian motion in $D_1$ started at $z \in D_1$ and run until it exits $D_1$ agrees in law with Brownian motion in $D_2$ started at $z$ and run until it exits $D_1$. Consequently $K$ is an LQG-IG local set, so  Theorem~\ref{thm-main-pt} gives the result. 
\end{proof}

\begin{proof}[Proof of Theorem~\ref{thm-ball}]
	The LQG metric ball $K$ is an LQG-IG local set by the locality property of the LQG metric. Thus the claim follows from  Theorem~\ref{thm-main-pt}. 
\end{proof}

\section{Independence of unexplored regions of the discretized LQG disk}\label{sec-n=1}

The goal of this section is to prove Proposition~\ref{prop-eps-indep} which we used earlier to prove Theorem~\ref{thm-main}. To that end we first rephrase it in a more convenient way, as Theorem~\ref{thm-eps-indep}. 

Fix $\eps > 0$ and $\ccL > 25$,  let $(\D, \Phi, -\bbi)$ be an embedding of a sample from $\mathrm{UQD}_{\ccL}$, and let  $\Lambda\subset \D$ be a Poisson point process with intensity measure $\eps^{-1}\mu_\Phi$. Independently of $(\Phi, \Lambda)$, let $(\D, \Psi, -\bbi)$ be an embedding of a sample from $\mathrm{IG}_{26 - \ccL}$ and let $\eta'$ be the counterclockwise space-filling $\SLE_{\kappa'}$ measurable with respect to $\Psi$. Here, the LQG parameter $\gamma \in (0,2)$ and SLE parameter $\kappa' > 4$ satisfy $\kappa' = \frac{16}{\gamma^2}$ since $\ccL+ (26 - \ccL) = 26$. Condition on $|\Lambda| > 0$ and write $N = |\Lambda| - 1$.  
The curve $\eta'$ is split by the points $\Lambda$ to give domains $D_i$ decorated by curve segments $\eta_i':[0,t_i] \to \ol \D$ and boundary points $a_i, b_i, c_i, d_i$ for $0\leq i \leq  N+1$, as in Definition~\ref{def-eps-local}. Recall 
$a_i  = \eta'_i(0)$ and $d_i = \eta_i'(t_i)$, while $b_i$ (resp.\ $c_i$) is the furthest point on the clockwise (resp.\ counterclockwise) arc of $\partial D_i$ from $a_i$ to $d_i$ such that the boundary arc from $a_i$ to $b_i$ (resp.\ $c_i$) is a subset of $\ol{(\partial \D\cup \bigcup_{j \leq i} D_i)}$.
See Figure~\ref{fig-discretization}.

Independently of $(\Phi, \Psi, \eta', \Lambda)$, sample a nonnegative integer $T \geq 0$ with $\P[T = t] = 2^{-t-1}$, let $I_0$ be the set of indices $i$ such that $D_i  \cap \partial \D \neq \emptyset$, and for $j \geq 0$  inductively define $I_{j+1}$ from $I_j$ by independently sampling a point $p_{j+1} \in (\partial \bigcup_{i \in I_j} D_i) \backslash \partial \D$ from the LQG length measure, then setting $I_{j+1} = I_j \cup \{ i \: : \: p_{j+1} \in \partial D_i\}$. We stop the process either at time $T$ or at the time $t$ when $\bigcup_{i \in I_t} D_i = \ol \D$, whichever is earlier; call this time $\tau$. Let 
\begin{equation}
M_\tau := \bigcup_{i \in I_\tau} D_i \quad \text{and} \quad	\cM_\tau = (M_\tau, \Phi, \Psi, \{ (\eta'_i, a_i, b_i, c_i, d_i)\: : \: i \in I_\tau \})/{\sim} . 
\end{equation}
See Figure~\ref{fig-thm-indep} (right).

\begin{thm}\label{thm-eps-indep}
	In the setting immediately above, let $L_1, \dots, L_m$ be simple loops in the interior of $M_\tau$ and let $O_k$ be the bounded connected component of $\C \backslash L_k$. Suppose $(M_\tau, \Phi, L_1, \dots, L_m)/{\sim}$ is measurable with respect to $\cM_\tau$, and the $O_k$ are pairwise disjoint. Then conditioned on $\cM_\tau$, the IG-decorated LQG surfaces $(O_1, \Phi, \Psi)/{\sim}, \dots, (O_m, \Phi, \Psi)/{\sim}$ are conditionally independent.  
\end{thm}

\begin{proof}[Proof of Proposition~\ref{prop-eps-indep}]
	When $|\Lambda| > 0$
	the claim is equivalent to  Theorem~\ref{thm-eps-indep}. When $|\Lambda| = 0$ the claim is vacuously true. 
\end{proof}

Our proof of Theorem~\ref{thm-eps-indep} involves non-probability measures. In the setting of probability theory, independent random variables $X$ and $Y$ have a joint law that  factorizes as a product measure $\mu_X \times \mu_Y$. In the non-probability setting, the notion of independence is no longer well defined, but factorization still makes sense. For instance, Lebesgue measure on $\R^2$ factorizes as a product of Lebesgue measures on $\R$. Our argument involves showing that certain measures, densities, and combinatorial structures factorize. 

In Section~\ref{subsec-mot} we explain the mating-of-trees theorem which identifies space-filling $\SLE_{\kappa'}$-decorated $\gamma$-LQG  disks with correlated 2D Brownian motion excursions. In Section~\ref{subsec-decomp-cells} we use the Markov property of Brownian motion to decompose the decorated $\gamma$-LQG disk into ``cells'' which are conditionally independent given their boundary lengths. We also obtain a factorized form for the joint density of the boundary lengths. In Section~\ref{subsec-quilt-def} we define decorated planar maps which we call ``quilts'' whose faces correspond to cells and whose edges are labelled by the corresponding LQG lengths, and we obtain a factorization for the law of the edge lengths for a given planar map. In Section~\ref{subsec-indep-quilts} we prove that conditioned on the existence of a specified subquilt, the complementary connected subquilts are conditionally independent. The inputs are the factorization from the previous section, a factorization of uniform measures, and a combinatorial factorization of planar maps whose proof is deferred to  Section~\ref{subsec-topo}. This last result is a generalization of Proposition~\ref{prop-meander} on  the independence of arc diagrams of a uniform meander given its winding function. Finally in Section~\ref{subsec-indep-main-proof} we prove Theorem~\ref{thm-eps-indep}.

\subsection{Mating-of-trees for the LQG disk}\label{subsec-mot}
In this section, we review the construction of Brownian excursions in $\R_+ \times \R$ and $\R_+^2$, then recall in Proposition~\ref{prop-mot-disk} the mating-of-trees theorem which identifies an SLE-decorated LQG disk with a Brownian excursion in $\R_+^2$.

We first recall Brownian motion in the cone. Let $\mathbbm a^2 = 2/\sin(\frac{\pi\gamma^2}4)$ be the mating-of-trees variance computed in \cite[Theorem 1.3]{ars-fzz}.
Consider Brownian motion $(L_t, R_t)$ with $\Var(L_t) = \Var(R_t) = \mathbbm a^2 t$ and $\Cov(L_t, R_t) = -\cos(\frac{\pi \gamma^2}4)\mathbbm a^2 t$. Let $\mu(t; z)$ be the law of Brownian motion started at $z$ and run for time $t$, and let $\mu(t; z,w)$ be the disintegration of $\mu(t; z)$ over the endpoint of the Brownian motion, i.e.\ $\mu(t; z) = \int_\C \mu(t; z, w)\, dw$ and the measure $\mu(t; z,w)$ is supported on trajectories ending at $w$. We emphasize that $|\mu(t;z)|=1$ (that is, $\mu(t;z)$ is a probability measure) but $|\mu(t; z,w)|$ is typically not 1 (rather, $|\mu(t;z,\cdot)|$ is the density of a two-dimensional Gaussian random variable with mean $z$ and covariance matrix $t\op{Id}$). The Markov property of Brownian motion implies  that for any $t_1, t_2 > 0$ and $z_1, z_2 \in \C$ we have $\mu(t_1+t_2; z_1, z_2) = \int_\C \mu(t_1; z_1,w) \mu(t_2; w, z_2)\, dw$ where the equality means that when we sample a pair of paths from the right hand side, concatenating them gives a sample from the left hand side.

We now recall the construction of non-probability measures which correspond to Brownian motion started at a vertex of a cone and restricted to staying in the cone. 
Let $E_{\R_+ \times \R}$ and $E_{\R_+^2}$ be the events that a Brownian motion trajectory stays in $\R_+ \times \R$ and $\R_+^2$ respectively. \cite{shimura1985} allows us to define the Brownian excursion in $\R_+ \times \R$ and Brownian excursion in $\R_+^2$ by
\[\mu_{\R_+ \times \R}(t; 0) = \lim_{\delta \to 0}  \delta^{-1} (1_{E_{\R_+ \times \R}} \mu(t; \delta)), \qquad \mu_{\R_+^2} (t; 0) = \lim_{\delta \to 0} \delta^{-\frac4{\gamma^2}} (1_{E_{\R_+^2}} \mu(t; \delta e^{i\pi/4})).  \]
Here the limits are in the sense of weak convergence on $C([0,t], \C)$ equipped with the uniform topology.
Indeed, for the  first limit, the existence of $\lim_{\delta \to 0} \delta^{-1} |1_{E_{\R_+ \times \R}}\mu(t; \delta)|$ is stated in \cite[(4.1)]{shimura1985}, and the conditional law of $\mu(t; \delta)$ given $E_{\R_+\times\R}$  has a limit as $\delta \to 0$ by  \cite[Theorem 2]{shimura1985}. The second limit is likewise justified. 
Define by disintegration the measures $\mu_{\R_+ \times \R}(t; 0, z)$ and $\mu_{\R_+^2}(t; 0, z)$ on the space of paths from $0$ to $z$. 
\cite[Theorem 2]{shimura1985} gives the following Markov properties: 
\eqb\label{eq-markov}
\begin{gathered}
	\mu_{\R_+ \times \R} (t_1 + t_2; 0, z) = 1_{E_{\R_+ \times \R}} \int_{\R_+ \times \R}  \mu_{\R_+ \times \R}(t_1; 0, w) \mu(t_2; w, z)\, dw, \\
	\mu_{\R_+^2} (t_1 + t_2; 0, z) = 1_{E_{\R_+^2}} \int_{\R_+^2}  \mu_{\R_+^2}(t_1; 0, w) \mu(t_2; w, z)\, dw.
\end{gathered}
\eqe
Here, on the right hand sides, the restriction to $E_{\R_+ \times \R}$ or $E_{\R_+^2}$ applies to the path from $0$ to $z$ obtained by \emph{concatenating} the pair of paths. 
Note that the non-probability measures in these identities induce  weightings. For instance, if we sample  $(L_t, R_t)$ from $\mu_{\R_+^2}(t_1 + t_2; 0, z)$, then the law of $(L_t, R_t)|_{[0,t_1]}$ is $\int_{\R_+^2}\mu_{\R_+^2}(t_1; 0, w) \, dw$ weighted by $\mu(t_2; L_{t_1} +  \bbi R_{t_1}, z)[E_{\R_+^2}]$.

Finally, we define the Brownian excursion in $\R_+^2$ from $\bbi$ to $0$ via 
\[\mu_{\R_+^2} (t; \bbi, 0) := 1_{E_{\R_+^2}} \int_{\R_+^2}  \mu_{\R_+\times \R} \left(\frac t2; \bbi, z\right) \times \mu_{\R_+^2} \left(\frac t2; z, 0\right)\, dz,\]
where we define $\mu_{\R_+\times \R} (t/2; \bbi, z)$ from $\mu_{\R_+\times \R} (t/2; 0, z-\bbi)$ by translation, and define $\mu_{\R_+^2} (t/2; z, 0)$ from $\mu_{\R_+^2} (t/2; 0, z)$ by time-reversal. By its definition, the measure  $\mu_{\R_+^2}(t; \bbi, 0)$ inherits an analogous Markov property from~\eqref{eq-markov}.

We can now state the mating-of-trees theorem for the LQG disk, see Figure~\ref{fig-BM} (left). 
Recall that for a finite measure $M$ the probability measure $M^\#$ is defined to be $M/|M|$. The following mating-of-trees theorem was first shown by \cite{wedges} for the restricted range $\gamma \in (\sqrt2, 2)$, then proved in full by \cite{ag-disk}.
\begin{proposition}[{\cite[Theorem 1.1]{ag-disk}}]\label{prop-mot-disk}
	Let $\gamma \in (0,2)$ and $\kappa' = \frac{16}{\gamma^2}$. 
	Let $(\D, \Phi, -\bbi)$ be an embedding of a $\gamma$-LQG disk conditioned to have unit boundary length, and let $\eta'$ be an independent counterclockwise space-filling $\SLE_{\kappa'}$ in $(\D, -\bbi)$.  Parametrize $\eta'$ by $\mu_\Phi$-area, and let $L_t$ (resp.\  $R_t$) be the LQG length of the counterclockwise (resp.\ clockwise) boundary arc of $\D \backslash \eta'([0,t])$ from $-\bbi$ to $\eta'(t)$. Then the law of the process $t \mapsto (L_t, R_t)$ is $ (\int_0^\infty \mu_{\R_+^2}(t; \bbi, 0) \, dt )^\#$. See Figure~\ref{fig-BM} (left).
\end{proposition} 

\begin{figure}[ht!]
	\begin{center}
		\includegraphics[scale=0.35]{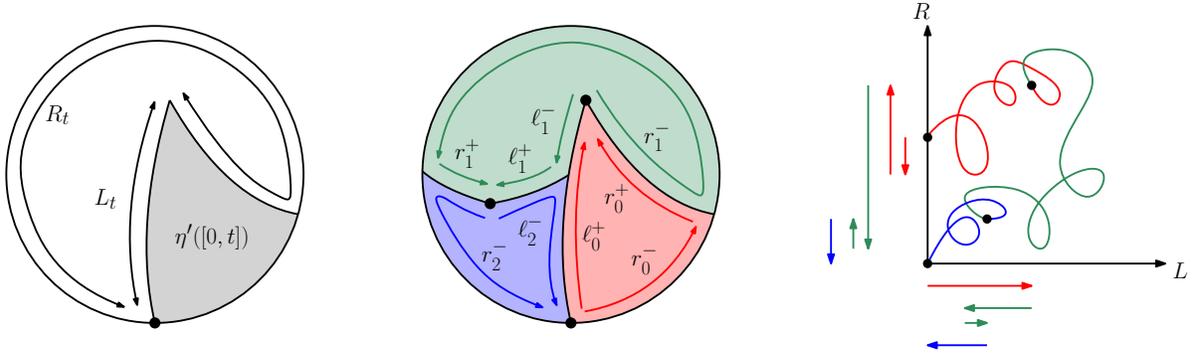}%
		\caption{\label{fig-BM}
			\textbf{Left:} A $\gamma$-LQG disk decorated by an independent counter-clockwise space-filling $\SLE_{\kappa'}$ curve $\eta'$ where $\kappa' = \frac{16}{\gamma^2}$. Parametrizing $\eta'$ by LQG area, $L_t$ and $R_t$ are the left and right LQG boundary lengths of the region not yet explored by $\eta'$ at time $t$. \textbf{Middle:} A Poisson point process with intensity given by the LQG area measure. Here it has two points, cutting $\eta'$ into three segments, which in order of traversal are red, green and blue. We mark out some LQG lengths for each region. Note that when $\gamma \in (\sqrt2, 2)$ these segments are not simply connected, see Figure~\ref{fig-discretization} (right). \textbf{Right:} The two points cut the Brownian motion into three paths.  The lengths $(\ell_0^+, r_0^-, r_0^+)$, $(\ell_1^-, \ell_1^+, r_1^-, r_1^+)$ and $(\ell_2^-, r_2^-)$ correspond to the lengths marked by arrows. 
		}
	\end{center}
\end{figure}

\subsection{Decomposition into cells and factorization}\label{subsec-decomp-cells}
In the setting of Theorem~\ref{thm-eps-indep}, for $\gamma \in (0,\sqrt2]$ let $\cC_i := (D_i, \Phi, \eta_i', a_i,b_i,c_i,d_i)/{\sim}$ be the $i$th (simply connected) curve-decorated LQG surface explored by the curve; a similar definition is given below for $\gamma \in (\sqrt2, 2)$. 
The goal of this section is to prove that the side lengths of $\cC_0, \dots, \cC_{N+1}$ have a joint probability density function that factors, and given these side lengths the $\cC_i$ are conditionally independent (Proposition~\ref{prop-decomp-quilt-cell}). In brief, this is a consequence of Proposition~\ref{prop-mot-disk} and the Markov property of Brownian motion. 

For $\gamma \in (\sqrt2, 2)$ the definition of $\cC_i$ is more complicated for topological reasons: the curve $\eta_i'$ fills out a countable chain of simply-connected domains, each parametrizing a curve-decorated LQG surface. See  Figure~\ref{fig-discretization} (right). Let $\cC_i$ be the decorated generalized LQG surface given by the chain of curve-decorated LQG surfaces, further decorated by boundary points $a_i, b_i, c_i, d_i$. Note that $(D_i, \Phi, \eta_i', a_i, b_i, c_i, d_i)/{\sim}$ is obtained from $\cC_i$ by identifying pairs of boundary points of $\cC_i$.

The decorated LQG surfaces $\cC_1, \dots, \cC_N$ each have four marked points, $\cC_0$ has three marked points (since $a_0 = b_0$), and $\cC_{N+1}$ has two marked points (since $b_{N+1} = c_{N+1} = d_{N+1}$). Given $\cC_0, \dots, \cC_{N+1}$, we can recover $(\D, \Phi, -i, \eta', \Lambda_\eps)/{\sim}$ by conformal welding. See for instance the discussion on conformal welding in  \cite[Section 3.5]{wedges}. 

Let $(\ell_0^+, r_0^+, r_0^-)$ be the boundary arc lengths of $\cC_0$ ordered clockwise from $a_0$, for $1 \leq i \leq N$ let $(\ell_i^-, \ell_i^+, r_i^+, r_i^-)$ be the boundary arc lengths of $\cC_i$ ordered clockwise from $a_i$, and let $(\ell_{N+1}^-, r_{N+1}^-)$ be the boundary arc lengths of $\cC_{N+1}$ ordered clockwise from $a_{N+1}$. See Figure~\ref{fig-BM} (middle). 
These lengths can be expressed explicitly in terms of the corresponding Brownian motion increments, see \eqref{eq-bm-length} below.

Since the Brownian motion remains in the cone, restricting to $\{ N = n\}$, we have the constraints
\begin{gather}
	\ell_0^+ + \sum_{j =1}^{i-1} (\ell_j^+ - \ell_j^-) - \ell_i^- > 0 \:\: \text{ for }1 \leq i \leq n,  \:\: \quad  1 + \sum_{j =0}^{i-1}  (r_j^+ - r_j^-) - r_i^- > 0 \:\:  \text{ for }0 \leq i\leq n,  \label{eq:SN2} \\
	\ell_{n+1}^- = \ell_0^+ + \sum_{i=1}^n (\ell_i^+ - \ell_i^-), \qquad \qquad r_{n+1}^- = 1 + \sum_{i=0}^{n} (r_i^+ - r_i^-). \label{eq-n+1-lengths}
\end{gather}
Let $S_n \subset \R_+^3 \times (\R_+^4)^n$ be the set of lengths $((\ell_0^+, r_0^-, r_0^+)$, $(\ell_i^-, \ell_i^+, r_i^-, r_i^+)_{i \leq n})$ satisfying~\eqref{eq:SN2}. 
Let $\mathrm{Leb}_{S_n}$ be Lebesgue measure on $S_n$. 
The goal of this section is to prove the following factorized description of the joint law of these lengths, and the conditional independence of the $\cC_i$ given these lengths.

\begin{proposition}\label{prop-decomp-quilt-cell}
	There exists a constant $Z_\eps > 0$ and functions $P^\eps_\mathrm{init}, P^\eps, P^\eps_\mathrm{end}$ such that the law of $((\ell_0^+, r_0^-, r_0^+)$, $(\ell_i^-, \ell_i^+, r_i^-, r_i^+)_{1\leq i \leq N})$ is $\cL^\eps := \sum_{n\geq 0} \cL^\eps_n$, where, defining $\ell_{n+1}^-$ and $r_{n+1}^-$ via~\eqref {eq-n+1-lengths}, 
	\eqb\label{lem-Lepsk}
	\cL^\eps_n :=  Z_\eps^{-1} P^\eps_\mathrm{init}(\ell_0^+, r_0^-, r_0^+) \times \prod_{i=1}^n P^\eps(\ell_i^-, \ell_i^+, r_i^-, r_i^+) \times  P^\eps_\mathrm{end}(\ell_{n+1}^- , r_{n+1}^-) \, \mathrm{Leb}_{S_n}.
	\eqe
	Conditioned on $\{N = n\}$ and on these $4n+3$ lengths, the LQG surfaces $\cC_0, \dots, \cC_{n+1}$ are conditionally independent, and the conditional law of $\cC_i$ depends only on $(\ell_i^-, \ell_i^+, r_i^-, r_i^+)$ (where we write $\ell_0^- = \ell_{n+1}^+ = r_{n+1}^+ = 0$). 
\end{proposition}
In our proof, we will obtain explicit  descriptions of  $P^\eps_\mathrm{init}, P^\eps$ and $P^\eps_\mathrm{end}$ in terms of Brownian motion, but these are not important for later arguments. 
The rest of this section is devoted to the proof of Proposition~\ref{prop-decomp-quilt-cell}. 

First, the Markov property of Brownian motion gives the following. Let $n \geq 0$ and $t_0, \dots,  t_{n+1} > 0$, then 
\[ \mu_{\R_+^2}\left(\sum_{i=0}^{n+1} t_i; \bbi, 0\right) = 1_{E_{\R_+^2}}\int_{\C^{n+1}}  \mu_{\R_+ \times \R}(t_0; \bbi, z_1) \left(\prod_{i=1}^n \mu(t_i; z_i, z_{i+1})\right) \mu_{\R_+^2}(t_{n+1}; z_{n+1}, 0) \, \prod_{i=1}^{n+1} dz_i.\]
Notice that the measures $\mu(t_i; z_i, z_{i+1})$ and $\mu_{\R_+\times \R}(t_0; \bbi, z_1)$ are not supported on the space of paths in $\R_+^2$, but the indicator $1_{E_{\R_+^2}}$ enforces that we only consider concatenated paths that stay in $\R_+^2$. We now rephrase in terms of the increments of the Brownian motion. Setting $w_0 = z_1 - \bbi$ and $w_i = z_{i+1}-z_i$ for $1 \leq i \leq n$, 
\eqb\label{eq-exc-fixed-times}
\mu_{\R_+^2} \left( \sum_{i=0}^{n+1} t_i; \bbi, 0\right) = 1_{E_{\R_+^2}} \int_{\C} \mu_{\R_+ \times \R}(t_0; 0, w_0)\, dw_0  \prod_{i=1}^n \left(\int_\C \mu(t_i; 0, w_i)\, dw_i\right)\, \mu_{\R_+^2}\biggl(t_{n+1}; \sum_{j=0}^n w_j + \bbi, 0\biggr).
\eqe
That is, if we translate and concatenate the $(n+2)$ paths of a sample from the right hand side, we get a single path from $\bbi$ to $0$ whose law is $\mu_{\R_+^2}( \sum_{i=0}^{n+1} t_i; \bbi, 0)$.  

Equation~\eqref{eq-exc-fixed-times} explains how to decompose a Brownian excursion of duration $t = \sum_{i=0}^{n+1} t_i$ into $n+2$ trajectories of durations $t_0, \dots, t_{n+1}$. Next, we will choose $t$ randomly, then partition $[0,t]$ using a Poisson point process.
Let $\mathrm{Exp}_{\lambda}(dx) = 1_{x>0} \lambda e^{-\lambda x}$ be the law of the exponential random variable with intensity $\lambda$.  
\begin{lemma}\label{lem-eps-time-decomp}
	Let $\eps > 0$. Sample $t \sim \mathrm{Leb}_{\R_+}$, let $\Lambda$ be a Poisson point process on $[0,t]$ with intensity $\eps^{-1} \mathrm{Leb}_{[0,t]}$, and let $t_0, \dots, t_m$ be the successive lengths of the intervals comprising $[0,t] \backslash \Lambda$. Then the law of $(t_0, \dots, t_m)$ is $\eps\sum_{m \geq 0} \prod_{i=0}^m \mathrm{Exp}_{1/\eps}(dt_i)$.	
\end{lemma}
\begin{proof} 
	Let $P$ denote the law of a Poisson point process on $\R_+$ with intensity $\eps^{-1}\mathrm{Leb}_{\R_+}$.
	Palm's theorem for Poisson point processes \cite[Page 5]{kallenberg-random-measures} implies the following two procedures give the same law of pairs $(\Lambda, t)$ where $\Lambda \subset \R_+$ is countable and $t \in \Lambda$:
	\begin{itemize}
		\item 
		Sample $(t, \Lambda_0) \sim (\eps^{-1}\mathrm{Leb}_{\R_+}) \times P$, and let $\Lambda = \Lambda_0 \cup \{t\}$. 
		
		\item Sample $(M, \Lambda) \sim \mathrm{Count}_{\Z_+} \times P$, and let $t$ be the $M$th smallest element of $\Lambda$. Here, $\mathrm{Count}_{\Z_+}$ denotes the counting measure on $\Z_+ = \{ 1, 2, \dots\}$.
	\end{itemize}
	Note that if $\Lambda \sim P$, the law of the lengths of successive intervals comprising $\R_+ \backslash \Lambda$ is $(\mathrm{Exp}_{1/\eps})^{\mathbb Z_+}$. Thus, multiplying both laws by $\eps$ and looking at the laws of the intervals of $[0,t] \backslash \Lambda$ gives the result.
\end{proof}
Write $\mu^\eps(z, w) := \int_0^\infty \mu(t; z, w) \mathrm{Exp}_{1/\eps}(dt)$ and similarly define $\mu_{\R_+^2}^\eps$ and $\mu_{\R_+\times \R}^\eps$. 
Using Lemma~\ref{lem-eps-time-decomp} then~\eqref{eq-exc-fixed-times}, 
\begin{align}
	\int_0^\infty &\mu_{\R_+^2}(t; \bbi, 0) \, dt = \eps \sum_{m \geq 0} \int_{\R_+^{m+1}}\mu_{\R_+^2}\left(\sum_{i=0}^{m} t_i; \bbi, 0\right) \, \prod_{i=0}^{m} \mathrm{Exp}_{1/\eps}(dt_i) \label{eq-decomp-bm} \\
	&= \eps \mu_{\R_+^2}^\eps(\bbi, 0)  + \eps \sum_{n \geq 0}1_{E_{\R_+^2}}  \int_\C \mu_{\R_+ \times \R}^\eps(0, w_0) \, dw_0 \prod_{i=1}^n \left(\int_\C \mu^\eps(0, w_i)\, dw_i \right) \, \mu_{\R_+^2}^\eps\left(\sum_{i=0}^n w_j + \bbi, 0\right). \nonumber
\end{align}
Equation \eqref{eq-decomp-bm} says that if we take the total time duration $t$ to be ``sampled" from Lebesgue measure on $\R_+$, then the Brownian excursion measure $\mu_{\R_+^2}(t; \bbi, 0)$ can be written as a countable sum of measures indexed by the total number of points $m \geq 0$ in the Poisson point process. For each of these measures for $m \geq 1$ ($n = m-1 \geq 0$), the Brownian excursion can be decomposed as the union of $n+2$ sub-paths. We can sample these $n+2$ sub-paths as Brownian bridges with exponential durations, restricted to the event that they stay in the cone. 

\begin{proposition}\label{prop-disk-BM-decomp}
	In the setting of Theorem~\ref{thm-eps-indep}, 
	let $L_t$ (resp.\  $R_t$) be the LQG length of the counterclockwise (resp.\ clockwise) boundary arc of $\D \backslash \eta'([0,t])$ from $-\bbi$ to $\eta'(t)$. Writing $T_i = \sum_{j < i} t_j$, for $0\leq i \leq N$  define $(L_t^i, R_t^i)  = (L_{T_i + t} - L_{T_i}, R_{T_i + t} - R_{T_i}) $ for $t\in [0,t_i]$, and define $(L_t^{N+1}, R_t^{N+1})= (L_{T_{N+1} + t} - L_{T_{N+1} + t_{N+1}}, R_{T_{N+1} + t} - R_{T_{N+1} + t_{N+1}})$ for $ t \in [0,t_{N+1}]$.
	Then the joint law of $(L_t^0, R_t^0), \dots, (L_t^{N+1}, R_t^{N+1})$ equals, for some constant ${Z_\eps}$,  
	\eqb\label{eq-disk-decomp}
	{Z_\eps^{-1}} \sum_{n \geq 0}1_{E_{\R_+^2}}   \left(\int_{\C} \mu_{\R_+ \times \R}^\eps(0, w_0)\, dw_0\right)  \prod_{i=1}^n \left(\int_\C \mu^\eps(0, w_i)dw_i\right) \mu_{\R_+^2}^\eps \biggl(\sum_{j=0}^n w_j + \bbi, 0\biggr) .
	\eqe
\end{proposition}
\begin{proof}
	We use Proposition~\ref{prop-mot-disk}. Let $t = \mu_h(\D)$. Since $\eta'$ is parametrized by LQG area, it is a measure-preserving transform from $([0,t], \mathrm{Leb}_{[0,t]})$ to $(\D, \mu_h)$, so if $\Lambda\subset \D$ is a Poisson point process with intensity $\eps^{-1}\mu_h$, then $\eta^{-1}(\Lambda)$ is a Poisson point process on $[0,t]$ with intensity $\eps^{-1}\mathrm{Leb}_{[0,t]}$. Restricting to the event that there are at least two cells corresponds to discarding the term $\eps \mu_{\R_+^2}^\eps(i, 0)$ in~\eqref{eq-decomp-bm}. This gives the result with ${Z_\eps}$ the normalization constant making~\eqref{eq-disk-decomp} a probability measure. 
\end{proof}

We now define the functions $P^\eps_\mathrm{init}, P^\eps, P^\eps_\mathrm{end}$, then Proposition~\ref{prop-decomp-quilt-cell} is essentially an immediate consequence of Proposition~\ref{prop-disk-BM-decomp}.
For a Brownian excursion $(L_t, R_t)$ sampled from  $ \int_{\C} \mu^\eps (0,w)\, dw$ whose (random) duration we call $T$, let $P^\eps$ be the density function for 
$(-\inf_{t \in [0,T]} L_t, L_T -\inf_{t \in [0,T]} L_t, -\inf_{t \in [0,T]} R_t, R_T -\inf_{t \in [0,T]} R_t )$. 
For $(L_t, R_t)$ sampled from $\int_{\R_+ \times \R}\mu_{\R_+ \times \R}^\eps(0, w)\, dw$ with random duration $T$, let $P^\eps_\mathrm{init}$ be the density function for $(L_T,-\inf_{[0,t]}R_t, R_T -\inf_{[0,t]}R_t)$.
Let $P^\eps_\mathrm{end}(\ell^-, r^-) = |\mu_{\R_+^2}^\eps(\ell^- + \bbi r^-, 0)|$ for $\ell^-, r^- > 0$.

\begin{proof}[Proof of Proposition~\ref{prop-decomp-quilt-cell}]
	Recall the processes $(L_t^0, R_t^0), \dots, (L_t^{N+1}, R_t^{N+1})$ defined in Proposition~\ref{prop-disk-BM-decomp}. Restrict to the event $\{ N = n\}$. 
	We can describe the lengths $\ell_i^\pm, r_i^\pm$ in terms of these processes: 
	\begin{equation} \label{eq-bm-length}
		\begin{gathered}
			\ell_0^+ = L^0_{t_0}, \quad  r_0^- = -\inf_{t \leq t_0} R^0_t, \quad r_0^+ = R_{t_0}^0 + r_0^-, \quad \ell_{n+1}^ - = L_0^{n+1}, \quad  r_{n+1}^- = R_0^{n+1},\\
			\ell_i^- = -\inf_{t \leq t_i} L^i_t,\quad \ell_i^+ = L^i_{t_i} + \ell_i^-, \quad r_i^- = -\inf_{t \leq t_i} R^i_t, \quad r_i^+ = R^i_{t_i} + r_i^- \qquad \text{ for } i = 1,\dots, n.
		\end{gathered}
	\end{equation}
	Now, since the Brownian motion event $E_{\R_+^2}$ corresponds to the lengths lying in $S_n$, by Proposition~\ref{prop-disk-BM-decomp} the law of the lengths $(\ell_0^+, r_0^-, r_0^+)$, $(\ell_i^-, \ell_i^+, r_i^-, r_i^+)_{i \leq n}$  is~\eqref{lem-Lepsk}.
	
	Next, we study the conditional law of the $\cC_i$ given the lengths. Let $\ell^\pm, r^\pm >0$. 
	Let $\cP_\mathrm{init}^{\ell^+, r^-, r^+}$ be the law of $(L_t, R_t)_{[0,T]}$ sampled from $\int_{\C} \mu_{\R_+ \times \R}^\eps(0, w)\, dw$ and conditioned on $\{\ell^+ = L_{T},\: r^- = -\inf_{t \leq T} R_t, \: r^+ = R_T + r^-\}$. Let $\cP^{\ell^-, \ell^+, r^-, r^+}$ be the law of $(L_t, R_t)_{[0,T]}$ sampled from $\int_\C \mu^\eps(0, w)\,dw$ and conditioned on $\{ \ell^- = -\inf_{t \leq T} L_t,\: \ell^+ = L_{T} + \ell^-, \: r^- = -\inf_{t \leq T} R_t,\: r^+ = R_T + r^-\}$. Let $\cP^{\ell^-, r^-}_\mathrm{end}$ be the probability measure proportional to $\mu_{\R_+^2}^\eps (\ell^- + \bbi r^-, 0)$.
	
	By Proposition~\ref{prop-disk-BM-decomp}, 
	conditioned on $N= n$ and on the lengths $(\ell_0^+, r_0^-, r_0^+)$, $(\ell_i^-, \ell_i^+, r_i^-, r_i^+)_{1 \leq i \leq n}$, $(\ell_{n+1}^-, r_{n+1}^-)$, the processes $(L_t^0, R_t^0), \dots, (L_t^{n+1}, R_t^{n+1})$ have the conditional law $\cP_\mathrm{init}^{\ell_0^+, r_0^-, r_0^+} \times \prod_{i=1}^n \cP^{\ell_n^-, \ell_n^+, r_n^-, r_n^+} \times \cP_\mathrm{end}^{\ell_{n+1}^-, r_{n+1}^-}$. By, e.g., \cite[Lemma 2.17]{ag-disk}, $\cC_i$ is measurable with respect to $(L_t^i, R_t^i)$ for each $i =  0, \dots, n+1$. This completes the proof. 
\end{proof}

\subsection{The planar map structure of the  discretized LQG disk}\label{subsec-quilt-def}
In this section, we introduce \emph{templates} and \emph{quilts} to describe the decorated planar map structure of the discretized LQG disk, see Figure~\ref{fig-template}. The main result of this section is a factorization result for quilts (Lemma~\ref{lem-unif-temp}), which arises from our factorization for sequences of cells (Proposition~\ref{prop-decomp-quilt-cell}).

Consider planar maps with the sphere topology, with a possibly empty subset of faces marked as holes; in what follows we will only use the term ``face'' to refer to non-hole faces. A \emph{template} $T$ is such a planar map such that on each face we mark one or more vertices, one of which is distinguished as the root vertex for the face. Note that if $v$ is a marked vertex for one face, it need not be a marked vertex for other faces. 
We call a face $F$ of $T$ a \emph{$k$-gon} if it has $k$ marked vertices, and we call a collection of edges between consecutive marked vertices of $F$ a \emph{side} of $F$. We denote the edge set of the planar map $T$ by $E_T$. 
A \emph{quilt} is obtained from a template by assigning a positive length to each edge of $E_T$; the set of all quilts with template $T$ can can be identified with $\R_+^{E_T}$. 

\begin{figure}[ht!]
	\begin{center}
		\includegraphics[scale=0.33]{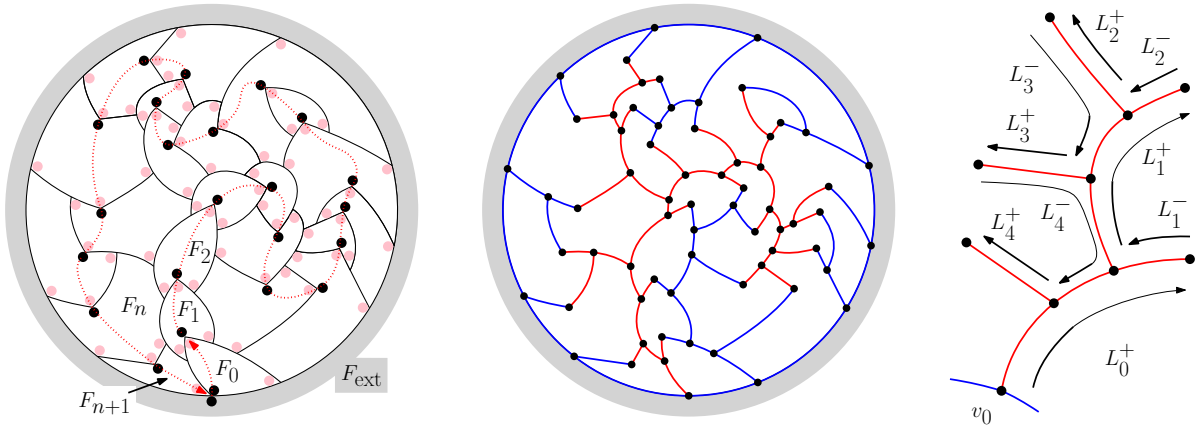}%
		\caption{\label{fig-template} \textbf{Left:} Picture of a template with one 1-gon (shaded), one 2-gon, one 3-gon and $n=21$ 4-gons. Each face has a root vertex (black) and possibly other marked vertices (pink). For $T \in \cT$ we can order the polygons (red curve passing through $F_0, F_1, \dots, F_n, F_{n+1}$). \textbf{Middle:} The template is a planar map with vertices shown in black. Its edge set is called $E_T$.  
			A quilt can be obtained from the template by assigning a positive length to each edge. If $T \in \cT$ we can color the left (resp.\ right) boundary arcs of each face red (resp.\ blue) to get a tree and a cycle-rooted forest. \textbf{Right:} Figure for Lemma~\ref{lem-unif-temp}. Order $E_T^\ell$ (red edges)  by the first time an edge is traversed in a counter-clockwise contour around the tree. Each black arrow corresponds to a left side of a face of $T$. Each arrow traverses one or more edges, the earliest of which is bolded. The key to proving Lemma~\ref{lem-unif-temp} is the observation that each edge is bolded in exactly one arrow.
		}
	\end{center}
\end{figure}

In the setting of Theorem~\ref{thm-eps-indep}, the curve-decorated LQG disk is  split into decorated LQG surfaces $\cC_0, \dots, \cC_{N+1}$ (defined in Section~\ref{subsec-decomp-cells}) parametrized by the regions $D_i$ traced out by the space-filling SLE between times when it hits the Poisson point process. 
We introduce templates and quilts to describe the planar map underlying this discretization, see Figure~\ref{fig-BM} (middle). We will construct a template with faces $F_0, \dots, F_{n+1}$ corresponding to the LQG surfaces $\cC_0, \dots, \cC_{n+1}$ respectively, and a final exterior face $F_\mathrm{ext}$. The marked points of $\cC_i$ correspond to the marked vertices of $F_i$, so the faces are a 2-gon, a 3-gon, and $n$ 4-gons. For a 3-gon, we call the vertex clockwise of the root vertex the terminal vertex, and for a 2-gon or 4-gon we call the vertex opposite the root vertex the terminal vertex. 
The root and terminal vertices of $F_i$ correspond to the starting and ending points of the SLE curve in $\cC_i$.

When $\gamma \in (0, \sqrt2]$, this quilt is obtained as follows. Define the faces $(F_\mathrm{ext}, F_0, \dots, F_{N+1}) = ( \hat \C \backslash \D, D_0, \dots, D_{N+1})$ and the vertex set  $\bigcup_{i=0}^{N+1} \{a_i, b_i, c_i, d_i\}$, where the points $a_i, b_i, c_i, d_i\in \partial D_i$ are as in Theorem~\ref{thm-eps-indep}. The edges are the connected components of $(\partial \D \cup  \bigcup_{i=0}^{N+1} \partial D_i) \backslash (\bigcup_{i=0}^{N+1} \{ a_i, b_i, c_i, d_i\})$, and each is labeled with its LQG length. Each $F_i$ is given the marked vertices $a_i, b_i, c_i, d_i\in \partial D_i$ as defined in Theorem~\ref{thm-eps-indep}, the root vertex of $F_i$ is $a_i$, and $F_\mathrm{ext}$ is a 1-gon with root vertex $a_0$.   See Figure~\ref{fig-template}.

We now give an equivalent iterative definition which applies for all $\gamma \in (0,2)$.  Let $a_0 = b_0 = -\bbi$, $c_0 = 1$ and $d_0 = 0$. Draw radii from $d_0$ to $b_0$ and $c_0$ to cut $\D$ into a quarter-circle $F_0$ and a three-quarter-circle $U_1$. Consider the quilt with three faces: the 1-gon $F_\mathrm{ext} = \hat \C \backslash \D$ with root vertex $a_0$, the 3-gon $F_0$ with root vertex $a_0$ and remaining marked vertices $c_0, d_0$, and the 2-gon $U_1$ with root vertex $d_0$ and other marked vertex $-\bbi$. Assign lengths to the edges such that the side lengths of $\partial F_0$ correspond to the side lengths of $\cC_0$, and the sum of lengths along $\partial \D$ is 1. 
For $j = 1, \dots, N$, we iteratively do the following. Let $\ell_j^-, \ell_j^+, r_j^-, r_j^+$ be the boundary lengths of $\cC_j$ as defined in Section~\ref{subsec-decomp-cells}. Let $a_j = d_{j-1}$. Add a vertex $b_j$ on the clockwise arc of $\partial U_j$ from $a_j$ to $-\bbi$ such that, if the previous vertex on this arc is $v$ and the next vertex is $v'$, the sum of lengths from $a_j$ to $v$ is less than $\ell_j^-$ and the sum of lengths from $a_j$ to $v'$ is greater than $\ell_j^-$. Split the edge between $v$ and $v'$ at $b_j$, with lengths assigned such that the sum of lengths from $a_j$ to $b_j$ is $\ell_j^-$. 
Likewise add a vertex $c_j$ on the counterclockwise arc from $a_j$ to $-\bbi$ such that the sum of lengths from $a_i$ to $c_i$ is $r_j^-$. Finally, mark a bulk point $d_j \in U_j$, add edges from $d_j$ to $b_j, c_j$, and give them lengths $\ell_j^+$ and $r_j^+$ respectively. This splits $U_j$ into two faces; let $U_{j+1}$ be the face having $-\bbi$ on its boundary and $F_j$ the other face. Let $F_j$ have root $a_j$ and remaining marked vertices $b_j, c_j, d_j$, and let $U_{j+1}$ have root vertex $d_j$ and other marked vertex $-\bbi$. After carrying this procedure out for $j = 1, \dots, N$, we are left with a quilt with faces $F_\mathrm{ext}, F_0, \dots, F_N, U_{N+1}$; rename $U_{N+1}$ to $F_{N+1}$ to complete the definition.

\begin{remark}[Relationship with mated-CRT map]
	An equivalent way to define the random quilt arising in the setting of Theorem~\ref{thm-eps-indep} is as follows. Our setting is closely related to that of the mated-CRT map (see \cite{ghs-map-dist}), except we discretize via a Poisson point process rather than at equal intervals (see  Remark~\ref{rem-mated-crt}). We can construct a variant of the mated-CRT map from the discretization in exactly the same way; in this map, each $D_i$ corresponds to a vertex. Take the dual of this mated-CRT map variant, and for each edge between a pair of consecutive faces (i.e.\ faces traced successively by the space-filling curve), split it by adding a degree 2 vertex. Marked vertices and lengths can be assigned to obtain the quilt.
\end{remark}

We now describe the set of templates $\cT$ that can arise from the discretized LQG disk by the above correspondence. See Figure~\ref{fig-template}. 

\begin{definition}\label{def-T}
	Let $\cT$ be the collection of templates $T$ with the following desciption: $T$ has no holes and the face set of $T$ can be called $F_\mathrm{ext}, F_\mathrm{0}, F_1, \dots, F_n, F_{n+1}$ for some $n \geq 0$. The face $F_\mathrm{ext}$ is a 1-gon, $F_0$ is a 3-gon, the faces $F_1, \dots, F_n$ are 4-gons, and $F_{n+1}$ is a 2-gon. Moreover,
	\begin{enumerate}[(a)]
		\item The root vertices of $F_\mathrm{ext}$ and $F_0$ agree, and the side of $F_0$ counterclockwise of the root vertex lies on $\partial F_\mathrm{ext}$. \label{item-T-a}
		\item For $i = 0, 1, \dots, n$ the terminal vertex of $F_i$ is the root vertex of $F_{i+1}$, and the terminal vertex of $F_{n+1}$ is the root vertex of $F_0$.  (This corresponds to the endpoint of the SLE curve in $\cC_i$ agreeing with the starting point of the curve in $\cC_{i+1}$.)
		\label{item-T-b}
		\item For $i = 1, \dots, n+1$ the root vertex of $F_i$ has degree 2, and all other marked vertices of $F_i$ have degree 3. (This corresponds to the non-consecutive $\cC_i$ intersecting generically almost surely, see e.g.\ \cite[Figure 1]{gms-harmonic}.)  \label{item-T-c}
		\item For $i = 1, \dots, n$, the two sides of $F_i$ adjacent to its root vertex both lie on $\partial (F_\mathrm{ext} \cup \bigcup_{j=0}^{i-1} F_j)$, while the other two sides of $F_i$ only intersect $\partial (F_\mathrm{ext} \cup \bigcup_{j=0}^{i-1} F_j)$ at endpoints. (This corresponds to the definition of the marked points of $\cC_i$.) \label{item-T-d}
	\end{enumerate}
	For $T\in \cT$ let $\cQ_T$ be the set of all quilts with template $T$.
\end{definition}
By the correspondences described in Definition~\ref{def-T} between~\eqref{item-T-b}-\eqref{item-T-d} and the properties of $(\cC_0, \dots, \cC_{N+1})$, 
there is a.s.\ a unique template $T \in \cT$ describing the discretized LQG disk via the correspondence stated right above Definition~\ref{def-T}. Thus, there is a.s.\ a unique quilt $Q \in \cQ_T$ such that the side lengths of $F_i$ agree with those of $\cC_i$ for $0\leq i \leq N+1$. Let $\cL^\eps$ denote the law of the random quilt $Q$ arising from the discretization of the LQG disk, so $\cL^\eps$ is a probability measure on $\bigcup_{T\in \cT}\cQ_T$. (This is a slight abuse of notation since  Proposition~\ref{prop-decomp-quilt-cell} defines $\cL^\eps$ to be the law of the side lengths of $\cC_0, \dots, \cC_{N+1}$, but these lengths encode the same information as $Q$.)

For each $T \in \cT$ we denote the restriction of $\cL^\eps$ to $\cQ_T$ by $\cL^\eps|_{\cQ_T}$; this is a finite non-probability measure. We will also need the following notation.

\begin{defn} \label{def-unif} 
Let $\mathrm{Unif}_T$ be the infinite measure on  $\cQ_T$ given by Lebesgue measure on $\R_+^{E_T}$. 
\end{defn}

We give a factorization for quilts with template $T$.

\begin{lemma}\label{lem-unif-temp}
	Fix a template $T \in \cT$ with $n+3$ faces $F_\mathrm{ext}, F_0, \dots, F_{n+1}$. For a quilt $Q \in \cQ_T$, let $C_i$ be the tuple of side lengths of face $F_i$, and define  $w(Q) := Z_\eps^{-1} P^\eps_\mathrm{init}(C_0) \times \prod_{i=1}^n P^\eps(C_i) \times P^\eps_\mathrm{end} (C_{n+1})$ with $Z_\eps,P^\eps_\mathrm{init}, P^\eps, P^\eps_\mathrm{end}$ as in~\eqref{lem-Lepsk}. Then
	\[\cL^\eps|_{\cQ_T}(dQ) =  w(Q)  \mathrm{Unif}_T(dQ).\] 
\end{lemma}
\begin{proof}
	Recall $S_n$ from~\eqref{lem-Lepsk} which parametrizes the space of quilts from $\cL^\eps$ with $n+3$ faces via the lengths $(\ell_0^+, r_0^-, r_0^+), (\ell_i^-, \ell_i^+, r_i^-, r_i^+)_{1\leq i \leq n}$. Let $S_T \subset S_n$ be the set of lengths such that the resulting quilt has template $T$. Recall that $E_T$ is the edge set of $T$ and that a quilt with template $T$ is identified with its collection of edge lengths $(x_e)_{e \in E_T} \in \R_+^{E_T}$. Consider the affine bijection $f:\R_+^{E_T} \to S_T$ sending $(x_e)_{e \in E_T} \mapsto (\ell_0^+, r_0^-, r_0^+), (\ell_i^-, \ell_i^+, r_i^-, r_i^+)_{i \leq n}$. We will give a combinatorial argument that shows the Jacobian matrix of $f$ has unit determinant, so $f^{-1}_*\mathrm{Leb}_{S_T}=\mathrm{Unif}_T$; combining with~\eqref{lem-Lepsk} then gives the desired result. 
	
	For $1 \leq i \leq n$, for the 4-gon $F_i$ let $L_i^-$ be the side clockwise of the root vertex and $L_i^+$ the side clockwise of $L_i^-$. 
	(Under the correspondence described in Definition~\ref{def-T}, $L_i^-$ and $L_i^+$ correspond to the lower left and upper left marked boundary arcs of the $i$th space-filling SLE segment, respectively.)
	Let $L_0^+$ be the side of $F_{0}$ clockwise of its root vertex. For $(i, \pm) \in \{ 1, \dots, n\} \times \{+,-\}$  and $(i, \pm) = (0, +)$, let $E_i^\pm \subset E_T$ be the set of edges that lie on the side $L_i^\pm$. Let $E_T^\ell = \bigcup_{i, \pm} E_i^\pm$. See  Figure~\ref{fig-template} (right). 
	By definition $\ell_i^\pm = \sum_{e\in E_i^\pm} x_e$;  this gives the affine bijection between $(\ell_0^+, \ell_1^-, \ell_1^+, \dots, \ell_n^-, \ell_n^+)$ and $(x_e)_{e \in E_T^\ell}$.
	We note for later use that this bijection implies $|E_T^\ell| = 2n+1$. 
	
	The edges $E_T^\ell$ form a tree rooted at the root vertex $v_0$ of $F_0$. (This tree corresponds to the tree of flow lines forming the left boundaries of the mated-CRT map cells.) 
	Consider the counterclockwise contour around this tree starting at $v_0$, i.e., the depth-first exploration starting at $v_0$. This contour traverses each edge of $E_T^\ell$ twice, once in each of the two possible directions. Furthermore, this contour passes through the faces in increasing order, see Figure~\ref{fig-template} (left) and (middle). This follows from property~\eqref{item-T-b} of Definition~\ref{def-T}. 
	Thus, it traverses the sides $L_0^+, L_1^-, L_1^+, \dots, L_n^+$ in order (tracing the edges on each $+$/$-$ side in the direction of increasing/decreasing distance from $v_0$  in the tree), then traces the left side of $F_{n+1}$ to return to $v_0$, see Figure~\ref{fig-template} (right). We assign an ordering $\prec$ to $E_T^\ell$ by the time an edge is \emph{first} visited by the contour. 
	
	We claim that for each $e \in E_T^\ell$, there is some $(i, \pm)$ such that $e \in E_i^\pm$ and  $e$ is the $\prec$-earliest edge in $E_i^\pm$, see Figure~\ref{fig-template} (right). 
	Let $j$ be the index such that $e \in E_j^+$. If $e$ is the $\prec$-earliest edge of $E_j^+$ (i.e.\ the edge of $L_j^+$ closest to $v_0$) we are done. Otherwise, let $v$ be the endpoint of $e$ closer to $v_0$. By~\eqref{item-T-c}, every internal vertex of the tree has degree 3, and, in our counterclockwise traversal, after tracing $L_j^+$  we have only visited two of the edges adjacent to $v$. Let $k > j$ be the index such that the side $L_k^+$ is traversed starting from $v$. Then $v$ is the endpoint of $L_k^-$, so $e$ is the $\prec$-earliest edge of $E_k^-$.
	Thus the claim holds.   Further, since $|E_T^\ell| = 2n+1=|\{L_0^+, L_1^-, \dots, L_n^+\}|$, each $e$ is the $\prec$-earliest edge of exactly one $L_i^\pm$.

	Consequently, if we order $E_T^\ell$ by $\prec$ and order the sides $L_i^\pm$ by the $\prec$-earliest edge they contain, the zero-one matrix sending $(x_e)_{e \in E_T^\ell}$ to $(\ell_0^+, \ell_1^-, \ell_1^+, \dots, \ell_n^-, \ell_n^+)$ is upper triangular with ones on the diagonal, and so has determinant $1$. 
	A similar argument shows the zero-one matrix 
	sending $(x_e)_{e \in E_T \backslash E_T^\ell}$ to $(r_0^-, r_0^+, r_1^-, r_1^+, \dots, r_n^-, r_n^+)$ has unit determinant. Thus $f$ has unit determinant as needed. 	
\end{proof}

\subsection{Independence of complementary subquilts}\label{subsec-indep-quilts}
Recall  $\cL^\eps$ is the probability measure on quilts arising from  discretized LQG disks (see Lemma~\ref{lem-unif-temp}). The goal of this section is to show Proposition~\ref{prop-indep-subquilts}: informally, a quilt sampled from $\cL^\eps$ conditioned on having a specific subquilt has conditionally independent  complementary subquilts. The proof uses our earlier factorization for quilts with a specified template (Lemma~\ref{lem-unif-temp}), a combinatorial factorization for the set of  templates having a fixed subtemplate (Proposition~\ref{prop-product-templates}) whose proof we defer to Section~\ref{subsec-topo}, and a factorization of Lebesgue measures (Lemma~\ref{lem-leb-fac}).

Let $T^*$ be a template in $\cT$ with a connected subset of faces marked. The marked faces give rise to another template as follows: say two faces are adjacent if their boundaries share an edge, replace each connected component of the non-marked faces with a hole, then delete every vertex except marked vertices of the faces and vertices adjacent to multiple faces. We call this the marked \emph{subtemplate} of $T^*$. Each edge of the marked subtemplate corresponds to the union of one or more edges of the quilt (since deleting vertices amounts to merging edges). 

Given a quilt $Q \in \bigcup_{T \in \cT} \cQ_T$, if we mark a connected subset of its faces, the marked \emph{subquilt} is the quilt whose template is the marked subtemplate and whose lengths are inherited from $Q$. We say that a quilt $M$ is a subquilt of $Q$ if one can mark a subset of faces of $Q$ so that the marked subquilt agrees with $M$. See Figure~\ref{fig-M} (a),(b).

Fix a connected template  $T_\mathrm{sub}$ with $b \geq 1$ holes labelled $1, \dots, b$, whose faces are all 4-gons except for a 1-gon, 2-gon and 3-gon. Suppose further that $T_\mathrm{sub}$ can be realized as a marked subtemplate of a template in $\cT$, see Figure~\ref{fig-M} (a)-(c), so in particular its holes share no common boundary edges. Let $\cQ_{T_\mathrm{sub}}$ be the set of quilts having template $T_\mathrm{sub}$ and whose 1-gon has length 1. The goal of this section is the following. 

\begin{proposition}\label{prop-indep-subquilts}
	Recall  the probability measure $\cL^\eps$ on quilts arising from  discretized LQG disks in Lemma~\ref{lem-unif-temp}.
	For fixed $T_\mathrm{sub}$, the following holds for all  $M \in \cQ_{T_\mathrm{sub}}$. Sample a quilt $Q$ from $\cL^\eps$ conditioned on $\{ Q \text{ has } M \text{ as a subquilt}\}$, and let $Q_i$ be the subquilt of $Q$ obtained by marking the faces of $Q$ in the $i$th hole of $M$. Then $Q_1, \dots, Q_b$ are conditionally independent. 
\end{proposition}
Implicitly, the proposition statement uses the fact that one can make sense of conditioning on the measure zero event of having a specific subquilt; this will be shown in the proof. 

Equip $\cQ_{T_\mathrm{sub}}$ with the metric $d(M^1, M^2) = \sum_{e} |x_e^1 - x_e^2|$, where the summation is taken over edges of $T_\mathrm{sub}$ and $x_e^i$ is the length of $e$ in the quilt $M^i$. Let $B_\delta(M^0)$ denote the radius $\delta$ ball with respect to this metric. 

\begin{proposition}\label{prop-indep-subquilts-trunc}
	For fixed $T_\mathrm{sub}$ as above, there is a family of finite measures $\{ m^\eps_M\}_{M \in \cQ_{T_\mathrm{sub}}}$  on $\bigcup_\cT \cQ_T$ such that the following holds. 
	\begin{itemize}
		\item The measure $m_M^\eps$ is supported on the subset of quilts having $M$ as a subquilt. Moreover, for all $M^0 \in \cQ_{T_\mathrm{sub}}$ and $\delta > 0$ sufficiently small in terms of $M^0$, letting $F_{\delta, M^0}$ be the event that a quilt has a subquilt lying in $B_\delta(M^0)$,  
		\eqb\label{eq-decomp-LM}
		\cL^\eps|_{F_{\delta, M^0}} =\int_{B_\delta(M^0)} m^\eps_M \mathrm{Unif}_{T_\mathrm{sub}}(dM).
		\eqe
		\item For a sample from $m_M^\eps/|m_M^\eps|$, the $b$ complementary subquilts of $M$ are independent.
	\end{itemize}
\end{proposition}
In~\eqref{eq-decomp-LM} we want  $\delta$ sufficiently small to avoid the situation that $Q$ has multiple subquilts lying in $B_\delta(M^0)$. If instead $\delta \in (0, \infty]$ is arbitrary, the left hand side of~\eqref{eq-decomp-LM} needs to be weighted by the number of subquilts lying in $B_\delta(M^0)$.

The desired Proposition~\ref{prop-indep-subquilts} is an immediate consequence. 
\begin{proof}[Proof of Proposition~\ref{prop-indep-subquilts}]
	The existence of the conditional law of $Q$ given $\{ Q \text{ has }M \text{ as a subquilt}\}$ follows from the first claim of Proposition~\ref{prop-indep-subquilts-trunc}, and the conditional independence of the $Q_i$ follows from the second claim. 
\end{proof}

We now turn to the proof of Proposition~\ref{prop-indep-subquilts-trunc}. The key ingredients are the factorization of quilts of a given template (Lemma~\ref{lem-unif-temp}),  a template factorization result (Proposition~\ref{prop-product-templates}), and a Lebesgue measure factorization statement (Lemma~\ref{lem-leb-fac}). 

Let $\cT^*$ be the collection of all $T^*$ having $T_\mathrm{sub}$ as its marked subtemplate. That is, $\cT^*$ contains all templates of $\cT$ obtained by filling in the holes of $T_\mathrm{sub}$ and marking the faces corresponding to $T_\mathrm{sub}$. Recall the holes of $T_\mathrm{sub}$ are labelled $1, \dots, b$. For $T^* \in \cT^*$ 
and $i \leq b$, let $T_i^*$ be the marked template with $b-1$ holes obtained from $T^*$ by replacing the $j$th connected component of the unmarked faces with a hole for $j \neq i$, see Figure~\ref{fig-M} (c),(d).  Let $\cT^*_i$ be the collection of all possible $T_i^*$ as $T^*$ varies. 

\begin{proposition}\label{prop-product-templates}
	The map $\cT^* \to \prod_{i=1}^b \cT^*_i$ sending $T^* \mapsto (T^*_1, \dots, T^*_b)$ is a bijection.  
\end{proposition}
Proposition~\ref{prop-product-templates} is an independence statement about templates: no matter how we fill the holes indexed by  $\{1, \dots, b\} \backslash \{i\}$, the choices for how to fill the $i$th hole are the same, namely $\cT_i^*$. See Figure~\ref{fig-M} (b),(c).  We defer the proof of Proposition~\ref{prop-product-templates} to Section~\ref{subsec-topo}.

\begin{figure}[ht!]
	\begin{center}
		\includegraphics[scale=0.35]{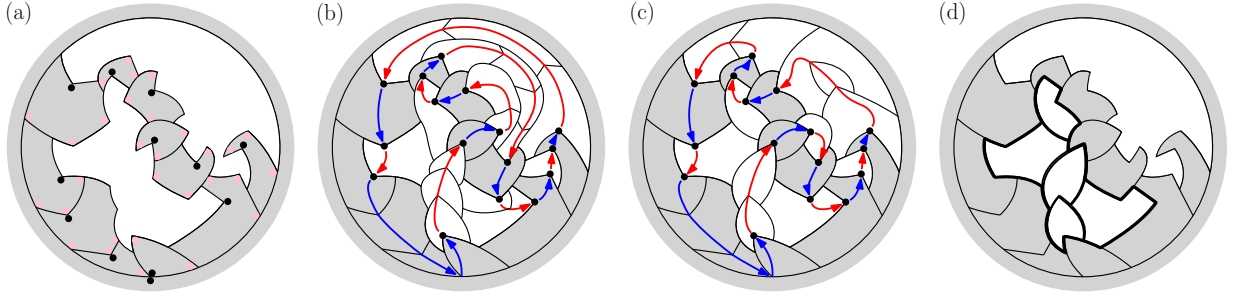}%
		\caption{\label{fig-M} \textbf{(a)} A template $T_\mathrm{sub}$. The white regions are holes of $T_\mathrm{sub}$. \textbf{(b), (c)} Two different $T^* \in \cT^*$ (vertices of polygons not shown). The marked gray region correspond to $T_\mathrm{sub}$, and the arrows depict the orderings on faces of $T^*$. Notice that the blue arrows are traversed in a different order in the two pictures. Perhaps surprisingly, regardless of how one hole is filled, the choices for filling the other hole are the same (Proposition~\ref{prop-product-templates}). \textbf{(d)} For the marked template $T^*$ depicted in (c), the marked template $T_i^*$ is obtained by replacing all but the $i$th connected component of non-marked faces (bolded) with holes. In the figure $T_i^*$ is the template with $b-1$ holes (white, not bolded). 			
			In Lemma~\ref{lem-leb-fac}  $J_{T_i^*} \subset E_{T_i^*}$ is the set of bolded edges, and the subquilt $Q_i$ (bolded) can be identified with a tuple of edge lengths (in $\R_+^{J_{T_i^*}}$).
		}
	\end{center}
\end{figure}

Recall that $\mathrm{Unif}_{T^*}$ is the measure on marked quilts having template $T^*$ given by Lebesgue measure on $\R_+^{E_{T^*}}$. For $Q$ sampled from $\mathrm{Unif}_{T^*}$, let $\wt{\mathrm{Unif}}_{T^*}$ be the  law of $(M, Q_1, \dots, Q_b)$ where $M$ and $Q_1, \dots, Q_b$ are the marked subquilt and complementary quilts of $Q$. For $T_i^* \in \cT_i^*$, let $J_{T_i^*} \subset E_{T_i^*}$ be the edges on the boundaries of the unmarked faces, see Figure~\ref{fig-M} (d). For any marked quilt $Q \in \cQ_{T^*}$, the subquilt $Q_i$ is specified by the lengths of the edges in $J_{T_i^*}$, so we may identify $Q_i$ with an element of $\R_+^{J_{T_i^*}}$. 

We now give a factorization for uniform measures on templates, which follows from the product structure of Lebesgue measure on Euclidean space, and the fact that the holes of $T_\mathrm{sub}$ do not share any common boundary edges. It is essentially just a linear change of variables. 
\begin{lemma}\label{lem-leb-fac}
	There is a collection $\{\mathrm{Unif}^M_{J_{T_i^*}}\}$ of measures on $\R_+^{J_{T_i^*}}$ for $M \in \cQ_{T_\mathrm{sub}}$, $1 \leq i \leq b$ and $T_i^* \in \cT_i^*$ such that 
	\[\wt{\mathrm{Unif}}_{T^*}(dM, dQ_1,\dots, dQ_b) = \left(\prod_{i=1}^b \mathrm{Unif}^{M}_{J_{T_i^*}}(dQ_i) \right) \mathrm{Unif}_{E_{T_\mathrm{sub}}}(dM) \qquad \text{ for all }T^* \in \cT^*,\] 
	where $T_1^*, \dots, T_b^*$ are defined via $T^*$ as in Proposition~\ref{prop-product-templates}. 
\end{lemma}
\begin{proof}
	Consider the disjoint union $E_{T_\mathrm{sub}} = E_0 \cup \bigcup_{i=1}^b E_i$ where $E_i$ is the set of edges lying on the boundary of the $i$th hole of $T_\mathrm{sub}$ and $E_0$ the set of edges not lying on the boundary of any hole. By the identification $\cQ_{T_\mathrm{sub}} \leftrightarrow \R_+^{E_{T_\mathrm{sub}}}$ we can identify $M$ in $\cQ_{T_\mathrm{sub}}$ with a tuple $(M_0, M_1, \dots, M_b)$ in $\R_+^{E_0} \times \prod_{i=1}^b \R_+^{E_i}$. Likewise identify a quilt in $\cQ_{T^*}$ with a tuple $(M_0, Q_1, \dots, Q_b)$ in $\R_+^{E_0} \times \prod_{i=1}^b \R_+^{J_{T_i^*}}$.  
	Let $\mathrm{Unif}_{E_i}$ (resp.\ $\mathrm{Unif}_{J_{T_i^*}}$) denote Lebesgue measure on $\R_+^{E_i}$ (resp.\ $\R_+^{J_{T_i^*}}$), then 
	\eqb\label{eq-leb-fac}
	\mathrm{Unif}_{T^*}(dM_0, dQ_1, \dots, dQ_b)  = \mathrm{Unif}_{E_0}(dM_0) \times \prod_{i=1}^b \mathrm{Unif}_{J_{T_i^*}}(dQ_i).  
	\eqe
	Since $T_\mathrm{sub}$ is the marked subtemplate of $T^*$, each edge $e\in E_i$ is a concatenation of one or more edges $e' \in J_{T_i^*}$.
	Define $M_i$ as a function of $Q_i$ by assigning each edge $e \in E_i$ the sum of lengths of the corresponding $e' \in J_{T_i^*}$. Push forward the 
	measures in~\eqref{eq-leb-fac} under the map $(M_0, Q_1, \dots, Q_b) \mapsto (M_0, Q_1, M_1, \dots, Q_b, M_b)$ and denote the new measures with tildes:
	\eqb\label{eq-leb-fac2}
	\widetilde{\mathrm{Unif}}_{T^*}(dM, dQ_1, \dots, dQ_b)  = \mathrm{Unif}_{E_0}(dM_0) \times \prod_{i=1}^b \widetilde{\mathrm{Unif}}_{J_{T_i^*}}(dQ_i, dM_i).  
	\eqe
	Finally, we can disintegrate each $\wt{\mathrm{Unif}}_{J_{T_i^*}}$ to obtain a family of measures\footnote{The general theory of disintegrations only specifies this family up to measure zero sets of $M_i$, but we are just applying a linear change of variables to a Lebesgue measure, so $\mathrm{Unif}_{J_{T_i^*}}^{M_i}$ is well defined for every $M_i$.} $\{\mathrm{Unif}_{J_{T_i^*}}^{M_i}\}$ for $M_i \in \R_+^{E_i}$ such that $\wt{\mathrm{Unif}}_{J_{T_i^*}}(dQ_i, dM_i) = \mathrm{Unif}_{J_{T_i^*}}^{M_i}(dQ_i) \mathrm{Unif}_{E_i}(dM_i)$. Since $M_i$ is a function of $M$ we may define  $\mathrm{Unif}_{J_{T_i^*}}^{M} = \mathrm{Unif}_{J_{T_i^*}}^{M_i}$, then~\eqref{eq-leb-fac2} equals $\left(\prod_{i=1}^b \mathrm{Unif}_{J_{T_i^*}}^{M}(dQ_i)\right)  \mathrm{Unif}_{E_0}(dM_0) \prod_{i=1}^b \mathrm{Unif}_{E_i}(dM_i)$. The $dM_0$ and $dM_i$ terms factorize to give $\mathrm{Unif}_{E_{T_\mathrm{sub}}}(dM)$, as desired. 
\end{proof}

\begin{proof}[Proof of Proposition~\ref{prop-indep-subquilts-trunc}]
	We will derive an expression for the law of $\cL^\eps|_{F_{\delta, M^0}}$, then write down a choice of $m_M^\eps$ for which~\eqref{eq-decomp-LM} holds, and which satisfies the remaining conditions of finiteness and independent subquilts. 
	
	Fix $M^0 \in \cQ_{T_\mathrm{sub}}$. By definition, the quilts $Q \in \bigcup_{T \in \cT} \cQ_T$ satisfying $F_{\delta, M^0}$ are those which have a subquilt in $B_\delta(M^0)$; this subquilt is a.s.\ unique when $\delta > 0$ is sufficiently small since edge lengths are a.s.\ distinct.
	Thus, viewing a sample of $\cL^\eps|_{F_{\delta, M^0}}$ as being a \emph{marked} quilt where the subquilt in $B_\delta(M^0)$ is marked, we have 
	\[ \cL^\eps|_{F_{\delta, M^0}} = \sum_{T^*\in \cT^*} \cL^\eps|_{\cQ_{T^*} \cap F_{\delta, M^0}}.\]
	Identify a marked quilt $Q$ with the tuple $(M, Q_1, \dots, Q_b)$ where $M$ is the marked subquilt and $Q_i$ are the complementary subquilts of $M$. 	Write $F_{\delta, M^0}^*$ for the event that the marked subquilt lies in $B_\delta(M^0)$. 
	Lemma~\ref{lem-unif-temp} gives $\cL^\eps|_{\cQ_{T^*}\cap F_{\delta, M^0}} = 1_{F_{\delta, M^0}^*} w(Q) \mathrm{Unif}_{T^*}(dQ)$ where $w(Q) = Z_\eps^{-1} P^\eps_\mathrm{init}(C_0) \times \prod_{i=1}^n P^\eps(C_i) \times P^\eps_\mathrm{end} (C_{n+1})$ and $C_0, \dots, C_{n+1}$ are the tuples of side lengths of the faces of $Q$. Since $M, Q_1, \dots, Q_b$ partition the faces of $Q$, we can group the terms to get $w(Q) = Z_\eps^{-1} \wt w(M) \prod_{i=1}^b \wt  w(Q_i)$.
	Applying Lemma~\ref{lem-leb-fac} thus gives 
	\[\cL^\eps|_{\cQ_{T^*} \cap F_{\delta, M^0}}(dQ) =  1_{F_{\delta, M^0}^*}w(Q)  \mathrm{Unif}_{T^*}(dQ) =  1_{F_{\delta, M_0}^*} Z_\eps^{-1} \left(\prod_{i=1}^b \wt w(Q_i)\mathrm{Unif}^{M}_{J_{T_i^*}}(dQ_i) \right) \wt w(M) \mathrm{Unif}_{E_{T_\mathrm{sub}}}(dM) . \]
	By Proposition~\ref{prop-product-templates},
	\[ \cL^\eps|_{F_{\delta, M^0}} = \sum_{T^* \in \cT^*} \cL^\eps|_{\cQ_{T^*} \cap F_{\delta, M^0}}(dQ) = 1_{F_{\delta, M^0}^*} Z_\eps^{-1} \left(\prod_{i=1}^b \left(\sum_{T_i^* \in \cT_i^*} \wt w(Q_i)\mathrm{Unif}^{M}_{J_{T_i^*}}(dQ_i) \right)\right) \wt w(M) \mathrm{Unif}_{E_{T_\mathrm{sub}}}(dM) . \]
	Define for each $M \in \cQ_{T_\mathrm{sub}}$ the measure
	\eqb\label{eq-Meps}
	m^\eps_M := Z_\eps^{-1} \wt w(M) \prod_{i=1}^b \left( \sum_{T_i^* \in \cT_i^*} \wt w(Q_i)\mathrm{Unif}^{M}_{J_{T_i^*}}(dQ_i)\right),
	\eqe
	and view $m^\eps_M$ as a measure on the set of quilts having $M$ as a subquilt. The above gives~\eqref{eq-decomp-LM}. Finiteness  of the measures follows from~\eqref{eq-decomp-LM} since $|\cL^\eps| = 1 < \infty$, and independence of the complementary subquilts of a sample from $m^\eps_M/|m^\eps_M|$ follows from the product structure of~\eqref{eq-Meps}. 
\end{proof}

\subsection{Combinatorial independence of complementary subtemplates}\label{subsec-topo}
The goal of this section is to prove Proposition~\ref{prop-product-templates}. The arguments are topological in nature and essentially depend on Hopf's Umlaufsatz: the total curvature/winding in a simple loop is $\pm 2\pi$ (Lemma~\ref{lem-hopf}).  See Section~\ref{subsec-winding} for details on total curvature/winding. The reader may want to read Appendix~\ref{sec-meander} before reading this section to motivate the arguments involved.

Throughout this section we work in the setting of Proposition~\ref{prop-product-templates} and use the notation of that proposition. 
It is notationally convenient for us to slightly modify $\cT^*$ and $T_\mathrm{sub}$ by replacing the 1-gon face $F_\mathrm{ext}$ with a hole $H_\mathrm{ext}$.

The proof makes repeated use of the following observation. We say two faces of a template are \emph{adjacent} if their boundaries intersect on an edge. 

\begin{lemma}\label{lem-adj-order}
	Suppose $T \in \cT$ has external hole $H_\mathrm{ext}$ and faces $F_0, \dots, F_{n+1}$. For each 4-gon $F_i$, let $\partial_- F_i$ be the union of the two sides adjacent to the root vertex of $F_i$. Let $\partial_- F_0$ be the side counterclockwise of the root vertex of $F_0$, and let $\partial_- F_{n+1} = \partial F_{n+1}$. Let $\partial_+ F_i = \partial F_i \backslash \partial_- F_i$ for all $i$.  
	\\ If $F_i$ and $F_j$ are adjacent, then $i < j$ if and only if $\partial_+F_i$ intersects $\partial_-F_j$ on an edge. 
\end{lemma}
\begin{proof}
	This is immediate from~\eqref{item-T-d} in Definition~\ref{def-T}. 
\end{proof}

Recall the holes of $T_\mathrm{sub}$ are called $H_\mathrm{ext}, H_1, \dots, H_b$. 
We embed $T_\mathrm{sub}$ in the Riemann sphere $\hat \C = \C \cup \{\infty\}$ such that $H_\mathrm{ext}$ contains $\infty$. Let $x\in \partial H_\mathrm{ext}$ be the root vertex of the 3-gon. For $i = 1, \dots, b$ let $V_i \subset \partial H_i$ be the set of all points on the $i$th boundary loop which are a root vertex or a terminal vertex of some face. 

In each face $F$ of $T_\mathrm{sub}$ we draw a smooth curve $\eta_F$ from its root vertex to its terminal vertex; we call $\eta_F$ the curve of $F$. We do this in a way such that if the curves of two faces meet at an endpoint, they can be smoothly concatenated. This gives us an embedded directed graph $G$ in $\hat \C \backslash (H_\mathrm{ext}\cup \bigcup_{i=1}^b H_i)$ whose vertex set is $\{x\}\cup \bigcup_{i=1}^b V_i$ and whose edges are the maximal concatenations of the curves, see Figure~\ref{fig-winding} (left) for an embedding of $G$ for the $T_\mathrm{sub}$ depicted in Figure~\ref{fig-M} (a).

Suppose $T^* \in \cT^*$, so $T_\mathrm{sub}$ is the marked subtemplate of $T^*$. Embed $T^*$ in $\hat \C$ such that $T_\mathrm{sub}$ has the same embedding as before; this corresponds to drawing faces to fill the holes $H_1, \dots, H_b$. We draw curves in each added face from its root vertex to its terminal vertex, and so link up the previous smooth curves to get a single simple smooth curve from $x$ to $x$. This curve visits the faces of $T_\mathrm{sub}$ in the order induced by the face ordering of $T^*$. Recall the notion of \emph{total curvature} from Section~\ref{subsec-winding}. Let $\theta(x) \in [0,2\pi)$ be the angle such that the initial tangent to the curve at $x$ is parallel to $e^{\bbi \theta(x)}$, and for each point $p$ on the curve, let $\theta(p)$ equal $\theta(x)$ plus the total curvature of the curve from $x$ to $p$.
We label each $v \in \bigcup_i V_i$ with the value $\theta(v)$.
See Figure~\ref{fig-winding} (right).

We will show in Lemma~\ref{lem-winding} that the winding on $\bigcup V_i$ does not depend on $T^*$. We first need the following lemma. 
For each face $F$, let $\partial_\ell F \subset \partial F$ (resp.\ $\partial_r F \subset \partial F$) be the clockwise (resp.\ counterclockwise) boundary arc of $F$ from its root vertex to its terminal vertex. 
\begin{lemma}\label{lem-adj-winding}
	Let $F$ and $F'$ be faces of $T_\mathrm{sub}$, and let the root and terminal vertices of $F$ (resp.\ $F'$) be $v_-$ and $v_+$ (resp.\ $v_-'$ and $v_+'$). 
	Suppose $F$ and $F'$ are adjacent (have a common boundary edge),  $F$ preceeds $F'$ in the ordering of faces, and $F$ and $F'$ are not consecutive faces ($v_+ \neq v_-'$). Suppose $\partial_r F$ intersects $\partial_r F'$ along an edge and let $\eta_{FF'}$ be the clockwise boundary arc of $\partial (F \cup F')$ from $v_+$ to $v_-'$. Then the curve obtained by concatenating $\eta_F$, $\eta_{FF'}$ and $\eta_{F'}$ has total curvature $\theta(v_+') - \theta(v_-)$, in the sense of Definition~\ref{def-winding-nonsmooth}. See Figure~\ref{fig-adjacency} (left). 
	
	The statement also holds if we replace $\partial_r$ with $\partial_\ell$ and clockwise with counterclockwise.  
\end{lemma} 
\begin{proof}
	We address the first claim, the other case follows the same argument. 
Let $\eta_1$ (resp.\ $\eta_2$) be the concatenations of the curves of all faces of $T^*$ preceding $F$ (resp.\ between $F$ and $F'$). Then $\eta_1$ is a curve from $x$ to $v_-$, and $\eta_2$ is a curve from $v_+$ to $v_-'$.  See Figure~\ref{fig-adjacency} (left). 
	The connected components of $\C \backslash \ol{(\eta_1 \cup F \cup F' \cup H_\mathrm{ext})}$ are simply connected. Let $D$ be the component containing $\eta_2$, then $\ol D$ also contains $\eta_{FF'}$. Since $\ol D$ is simply connected, $\eta_2$ and $\eta_{FF'}$ are homotopic in $\C \backslash (\eta_F\cup \eta_{F'})$, so by Definition~\ref{def-winding-nonsmooth} the total curvature of the concatenation of $\eta_F, \eta_{FF'}$ and $\eta_{F'}$ agrees with that of the concatenation of $\eta_F, \eta_2, \eta_{F'}$, which by definition is $\theta(v_+') - \theta(v_-)$. 
\end{proof}

\begin{lemma}\label{lem-winding}
	The function $\theta(v)$ on $\bigcup_i V_i$ defined above depends on the embedding of $T_\mathrm{sub}$ and the curves of the faces of $T_\mathrm{sub}$, but does not depend on the choice of $T^* \in  \cT^*$ or the embedding of $T^*$. 
\end{lemma}
\begin{proof}
	We claim that if $F$ and $F'$ are adjacent  faces of $T_\mathrm{sub}$ (i.e.\ have a common boundary edge), then the values of $\theta$ on the curve  of $F$ determine its values on the curve of $F'$, and vice versa. Each $T^* \in \cT^*$ induces an ordering on the faces of $T_\mathrm{sub}$, and Lemma~\ref{lem-adj-order} implies the relative order of $F$ and $F'$ does not depend on $T^*$ (since $F$ and $F'$ are adjacent); without loss of generality $F$ precedes $F'$. If the curve passes directly from $F$ to $F'$, then the claim is immediate. Otherwise, by Lemma~\ref{lem-adj-winding} $\theta(v_+') - \theta(v_-)$ equals the total curvature of a curve that does not depend on $T^*$, so the claim also holds in this case.

	The set of faces of $T_\mathrm{sub}$ is connected under the adjacency relation, and the values of $\theta$ on the curve segment in the 3-gon are known by definition, so repeatedly using the claim gives the result. 
\end{proof}

\begin{figure}[ht!]
	\begin{center}
		\includegraphics[scale=0.37]{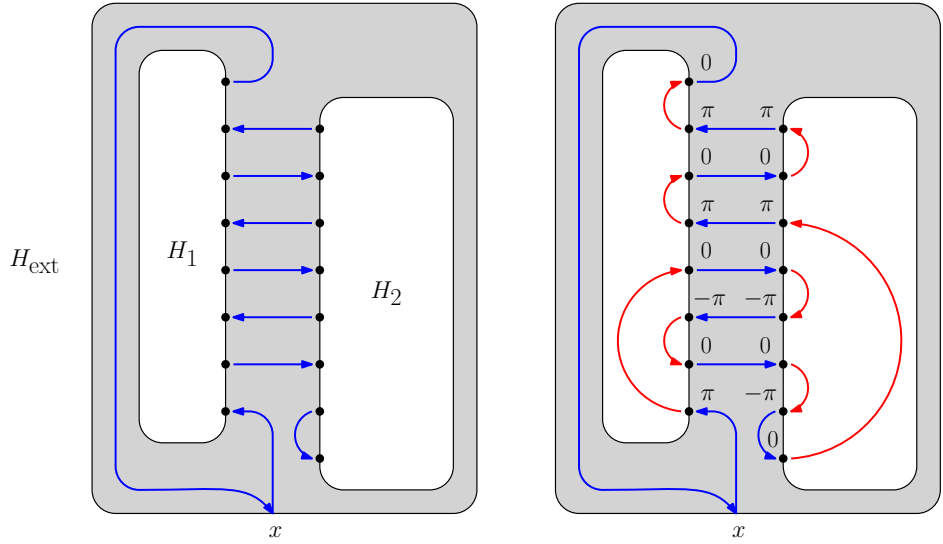}%
		\caption{\label{fig-winding} \textbf{Left:} We represent $T_\mathrm{sub}$ in Figure~\ref{fig-M} (a) by a graph embedded in the plane. Each edge is a smooth curve (blue arrow) corresponding to a sequence of faces of $T_\mathrm{sub}$ (blue arrows in Figure~\ref{fig-M} (b), (c)).  \textbf{Right:} We add red edges according to the template $T^*$ shown in  Figure~\ref{fig-M} (c).  We define $\theta(v)$ to be the total curvature of the concatenated curve from $x$ to $v$. This function $\theta$ (labelled values) turns out not to depend on the choice of $T^*$. }
	\end{center}
\end{figure}

In our embedded directed graph $G$, the degree of each vertex in $\bigcup_{i=1}^b V_i$ is one, each curve lies in $\hat \C \backslash (H_\mathrm{ext} \cup \bigcup_{i=1}^b H_i)$, and the in-degree and out-degree of $x$ is one. 
Let $A_i$ be the finite collection of ways of adding directed edges between vertices in $V_i$ such that the edges lie in $H_i$, the resulting graph is still planar, the indegree and outdegree of each vertex in $V_i$ is one, and the following compatibility relation holds: any added edge $e$ from $v$ to $w$, drawn as a simple smooth curve in $H_i$ which joins smoothly with the edges of $G$ at $v$ and $w$, has total curvature equal to  $\theta(w)-\theta(v)$. This compatibility condition makes sense since $\theta(u), \theta(v)$ are well defined (Lemma~\ref{lem-winding}), and every embedding of $e$ as a simple smooth curve in $H_i$ that joins smoothly with edges of $G$ has the same total curvature (Lemma~\ref{lem-total-winding}). See Figure~\ref{fig-winding} (right) where for each hole $H_i$ the collection of red arrows represents an element of $A_i$. 

\begin{lemma}\label{lem-Ti-a}
	Suppose $T_i^* \in \cT_i^*$. In each unmarked face of $T_i^*$ draw a curve from its root vertex to its terminal vertex. Concatenating these gives a collection of curves from $V_i$ to $V_i$; view each curve as an edge from $V_i$ to $V_i$ and let  $\alpha_i$ be the set of these edges.  Then $\alpha_i \in A_i$. 
\end{lemma}
\begin{proof}
	This follows from Lemma~\ref{lem-winding} since $T_i^*$ is a subtemplate of some marked template in $\cT^*$.
\end{proof}

\begin{lemma}\label{lem-hamiltonian}
	For any $(\alpha_1, \dots, \alpha_n) \in \prod_{i=1}^n A_i$, adding the edges of each of the $\alpha_i$ to $G$ gives a Hamiltonian cycle, i.e.\ a cycle that visits every vertex. 
\end{lemma}
\begin{proof}
	In the new graph every vertex has indegree and outdegree one. Thus the edges form a collection of cycles. Suppose for the sake of contradiction that there is a cycle not containing $x$. Since for any edge from $w$ to $v$ the total curvature is $\theta(v)-\theta(w)$, the total curvature of the cycle is 0, contradicting the fact that simple smooth closed curves (loops)  have total curvature $\pm 2\pi$ (Hopf's Umlaufsatz). 
\end{proof}

By definition, each $T_i^* \in \cT_i^*$ is a subtemplate of some $T^* \in \cT^*$; this $T^*$ defines an ordering on the faces of $T_i^*$.  Lemma~\ref{lem-adj-order} says the relative ordering of adjacent faces of $T_i^*$ is determined by $T_i^*$, so this relative ordering does not depend on the choice of $(\alpha_j)_{j \neq i}$ arising from the choice of $T^*$. More strongly, as we see next, \emph{every} $(\alpha_j)_{j \neq i} \in \prod_{j\neq i}A_j$ gives the same partial ordering on adjacent faces, not just those coming from some $T^*$. (We expect, but do not prove, that every $(\alpha_j)_{j \leq b}\in \prod_{j \leq b} A_j$ arises from some choice of $T^*$.)

\begin{lemma}\label{lem-adj-order-i}
	Fix $i \leq b$ and $T_i^* \in \cT_i^*$. Suppose $\alpha_j \in A_j$ for each $j \neq i$. Draw a curve in each face of $T_i^*$ from the root vertex to the terminal vertex, and draw curves in each $H_j$ according to $\alpha_j$, to get a simple curve from $x$ to $x$ visiting each face exactly once (Lemma~\ref{lem-hamiltonian}). For the induced ordering on faces of $T_i^*$, the relative order of adjacent faces does not depend on $(\alpha_j)_{j \neq i}$. 
\end{lemma}
\begin{proof}
	Let $F,F'$ be adjacent faces of $T^*_i$, and let $v_-, v_+$ (resp.\ $v_-', v_+'$) be the root and terminal vertices of $F$ (resp.\ $F'$). If $\{ v_-, v_+\} \cap \{ v_-', v_+'\} \neq \emptyset$, then one face's terminal vertex is the other's root vertex, and the former face precedes the latter regardless of the choice of $(\alpha_j)_{j \neq i}$.

	Otherwise, suppose for the sake of contradiction that there exists $\{\alpha_j^{FF'}\}_{j \neq i}$ for which $F$ precedes $F'$, and $\{\alpha_j^{F'F}\}_{j \neq i}$ for which $F'$ precedes $F$. Assume  $\partial_r F$ intersects $\partial_r F'$ on an edge; the other case is similar.
	Let $\eta_{FF'}$ be the clockwise boundary arc of $\partial (F \cup F')$ from $v_+$ to $v_-'$, and let $\eta_{F'F}$ be the clockwise boundary arc of $\partial (F \cup F')$ from $v_+'$ to $v_-$, see Figure~\ref{fig-adjacency} (middle).
	By precisely the same argument of Lemma~\ref{lem-adj-winding} with the curve from $T^*$ replaced by the curve arising from $\{ \alpha_j^{FF'}\}_{ j \neq i}$, the total curvature of the concatenation of $\eta_F$, $\eta_{FF'}$ and $\eta_{F'}$ is $\theta(v_+') - \theta(v_-)$. By instead using the curve arising from $\{ \alpha_j^{F'F}\}_{ j \neq i}$, the total curvature of the concatenation of $\eta_{F'}, \eta_{F'F}, \eta_F$ is $\theta( v_+) - \theta(v_-')$. Let $\eta$ be a simple loop obtained by concatenating $\eta_F, \eta_{FF'}, \eta_{F'}, \eta_{F'F}$ and having its starting and ending point be a point on the interior of $\eta_F$, so its total curvature is well-defined by Definition~\ref{def-winding-nonsmooth}. By combining the previous two total curvature values, and the fact that the total curvatures of $\eta_F$ and $\eta_{F'}$ are $\theta(v_+) - \theta(v_-)$ and $\theta(v_+') - \theta(v_-')$, we conclude that the total curvature of $\eta$  is $0$. On the other hand, $\eta$ is a simple loop so it should have total curvature $\pm 2\pi$ (Lemma~\ref{lem-hopf}), a contradiction. 
\end{proof}
We now turn to the main proof. 

\begin{figure}[ht!]
	\begin{center}
		\includegraphics[scale=0.37]{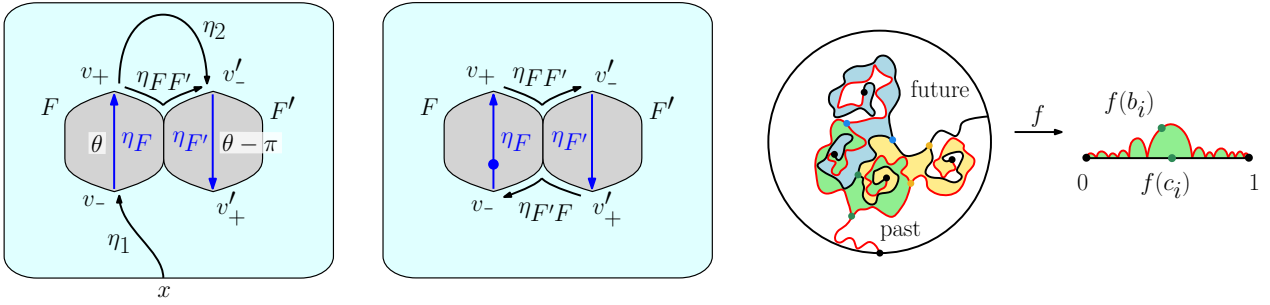}%
		\caption{\label{fig-adjacency}
				\textbf{Left:} Figure for Lemma~\ref{lem-adj-winding}. Suppose $F$ preceeds $F'$ but are not consecutive faces, and suppose $\partial_r F$ intersects $\partial_r F'$.
				 Then $\theta(v_+') - \theta(v_-)$ is the total curvature of the concatenation of $\eta_F$, $\eta_{FF'}$ and $\eta_{F'}$. 
			In this figure, the curves of $F$ and $F'$ are line segments, and the winding in $\eta_{F'}$ is $\pi$ less that that in $\eta_F$. 
			\textbf{Middle:} Figure for Lemma~\ref{lem-adj-order-i}. If $F$ preceeds $F'$ for some $\{ \alpha_j^{FF'}\}_{j \neq i}$ but $F'$ preceeds $F$ for some $\{ \alpha_j^{F'F}\}_{j \neq i}$, then the
the loop obtained by concatenating {$\eta_F, \eta_{FF'}, \eta_{F'},  \eta_{F'F}$} and rooting at a point in the interior of $\eta_F$ has total curvature 0 (in the sense of Definition~\ref{def-winding-nonsmooth}). This contradicts the fact that this loop has total curvature $\pm 2\pi$. \textbf{Right:} Figure for Lemma~\ref{lem-recover-psi-curve-cell} for $\gamma \in (\sqrt2, 2)$, where the interior of $D_i$ (green, left) is conformally mapped into a subset of $\bbH$ such that the right boundary of $D_i$ is sent to $[0,1]$.
		}
	\end{center}
\end{figure}

\begin{proof}[Proof of Proposition~\ref{prop-product-templates}]
	Clearly the map is injective. We will show it is surjective. 
	Suppose $(T_1^*, \dots, T_b^*) \in \prod_{i=1}^b \cT_i^*$. Let $T^*$ be the marked template obtained by filling in the $i$th hole of $T_\mathrm{sub}$ in the same way as in  $T_i^*$. We will show that $T^* \in \cT^*$. Since $T^*$ has $T_\mathrm{sub}$ as its marked subtemplate, we need to show the four conditions~\eqref{item-T-a}-\eqref{item-T-d} defining $\cT$ in Definition~\ref{def-T}  hold for $T^*$.~\eqref{item-T-a} is automatic since $T^*$ has $T_\mathrm{sub}$ as a subtemplate. 
	By Lemmas~\ref{lem-Ti-a} and~\ref{lem-hamiltonian} the faces of $T^*$ can be ordered as $F_0, F_1, \dots, F_{n+1}$ so that~\eqref{item-T-b} holds. Next~\eqref{item-T-c} holds for each $T_i^*$ since $T_i^*$ is a subtemplate for some element of $\cT^*$, so~\eqref{item-T-c} also holds for $T^*$.  Finally we address~\eqref{item-T-d}. First suppose $F$ is an unmarked face of $T^*$, i.e. $F$ lies in the $i$th hole of $T_\mathrm{sub}$ for some $i$. We know $T_i^*$ is a subtemplate of some $\hat T^* \in \cT^*$, and $\hat T^*$ satisfies~\eqref{item-T-d} for the face $F$. By Lemma~\ref{lem-adj-order-i}, the relative order of $F$ and any adjacent face is the same in $T^*$ and $\hat T^*$, so $T^*$ also satisfies~\eqref{item-T-d} for the face $F$. A similar argument applies if $F$ is a marked face of $T^*$, so~\eqref{item-T-d} holds. Thus, $T$ satisfies all four conditions so $T \in \cT$, and so $T^* \in \cT^*$.
\end{proof}

\subsection{Proof of Theorem~\ref{thm-eps-indep}}\label{subsec-indep-main-proof}

In order to prove Theorem~\ref{thm-eps-indep}, we will  show that the IG field is measurable with respect to the LQG surface decorated by space-filling SLE curve segments. We do this first for each $D_i$ (Lemma~\ref{lem-recover-psi-curve-cell}), then for the whole explored region (Lemma~\ref{lem-recover-psi-curve}).
	
	\begin{lemma}\label{lem-recover-psi-curve-cell}
		There is a measurable function $F$ such that the following holds. 
		In the setting described just above Theorem~\ref{thm-eps-indep}, a.s.\ $ (D_i, \Phi, \Psi, (\eta_i', a_i, b_i, c_i, d_i) )/{\sim}  = F((D_i, \Phi, \eta_i', a_i, b_i, c_i, d_i)/{\sim})$ for each $i=0,\dots,N+1$.
	\end{lemma}

\begin{proof}
	Let $\wh \Psi$ be the distribution $\Psi$ viewed modulo additive constant in $2\pi \Z$. We will show that $\wh \Psi$ is measurable with respect to $(D_i, \Phi, \eta_i', a_i, b_i, c_i, d_i)$. Since the arguments we will use do not depend on the choice of embedding of the decorated LQG surface, this implies the lemma. 
		
	By \cite[Theorem 1.10]{ig4}, for each connected component $U$ of $\mathrm{int}(D_i)$ (the interior of $D_i$), the field $\Psi|_U$ is determined up to additive constant in $2\pi \Z$. This completes the proof for $\gamma \in (0,\sqrt2]$ since $\mathrm{int}(D)$ is connected in that regime. For $\gamma \in (\sqrt2, 2)$, it remains to show that fixing one of these additive constants fixes all the remaining constants. Basically, the reason for this is that $\partial D_i$ is the union of four flow lines of $\Psi$, and the values of $\Psi$ along a flow line are well-defined up to adding a single integer multiple of $2\pi\chi$.
	
	Recall that $D_i$ was defined to be the region traced by the space-filling curve $\eta'$ between hitting two points. Therefore, $\partial D_i$ comprises segments of flow lines of $\Psi$ with specified angles. For each connected component $U$ of $D_i$ let $\partial^p U$ denote the prime end boundary of $D_i$, let $\partial^p_r U$ be the right boundary arc of $U$ with orientation chosen  such that its starting and ending points agree with those of $\eta$ in $U$, and let $\partial_r^p D_i$ be the concatenation of the $\partial^p_r U$ ordered by the time $U$ is filled out by $\eta$. 
	
	 Let $f: \mathrm{int}(D_i) \to \bbH$ be a map such that $f: \mathrm{int}(D_i) \to f(\mathrm{int}(D_i))$ is conformal, and whose extension to $\partial^p_r D_i$ is a homeomorphism from  $\partial^p_r D_i$ to $[0,1] \subset \bbH$. See Figure~\ref{fig-adjacency} (right). Let $\wt \Psi$ be the distribution on $F(\mathrm{int}(D_i))$ given by $\wt \Psi = \Psi \circ f^{-1} - \chi(26 - \ccL)\arg ((f^{-1})')$, where  $\mathrm{arg}((f^{-1})')$ is as defined in~\eqref{eq-def-arg}. 
	 By the definition of boundary conditions of interior flow lines~\cite[Theorem 1.1]{ig4}, and the consistency of~\eqref{eq-def-arg} with the usual definition of $\arg ((f^{-1})')$ for simply connected domains (Proposition~\ref{prop-arg-f'}), $\wt \Psi$ is constant on $[0, f(c_i))$ and on $(f(c_i), 1]$. Furthermore, the point $c_i$ lies on the boundary of one of the connected components of the interior of $D_i$. Consequently, fixing the additive constant of one of the $\Psi|_U$ fixes the additive constants of $\Psi|_{U'}$ for all $U'$, as needed. 
\end{proof}

\begin{lemma}\label{lem-recover-psi-curve}
	In the setting of Theorem~\ref{thm-eps-indep}  let $\cM_\tau'= (M_\tau, \Phi, \{(\eta'_i, a_i, b_i, c_i, d_i)\: : \: i \in I_\tau \})/{\sim}$ be the decorated LQG surface obtained from $\cM_\tau$ by forgetting the IG field $\Psi$. 
	Then
	$\cM_\tau$ is measurable with respect to $\cM_\tau'$. 
\end{lemma}
\begin{proof}
	Let $\wh \Psi$ be $\Psi$ viewed modulo additive constant in $2\pi \Z$. 
	We will show that $\wh \Psi|_{M_\tau}$ is measurable with respect to $(M_\tau, \Phi, \{(\eta'_i, a_i, b_i, c_i, d_i)\: : \: i \in I_\tau \})$, i.e., we prove that the desired result holds when LQG surfaces are replaced by their embeddings in $M_\tau$. None of the arguments we will use depend on the choice of embedding, so this implies the desired result. 
	
	Condition on $(M_\tau, \Phi, \{(\eta'_i, a_i, b_i, c_i, d_i)\: : \: i \in I_\tau \})$.
	By Lemma~\ref{lem-recover-psi-curve-cell}, $\Psi|_{D_i}$ is determined modulo additive constant in $2\pi \Z$ for each $i$. Next, if $D_i$ intersects $D_{i'}$ on an interval $\eta$,  then since $\eta$ is a segment of a  flow line, the additive constant for $\Psi|_{D_i}$ fixes the additive constant for $\Psi|_{D_{i'}}$  by  flow line boundary conditions  \cite[Theorem 1.1]{ig4}. Finally, for any $i, i' \in I_\tau$ it is possible to find a sequence $i_1, \dots, i_j$ with $i_0 = i$ and $i_j = i'$ such that $D_{i_k}$ and $D_{i_{k+1}}$ intersect on an interval for all $k$. We conclude that under this conditioning $\wh \Psi|_{M_\tau}$ is deterministic, i.e.,  $\wh \Psi|_{M_\tau}$ is measurable with respect to  $(M_\tau, \Phi, \{(\eta'_i, a_i, b_i, c_i, d_i)\: : \: i \in I_\tau \})$.
\end{proof}

Finally, we show that for all quilts having a certain subtemplate, the peeling procedure of Theorem~\ref{thm-eps-indep} chooses this subtemplate with probability not depending on the quilt.
\begin{lemma}\label{lem-peel-prob-same}
	Fix $M \in \cQ_{{T_\mathrm{sub}}}$ for some $T_\mathrm{sub}$ as in Section~\ref{subsec-indep-quilts}, and let $Q$ be any quilt having $M$ as a subtemplate. 
	Carry out the peeling procedure on $Q$ analogous to that of  Theorem~\ref{thm-eps-indep}. That is, sample $T \geq 0$ with $\P[T = t] = 2^{-t-1}$, explore all faces adjacent to the external face at step 0, at each subsequent step explore the faces adjacent to a boundary point chosen uniformly at random according to length, and stop after step $T$ or when all faces are explored, whichever is earlier. Then the probability the explored subtemplate is $M$ does not depend on the choice of $Q$.
\end{lemma}
\begin{proof}
	Given $Q$, the event $\{ \text{explored subtemplate is } M\}$ occurs when $T$ equals the number of non-boundary-touching faces of $M$,  and during each of the $T$ peeling steps we pick a point that lies in $M$. The probability of this event depends only on $M$ since $M$ determines all the lengths involved in the probability computation for the $T$ peeling steps. 
\end{proof}

We can now prove Theorem~\ref{thm-eps-indep}.

\begin{proof}[Proof of Theorem~\ref{thm-eps-indep}]
	We can sample the $\eps$-discretized LQG disk in two steps. First, we sample a quilt $Q \in \bigcup_{T \in \cT} \cQ_T$ from $\cL^\eps$. Then conditioned on $Q$, for each face $F$ we  sample an independent LQG cell $\cC_F$ conditioned on having boundary lengths given by the side lengths of $F$, and conformally weld the $\cC_F$ according to $Q$. By  Proposition~\ref{prop-decomp-quilt-cell} this gives an $\eps$-discretized LQG disk. 
	
	The peeling process described can already be applied after the first step of the above sampling procedure, namely to the quilt $Q$ (see Lemma~\ref{lem-peel-prob-same}). Mark the faces of $Q$  discovered during the peeling process, and let $M$ be the marked subquilt. By Bayes' theorem, we may flip this procedure around: sample $M$ from its marginal law, then sample $Q \sim \cL^\eps$ conditioned on having $M$ as a subquilt and weighted by the probability of marking $M$ given $Q$. By Lemma~\ref{lem-peel-prob-same} this probability does not depend on $Q$, so there is in fact no weighting. Thus, by  Proposition~\ref{prop-indep-subquilts}, conditioned on $M$ the complementary subquilts $Q_1, \dots, Q_b$ of $M$ are conditionally independent.
	Conditioned on $M, Q_1, \dots, Q_b$, sample independent collections of LQG cells $(\cC_F)_{F \in M}, (\cC_F)_{F \in Q_1}, \dots, (\cC_F)_{F \in Q_b}$.

	Suppose $\gamma \in (0, \sqrt 2]$, and for $1\leq i \leq b$ let $\cD_i$ be the decorated LQG surface arising from conformally welding $(\cC_F)_{F \in Q_j}$ according to $Q_j$. Then $\cD_1, \dots, \cD_b$ are conditionally independent given $\sigma (M, (\cC_F)_{F \in M})$. 
	Let $\wt \cM_\tau$ be the conformal welding of $(\cC_F)_{F \in M}$ according to $M$, and $\cM_\tau'$ the decorated LQG surface $\cM_\tau$ after forgetting the IG field. Since $\wt \cM_\tau = \cM_\tau'$, and using Lemma~\ref{lem-recover-psi-curve}, we have $\sigma(M, (\cC_F)_{F \in M}) = \sigma(\wt \cM_\tau) = \sigma(\cM_\tau') = \sigma(\cM_\tau)$. Thus $\cD_1, \dots, \cD_b$ are conditionally independent given $\cM_\tau$. Finally, the loops $\eta_1, \dots, \eta_m$ partition $\{\cD_1, \dots, \cD_b\}$, and each $(O_k, \Phi, \Psi)/{\sim}$ is measurable with respect to $\cM_\tau$ and the $k$th subset of the partition (Lemma~\ref{lem-recover-psi-curve}), so $(O_1, \Phi, \Psi)/{\sim}, \dots, (O_m, \Phi, \Psi)/{\sim}$ are conditionally independent given $\cM_\tau$.

	When $\gamma \in (\sqrt2, 2)$, we do not necessarily have $\wt \cM_\tau = \cM_\tau'$. Each $\cC_F$ has cut points, so even though the planar maps $Q_i$ each have the disk topology, the LQG surfaces $\cD_i$ may have cut points. Thus $\cM_\tau'$ arises from $\wt \cM_\tau$ by identifying pairs of boundary points of $\wt \cM_\tau$ according to the cut points of $\cD_1, \dots, \cD_b$. Since the identifications on the $i$th boundary component of $\wt \cM_\tau$ are a function of $\wt \cM_\tau$ and $\cD_i$, and $\cD_1, \dots, \cD_b$ are conditionally independent given $\wt \cM_\tau$, we conclude $\cD_1, \dots, \cD_b$ are conditionally independent given $\cM_\tau'$, and thus conditionally independent given $\cM_\tau$ (Lemma~\ref{lem-recover-psi-curve}). The rest of the argument is identical to the $\gamma \in (0,\sqrt2]$ case. 
\end{proof}

\section{Open problems}
\label{sec-open-problems}

The results of this paper are only proven for LQG in the subcritical phase, i.e., when $\ccL > 25$ or equivalently $\gamma\in (0,2)$. 
The reason for this is that our proofs rely on the mating of trees theorem~\cite{wedges,ag-disk} in order to get a decomposition of the LQG disk into small pieces with nice independence properties. An analog of the mating of trees theorem in the critical case $\ccL = 25$, $\gamma=2$ is proven in~\cite{ahps-critical-mating}, but the analog of space-filling SLE is not a continuous curve, so this theorem does not give a decomposition of the LQG disk into \emph{small} independent pieces. 

\begin{prob} \label{prob-critical}
	Extend the theorems from Section~\ref{subsec-intro-results} to the case of critical LQG ($\ccL =25$, $\gamma=2$).
\end{prob}

One possible approach to Problem~\ref{prob-critical} is to take a limit of the objects involved in our theorem statements as $\gamma\to 2^-$. 

The recent paper~\cite{ag-supercritical-cle4} defines an analog of the LQG disk in the supercritical phase $\ccL \in (1,25)$. The paper also constructs a coupling of the supercritical LQG disk with CLE$_4$, which has similar properties to the relationships between LQG with $\ccL \geq 25$ and independent SLE with central charge $\cc_\mathrm{M} = 26-\ccL$. 
It is natural to ask whether there are also ``mismatched" coupling statements for LQG with $\ccL \geq 25$. 

\begin{prob} \label{prob-supercritical}
	Extend the theorems from Section~\ref{subsec-intro-results} to the case of supercritical LQG ($\ccL \in (1,25)$, $\gamma\in\mathbb C$ with $|\gamma|=2$).
\end{prob}

An especially interesting special case of Problem~\ref{prob-supercritical} is when we have a supercritical LQG surface of central charge $\ccL = 26-k \in \{2,\dots,24\}$ decorated by $k$ independent Gaussian free fields (i.e., $\mathbf c_1 = \dots = \mathbf c_k = 1$). 
The reason is that, from the work of Polyakov~\cite{polyakov-qg1}, this combination of objects is expected to be related to bosonic string theory in $\BB R^k$. Moreover, this same combination of objects is also potentially connected to Yang-Mills theory (see, e.g.,~\cite{cps-random-surface-ym}). See Section~\ref{sec-string-theory} for further discussion.

In the matched case $\kappa \in \{\gamma^2,16/\gamma^2\}$, there are relationships between SLE and LQG for a wide variety of different LQG surfaces (see, e.g.,~\cite{wedges,sphere-constructions,ag-disk,ahs-integrability,asy-quantum-triangle}). For simplicity, in this paper we focused exclusively on the case of the LQG disk, but we expect that analogs of our results are true for other types of LQG surfaces. 

\begin{prob} \label{prob-other-surfaces}
	Prove analogs of the theorems from Section~\ref{subsec-intro-results} for other LQG surfaces besides the LQG disk, such as LQG disks with boundary insertions, LQG triangles, LQG spheres, LQG wedges, LQG cones, and non-simply connected LQG surfaces. 
\end{prob}

It is of interest to consider variants of the results of this paper where the imaginary geometry fields are replaced by other types of ``matter fields" (random generalized functions satisfying a conformal covariance rule).
Some examples of possible matter fields to consider are conformal loop ensemble nesting fields~\cite{mww-extremes}, Ising spin fields~\cite{cgn-ising-spin1}, \textit{winding fields} and \textit{layering fields} associated with Brownian loop soups~\cite{cgk-loop-soup-correlation,vcl-spin-systems,le-jan-gff-loup-soup,cgpr-layering}, and the conformally invariant fields associated with Brownian loop soups constructed in~\cite{jlq-cle-field}. 

\begin{prob}
	Are there analogs of the conditional independence results of this paper where the auxiliary imaginary geometry fields are replaced by fields of the types mentioned just above?
\end{prob}

Our main results give exact conditional independence statements for LQG surfaces cut by various types of sets, but, in contrast to the matched case, we do not have an exact description of either the law of the information we are conditioning on or the conditional laws of the complementary LQG surfaces given this information. 

\begin{prob} \label{prob-exact}
	Assume that we are in the setting of Theorem~\ref{thm-sle-chordal}. 
	Is there an explicit random process which generates the $\sigma$-algebra $\cF $ of~\eqref{eqn-restriction-sigma-algebra}? 
	Can the law of such a random process can be described explicitly? 
	What about the analogous questions in the settings of the other main results from Section~\ref{subsec-intro-results}?
\end{prob}

\begin{prob} \label{prob-cond-law}
	In the setting of Theorem~\ref{thm-sle-chordal}, is there an explicit description of the conditional laws of the IG-decorated LQG surfaces parametrized by the complementary connected components of the SLE curve $K$ given the $\sigma$-algebra $\mcl F$?
	What about in the settings of the other main results from Section~\ref{subsec-intro-results}?
\end{prob}

Here is a possible approach for Problem~\ref{prob-exact}. 
For simplicity, assume that $K$ is the image of an SLE$_\kappa(\rho_L;\rho_R)$ curve $\eta$ with $\rho_L \geq \kappa/2-4$,  so that $\eta$ does not hit the clockwise arc of $\bdy\BB D$ from its starting point to its ending point~\cite[Lemma 15]{dubedat-duality}.   
Let $\nu_\Phi$ be the $\gamma$-LQG length measure on $\eta$, which should be defined as a Gaussian multiplicative chaos measure with respect to the Euclidean Minkowski content measure on $\eta$~\cite{lawler-rezai-nat}. 
Assume that $\eta$ is parametrized so that $\nu_\Phi(\eta[0,t]) = t$ for each $t \in [0,T]$, where $T = \nu_\Phi(\eta)$. 

Let $f$ be a conformal map from the connected component of $\BB D\setminus \eta$ which lies to the left of $\eta$ to the upper half-plane $\BB H$, taking the starting and ending points of $\eta$ to 0 and $\infty$, respectively. For $j=1,\dots,n$, define 
\eqb \label{eqn-int-along-sle}
F_j(t) = \int_0^t \left( \Psi_j - \chi(\mathbf c_j) \op{arg} f' \right)(\eta(s)) \,ds ,\quad \forall t \in [0,T], 
\eqe 
where we recall that $(\Psi_1,\dots,\Psi_n)$ are the auxiliary imaginary geometry fields of central charges $(\cc_1,\dots,\cc_n)$. 
We conjecture that the process $(F_1,\dots,F_n)$ generates the $\sigma$-algebra~\eqref{eqn-restriction-sigma-algebra}. It would be very interesting to find an explicit description of the law of $(F_1,\dots,F_n)$.

One of the most important relationships between SLE and LQG in the matched case is the \textbf{mating of trees} theorem~\cite[Theorem 1.9]{wedges}. This theorem says the following: suppose we have a $\gamma$-LQG cone together with an independent space-filling SLE$_{\kappa'}$ curve $\eta'$ for $\kappa'=16/\gamma^2$. Parametrize $\eta'$ so that it traverses one unit of $\gamma$-LQG mass in one unit of time. 
Let $L_t$ (resp.\ $R_t$) denote the net change in the $\gamma$-LQG length of the left (resp.\ right) outer boundary of $\eta'$ relative to time 0. Then the process $(L,R)$ is a two-dimensional Brownian motion with correlation $\op{Corr}(L_t,R_t) = -\cos(\pi\gamma^2/4)$ (note that the correlation is zero if $\kappa'=8$). This theorem is proven building on the quantum zipper theorem in~\cite{shef-zipper}. Our results can be viewed as a mismatched analog of the quantum zipper theorem. It is therefore natural to ask of there is also a mismatched mating of trees theorem.

\begin{prob} \label{prob-mating}
	Suppose we have an appropriate $\gamma$-LQG surface together with a space-filling SLE$_{\kappa'}$ curve for $\kappa'\not=16/\gamma^2 > 4$, sampled independently from the LQG surface and then parametrized by LQG mass. 
	Define the (mismatched) left/right $\gamma$-LQG boundary length process $(L ,R)$ as above.
	Is there a variant of our main results which says that the joint conditional law of $L$ and $R$ given certain auxiliary information along the curve $\eta'$ is particularly simple? Are $L$ and $R$ conditionally independent given this information for certain special values of $\gamma$ and $\kappa'$?
\end{prob} 

In a similar vein to~\eqref{eqn-int-along-sle}, a possible choice of auxiliary information to consider in Problem~\ref{prob-mating} is the following. Let $\Psi_1$ be the whole-plane GFF coupled with $\eta'$ in the sense of imaginary geometry and let $\Psi_2$ be another imaginary geometry field (independent from everything else) such that the central charges associated with $\eta'$, the $\gamma$-LQG surface, and $\Psi_2$ sum to 26. Let $a,b  \in [-1,1]$ with $a^2+b^2 = 1$ be chosen so that $a \Psi_1 + b\Psi_2$ is a field of central charge 1 (i.e., with $\chi = 0$). Then, consider the $\sigma$-algebra generated by the function  
\eqb \label{eqn-int-along-sle-2}
t\mapsto  \int_0^t \left( a \Psi_1 + b \Psi_2 \right)(\eta(s)) \,ds . 
\eqe 
In the special case when our LQG surface is an LQG sphere with $\ccL = 30$, $\kappa' = 8$, and $\Psi_2$ is an imaginary geometry field of central charge $-2$, we expect that $L$ and $R$ are conditionally independent given the $\sigma$-algebra generated by~\eqref{eqn-int-along-sle-2}. Indeed, such a conditional independence statement would be a continuum analog of the conditional independence statement for meanders given in Proposition~\ref{prop-meander}; c.f.~\cite[Conjecture 1.3]{bgs-meander}.

Next, let $\Psi_1$ and $\Psi_2$ be independent imaginary geometry fields on $\BB D$ of central charges $\cc_1$ and $\cc_2$, respectively. By Proposition~\ref{prop-ig-rotate}, if $a,  b \in [-1,1]$ with $a^2+b^2 = 1$, then $\wh\Psi := a\Psi_1 + b \Psi_2$ is another imaginary geometry field with a possibly different central charge. It is natural to wonder how the flow lines of $\Psi_1$, $\Psi_2$, and $\wh\Psi$ are related. 

\begin{prob} \label{prob-flow-lines}
	Is there a direct way to construct the flow lines of $\wh\Psi$ from the flow lines of $\Psi_1$ and $\Psi_2$?
\end{prob}

Here is a possible approach to Problem~\ref{prob-flow-lines}. Let $q > 0$ and $\ep>0$. Grow a flow line of $\Psi_1$ (parametrized by Minkowski content, say) up to time $\ep$. Then, grow a flow line of $\Psi_2$ started from the tip of this first flow line, stopped at time $q\ep$. Continue alternating $\ep$-increments of flow lines of $\Psi_1$ and $q\ep$-increments of flow lines of $\Psi_2$ in this manner to construct a curve $\eta_\ep$. It is reasonable to guess that if we choose $q = q(a,b)$ appropriately, then $\eta_\ep$ should converge to a flow line of $\wh\Psi$ as $\ep\to 0$. 

One motivation for Problem~\ref{prob-flow-lines} comes from lattice models. To explain this, we follow the exposition in~\cite[Section 1]{dubedat-coupling}. Let $\ep > 0$ and let $\mcl D_1$ and $\mcl D_2$ be two independent copies of the dimer model on $\ep \BB Z^2 \cap \BB D$. By the Temperlyan bijection, $\mcl D_1$ gives rise to a height function $\psi_1$ on the dual graph of $\ep\BB Z^2 \cap \BB D$, and similarly for $\mcl D_2$. These same height functions can also be obtained as height functions of two independent uniform spanning trees $T_1, T_2$ on $\ep\BB Z^2\cap \BB D$. The height functions $\psi_1$ and $\psi_2$ converge as $\ep\to 0$ to two independent copies of the $\kappa=2$ ($\cc = -2$) imaginary geometry field on the disk~\cite{kenyon-dominos-gff}. Moreover, the peano curves associated with $T_1$ and $T_2$ converge to two independent SLE$_8$s~\cite{lsw-lerw-ust} which should be counterflow lines of these two imaginary geometry fields. 

Now, consider the double dimer model $\mcl D_1 \cup \mcl D_2$. It is believed (but not proven) that the set of boundaries of connected components of $\mcl D_1 \cup \mcl D_2$ converges as $\ep\to 0$ to CLE$_4$. As explained in~\cite[Section 1.2]{dubedat-coupling}, the height function associated with $\mcl D_1\cup \mcl D_2$ is the difference $\wh\psi = \psi_1 - \psi_2$. This height function should converge as $\ep\to 0$ to a GFF on $\BB D$ whose level lines are the CLE$_4$ loops. Thus, the continuum limit of the double dimer model gives us two $\cc=-2$ imaginary geometry fields on $\BB D$ whose difference is a GFF. 

\begin{prob} \label{prob-lattice}
	Are there any other collections of lattice models where one naturally sees linear combinations of fields which converge to imaginary geometry fields? 
	Can the solution of Problem~\ref{prob-flow-lines} be useful in analyzing these discrete models?
\end{prob}

\appendix

\section{Planar maps decorated by multiple statistical physics models} 
\label{sec-rpm}

Theorem~\ref{thm-sle-chordal} is the continuum analog of certain Markovian properties for random planar maps {with boundary} decorated by two or more statistical physics models. The general principle is as follows. Suppose that one of our statistical physics models gives rise to an ``interface" $\eta$, viewed as a curve on the planar map $M$ (or on a closely related planar, such as the dual map). 
We can view the connected components of $M\setminus \eta$ as submaps of $M$. Then in many cases these submaps are conditionally independent given the appropriate information about the restrictions to $\eta$ of the \emph{other} statistical physics models (besides the one used to generate $\eta$). {In a similar vein, our other main results can be seen as discrete analogs of Markov properties for, e.g., random walk or graph distance balls on random planar maps decorated by one or more statistical physics models.}

For concreteness, we now fully explain this principle in one special case: planar maps decorated by two instances of the Gaussian free field and one spanning tree. This case is particularly nice since the underlying random planar map is exactly uniform. In general, the marginal law of the underlying random planar map may not admit a simple description.

For a graph $G$, write $\mcl V(G)$ and $\mcl E(G)$ for its sets of vertices and edges, respectively. 
For $n\in\BB N$, let $\mcl M_n$ be the set of triangulations with the disk topology with simple boundary of perimeter $n$ (i.e., triangulations with a marked exterior face whose boundary is a simple cycle of length $n$). For concreteness, we allow multiple edges but not self-loops. We note that everything in this section also works with other types of planar maps, e.g., triangulations with no multiple edges, quadrangulations, or planar maps with mixed face degrees.

Let $\mcl M_n^\bullet$ be the set of 4-tuples $(M , T , \phi_1,\phi_2)$ where $M\in \mcl M_n$, $T$ is a spanning tree on $M$ with wired boundary condition, and $\phi_1 , \phi_2 : \mcl V(M) \to \BB R$ are functions on the vertex set of $M$ which vanish at each boundary vertex. 

To define a probability measure on $\mcl M_n^\bullet$, let $ \beta > 0$ be a universal constant to be chosen later. For $M\in\mcl M_n$ write $\lambda^M$ for Lebesgue measure on the space of functions $\phi : \mcl V(M) \to\BB R$ which vanish on the boundary. 
Also write $\mu$ for counting measure on pairs $(M,T)$ where $M \in \mcl M_n$ and $T$ is a spanning tree on $M$ with wired boundary condition. 
We define a probability measure on $\mcl M_n^\bullet$ by
\eqb \label{eqn-map-gff-tree}
\frac{1}{Z} e^{-\beta\# \mcl E(M)} \exp\left( -   \sum_{i\in \{1,2\} } \sum_{ \{ x,y\} \in \mcl E(M)} |\phi_i(x) - \phi_i(y)|^2    \right) \,d\lambda^M(\phi_1) \, d\lambda^M(\phi_2) \, d\mu(M,T)
\eqe  
where $Z$ is a normalizing constant.

Under this probability measure, the conditional law of $(T,\phi_1,\phi_2)$ given $M$ is that of an independent triple consisting of a uniform spanning tree on $M$ with wired boundary conditions and two discrete zero-boundary Gaussian free fields (GFFs) on $M$. 
Moreover, the marginal law of $M$ is given by the counting measure on $\mcl M_n$ weighted by (a quantity proportional to) 
\allb \label{eqn-map-partition}
e^{-\beta\# \mcl E(M)}  \left[ \int \exp\left( -  \sum_{ \{ x,y\} \in \mcl E(M)} |\phi (x) - \phi (y)|^2    \right) \,d\lambda^M(\phi )  \right]^2   
\times \#\left\{\text{spanning trees of $M$} \right\} .
\alle

By a standard computation for the discrete GFF (see, e.g., \cite[Section 4.1]{dubedat-coupling}), the integral with respect to $\lambda^M$ in~\eqref{eqn-map-partition} is equal  to an exponential function of $\#\mcl V(M)$ times $(\det\Delta_M)^{-1/2}$, where $\Delta_M$ is the discrete Laplacian determinant on $M$ (restricted to functions which vanish on the boundary). Furthermore, by the Kirchoff matrix-tree theorem, the number of spanning trees of $M$ is equal to $\det\Delta_M$. 
Hence, by canceling two factors of $\det\Delta_M$, we get that the quantity~\eqref{eqn-map-partition} depends only on $\#\mcl V(M)$. Therefore, the marginal law of $M$ under~\eqref{eqn-map-gff-tree} is given by the counting measure on $\mcl M_n$ weighted by an exponential function of $\#\mcl V(M)$. Since we are looking at triangulations, the Euler characteristic formula shows that this is the same as the counting measure on $\mcl M_n$ weighted by an exponential function of $\#\mcl E(M)$.

If we choose the constant $\beta$ appropriately, then under the above marginal law of $M$ the law of $\#\mcl E(M)$ has a polynomial tail (see, e.g.,~\cite[Section 3.2]{curien-peeling-notes} for a more general statement). Henceforth fix this choice of $\beta$.
For this $\beta$, the triangulation $M$ converges in law $n\to\infty$ in the Gromov-Hausdorff sense to the so-called \textbf{free area Brownian disk}~\cite{bet-mier-disk,gwynne-miller-simple-quad,aasw-type2}, which is in turn equivalent (as a metric measure space) to the unit boundary length $\sqrt{8/3}$-LQG disk~\cite{lqg-tbm2}.

Due to the convergence of the discrete GFF to the continuum GFF and the convergence of the peano curve of the uniform spanning tree on a fixed lattice to SLE$_8$~\cite{lsw-lerw-ust}, the above convergence leads naturally to the following conjecture. 

\begin{conj} \label{conj-map-gff-tree} 
	Let $(M,T,\phi_1,\phi_2) \in \mcl M_n^\bullet$ be sampled from~\eqref{eqn-map-gff-tree} and let $P$ be the path in $M$ which traverses $T$ in contour (depth-first) order starting and ending at a marked boundary vertex. The 4-tuple $(M,P,\phi_1,\phi_2)$ converges in law under an appropriate scaling limit as $n\to\infty$ to a unit boundary length $\sqrt{8/3}$-LQG disk together with an SLE$_8$ curve starting and ending at the same marked boundary point and two independent zero-boundary GFFs. 
\end{conj}

Possible topologies of convergence in Conjecture~\ref{conj-map-gff-tree} include generalizations of the Gromov-Hausdorff topology for metric spaces decorated by a curve and a pair of generalized functions (see, e.g.,~\cite{gwynne-miller-uihpq,khezeli-gh}) or convergence under an appropriate embedding of $M$ into the disk (e.g., the Tutte embedding~\cite{gms-tutte} or the Cardy embedding~\cite{hs-cardy-embedding}). 

Let us now describe a Markov property for $(M,T,\phi_1,\phi_2)$. See Figure~\ref{fig-map-gff-tree} for an illustration. Let $M^*$ be the dual map of $M$ and let $T^*$ be the dual spanning tree of $T$, which consists of all edges of $M^*$ which do not cross edges of $T$. Let $k\leq n-2$ and let $e_1,e_2\in \bdy M$ be two edges with the property that the two boundary arcs separating them have lengths $k$ and $n-k-2$, respectively. There is a unique path $\eta$ in $T^*$ from the vertex of $M^*$ corresponding to the face of $M$ incident to $e_1$, to the vertex of $M^*$ corresponding to the face of $M$ incident to $e_2$. 
The path $\eta$ forms the outer boundary of the contour path $P$ of Conjecture~\ref{conj-map-gff-tree} if we start $P$ at the leftmost vertex of $e_1$ and stop $P$ when it first hits a vertex of $e_2$. By, e.g.,~\cite[Lemma 2.6]{ag-disk}, Conjecture~\ref{conj-map-gff-tree} suggests that the scaling limit of $\eta$ should be a chordal SLE$_2(-1,-1)$ curve between two boundary points of the disk (provided we choose $k = k(n)$ so that $k(n) / n$ converges to a number in $(0,1)$). 

Let $H$ be the sub-map of $M$ consisting the faces of $M$ traversed by $\eta$ and the vertices and edges of $M$ on the boundaries of these faces. 
We view $H$ as a planar map with boundary equipped with two marked edges, namely $e_1$ and $e_2$. 
We count each edge of $H$ which is not crossed by $\eta$ as a boundary edge (even if this edge is not part of the external face of $H$). See Figure~\ref{fig-map-gff-tree}. 

Let $M_L$ (resp.\ $M_R$) be the submap of $M$ consisting of the vertices, edges, and faces of $M$ which lie strictly to the left (resp.\ right) of $\eta$ (not including $e_1$, $e_2$, or the edges and faces which are crossed by $\eta$). 
Each of $M_L$ and $M_R$ is a triangulation with (not necessarily simple) boundary. 
We view $M_L$ (resp.\ $M_R$) as being equipped with two marked vertices, namely the endpoints of $e_1$ and $e_2$. 
Let $T_L$ (resp.\ $T_R$) be the subgraph of $T$ induced by the set of edges of $T$ which lie on $M_L$ (resp.\ $M_R$). 

The following is a discrete analog of Theorem~\ref{thm-sle-chordal}.

\begin{figure}[ht!]
	\begin{center}
		\includegraphics[scale=0.75]{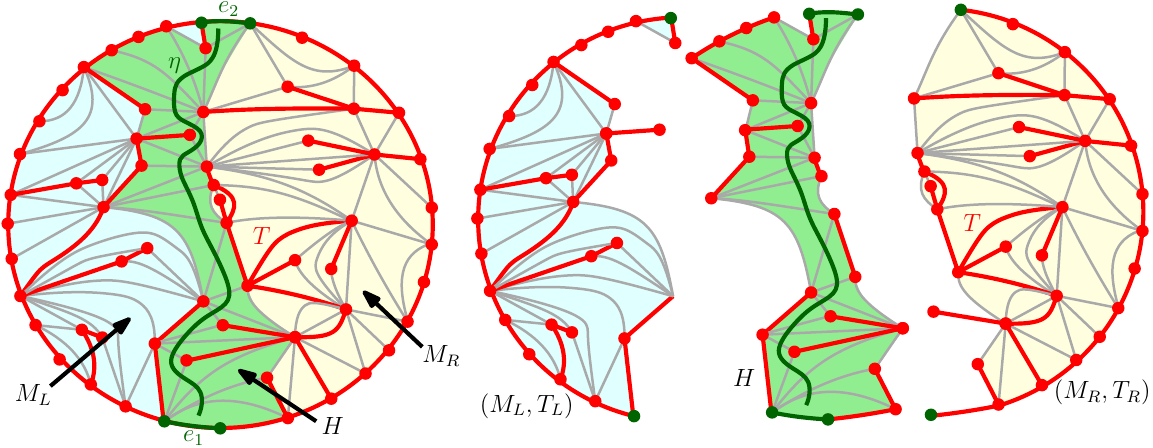}  
		\caption{\label{fig-map-gff-tree} \textbf{Left:} A triangulation with boundary $M$ (red, gray, and green edges) decorated by a spanning tree $T$ with wired boundary condition (red). The path $\eta$ is the branch of the dual tree between the marked boundary edges $e_1,e_2$, and $H$ (light green) is the submap of $M$ induced by the set of faces traversed by $\eta$. 
			\textbf{Right:} The decorated planar maps $(M_L, T_L)$ and $(M_R,T_R)$ lying to the left and right of $H$ (light blue and light yellow) are connected but do not necessarily have simple boundary. Identifying one boundary arc of each of these maps with the boundary arcs of $H$ to the left and right of $\eta$, respectively, gives us back $(M,T)$. 
			If we sample the tree / discrete GFF decorated triangulation $(M,T,\phi_1,\phi_2)$ from the probability measure~\eqref{eqn-map-gff-tree}, then Proposition~\ref{prop-discrete-markov} tells us that $(M_L , T_L ,\phi_1|_{M_L} , \phi_2|_{M_L})$ and $(M_R , T_R ,\phi_1|_{M_R} , \phi_2|_{M_R})$ are conditionally independent given $(H ,  \phi_1|_H,\phi_2|_H)$.  
		}
	\end{center}
	\vspace{-3ex}
\end{figure}

\begin{prop} \label{prop-discrete-markov}
	The left and right decorated random planar maps $(M_L , T_L ,\phi_1|_{M_L} , \phi_2|_{M_L})$ and $(M_R , T_R ,\phi_1|_{M_R} , \phi_2|_{M_R})$ are conditionally independent given the middle decorated planar map $(H ,  \phi_1|_H,\phi_2|_H)$. 
\end{prop}
\begin{proof}
	We claim that the probability measure~\eqref{eqn-map-gff-tree} admits a factorization of the form
	\eqb \label{eqn-law-decomp}
	d\nu_{H,L} (M_L , T_L ,\phi_1|_{M_L} , \phi_2|_{M_L})  \, d\nu_{H,R} (M_L , T_L ,\phi_1|_{M_L} , \phi_2|_{M_L})  \, d\nu_0(H,  \phi_1|_H,\phi_2|_H)
	\eqe
	where $\nu_0$ is the law of $(H,  \phi_1|_H,\phi_2|_H)$ and $\nu_{H,L}$ and $\nu_{H,R}$ are probability measures that depend only on $H$. This immediately implies the desired conditional independence statement. We will now check~\eqref{eqn-law-decomp} via elementary manipulations of~\eqref{eqn-map-partition}.
	
	Let us first note that by construction, 
	\eqb \label{eqn-edge-decomp}
	\#\mcl E(M) = \#\mcl E(H\setminus \bdy H) + \#\mcl E(M_L) + \#\mcl E(M_R) .
	\eqe
	Next, note that the only edges of $M$ which join a vertex of $M_L$ and a vertex of $M_R$ are the marked boundary edges $e_1,e_2$. Since $\phi_1$ and $\phi_2$ vanish on $\bdy M$, for $i\in\{1,2\}$ we have
	\allb \label{eqn-gff-decomp}
	\exp\left( -  \sum_{ \{ x,y\} \in \mcl E(M)} |\phi_i(x) - \phi_i(y)|^2    \right) 
	= \prod_{A \in \{M_L,M_R,H\}}\exp\left(  -  \sum_{ \{ x,y\} \in \mcl E(A)} |\phi_i(x) - \phi_i(y)|^2    \right) 
	\alle
	
	Finally, we deal with the counting measure term in~\eqref{eqn-map-gff-tree}.
	We first claim that if we see the decorated planar maps $(M_L , T_L)$, $(M_R,T_R)$, and $H$, then there is a unique way of gluing them together to recover $(M,T)$. Indeed, exactly one of the two boundary arcs of $M_L$ separating its two marked points includes an edge which is not in $T_L$. We identify the edges of this boundary arc in chronological order with the right boundary arc of $H$ between its two marked edges, and similarly for $M_R$. 
	This allows us to recover $M$. 
	By definition, $\eta$ crosses no edges of $T$, so $T_L\cup T_R = T \setminus \{e_1 , e_2\}$. 
	
	One easily sees that the set of possibilities for $(M_L,T_L)$ depends only on $H$ (not on $(M_R,T_R)$) and similarly with $L$ and $R$ interchanged. This allows us to decompose the counting measure $\mu$ on pairs $(M,T)$ as a product measure 
	\eqb  \label{eqn-map-decomp}
	d\mu(M,T) =   d \mu_{H,L}(M_L,T_L) \, d\mu_{H,R}(M_R,T_R) \, d\mu_0(H) 
	\eqe 
	where $\mu_0$ is the law of $H$ and $\mu_{H,L}$ and $\mu_{H,R}$ are counting measures on $H$-dependent sets of tree-decorated maps. Plugging~\eqref{eqn-edge-decomp}, \eqref{eqn-gff-decomp}, and \eqref{eqn-map-decomp} into~\eqref{eqn-map-gff-tree} gives \eqref{eqn-law-decomp}.  
\end{proof}

Finally, we make some general comments on the central charges associated to statistical physics models. 

\begin{remark}\label{rem-discrete-cc}
	If $M$ is a planar map and $S$ is a statistical mechanics model on $M$, then the central charge associated with $S$ is the number $\cc$ such that when $M$ is large, the partition function of $S$ is ``close" (in a sense that we deliberately leave vague) to an exponential function of $\#\mcl E(M)$ or $\#\mcl V(M)$ times $(\det \Delta_M)^{-\cc/2}$, where $\Delta_M$ is the discrete Laplacian.
	By the discussion immediately after~\eqref{eqn-map-partition}, spanning trees have central charge $-2$ and the GFF has central charge $1$. As another example, random walks have central charge $0$. 
	Indeed, let $\mathfrak M$ be a finite collection of planar maps with boundary decorated by a  bulk marked point, and consider all pairs $(M,P)$ where $M \in \mathfrak M$ and $P = (v_1, \dots, v_n)$ is a path in $M$ started at the marked point and terminated when it hits the boundary. Each pair $(M,P)$ is given a weight $\prod_{i=1}^{n-1}{d_i^{-1}}$ where $d_i$ is the degree of $v_i$. If we sample $(M,P)$ proportionately to its weight, then the marginal law of $M$ is the uniform measure on $\mathfrak M$; this lack of weighting means random walks have central charge 0. Brownian motion, as the scaling limit of random walk, should also be viewed as having central charge 0.
\end{remark}

\section{Conditional independence for uniform meanders}
\label{sec-meander}

In this appendix we prove a conditional independence property for uniform meanders (Proposition~\ref{prop-meander} below). This property is the inspiration for the arguments in Section~\ref{subsec-topo} and can also be viewed as a discrete analog of certain variants of our main results.

An \emph{(open) meander} of size $m$ is a self-avoiding loop from $\infty$ to $\infty$ in $\C \cup \{\infty\}$ which intersects $\R$ at points $1, 2, \dots, 2m-1$, crossing it transversally at each of those points, and viewed modulo homeomorphisms $f: \C \to \C$ with $f(\bbH)= \bbH$ and $f(i) = i$ for $i \in \{ 1, \dots, 2m-1\}$. 
Meanders have been studied at least since the work of Poincar\'e~\cite{poincare-meander}, and are connected to many different topics in math and physics, but are notoriously difficult to analyze mathematically. See~\cite{zvonkin-meander-survey,lacroix-meander-survey} for expository works on meanders. 
A meander of size $m$ can be viewed as a planar map decorated by two Hamiltonian paths (corresponding to the loop and the real line). It is conjectured in~\cite{bgs-meander}, building on~\cite{dgg-meander-asymptotics}, that uniform random meanders of size $m$ converge in the scaling limit to the LQG sphere with central charge $\ccL = 30$ together with two independent SLE$_8$ curves.

A meander can be encoded by a pair of \emph{arc diagrams} (non-crossing perfect matchings) on $\{ 1, \dots, 2m - 1\} \cup \{ \infty\}$, one drawn above and the other drawn below $\R$. Orient the meander curve so it starts at $\infty$ in the lower half-plane and ends at $\infty$ in the upper half-plane. Then each arc inherits an orientation from the meander. 

We can define a \emph{winding function} $\theta : \{ 1, \dots, 2m-1\} \to \R$ as follows. With the above orientation of the curve, let it intersect $\R$ at points $v_1, \dots, v_{2m-1}$ in order (so $v_1, \dots, v_{2m-1}$ is a permutation of $1,\dots, 2m-1$). Let $\theta(v_1) = 0$, and inductively, we set  $\theta(v_{i+1}) = \theta(v_i) + \pi$ if the arc from $v_i$ to $v_{i+1}$ is traversed in the counterclockwise direction  and otherwise set $\theta(v_{i+1}) = \theta(v_i) - \pi$. See Figure~\ref{fig-meander}. 
Note that if we draw the arcs as semicircles and rays as in the figure, then  $\theta(v) + \frac\pi2$ is the restriction to $\{ 1, \dots, 2m-1\}$ of a  winding function of the curve (defined in Section~\ref{subsec-winding}).

\begin{figure}[ht!]
	\begin{center}
		\includegraphics[scale=0.54]{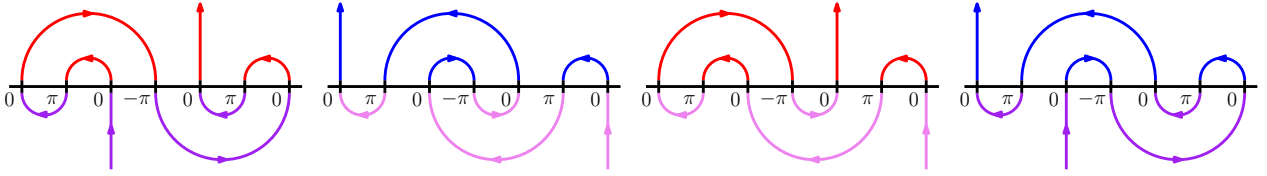}%
		\caption{\label{fig-meander} The four meanders with winding function given by $(\theta(1), \dots, \theta(7)) = (0, \pi, 0, -\pi, 0, \pi, 0)$. 
			In the leftmost diagram, we have $(v_1, \dots, v_7) = (3,2,1,4,7,6,5)$.
			In the proof of Proposition~\ref{prop-meander}, we have $A_\theta^+ = \{ \text{red, blue}\}$ and $A_\theta^- = \{ \text{purple, pink}\}$, and the set of all meanders with winding function $\theta$ corresponds to $A_\theta^+ \times A_\theta^-$. }
	\end{center}
\end{figure}

We have the following conditional independence statement for meanders, whose proof is a simpler version of the argument of Section~\ref{subsec-topo}.

\begin{proposition}\label{prop-meander}
	For $m \geq 1$, sample a meander uniformly at random from the set of all meanders of size $m$. Conditioned on its winding function $\theta : \{ 1, \dots, 2m-1\} \to \R$, the (oriented) arc diagrams above and below $\R$ are conditionally independent. 
\end{proposition}
\begin{proof}
	Fix $\theta$, and let $A^+_\theta$ (resp.\ $A^-_\theta$) be the set of all arc diagrams on $\{ 1, \dots, 2m-1\} \cup \{ \infty\}$ drawn above $\R$ (resp.\ below $\R$), such that for every arc with endpoints $1 \leq x < y \leq 2m-1$ we have $\theta(y) = \theta(x) - \pi$ (resp.\ $\theta(y) = \theta(x) + \pi$). For each arc diagram in $A_\theta^+$ (resp.\ $A_\theta^-$), we give  the arc hitting $\infty$ the orientation pointing towards (resp.\ away from) $\infty$, then orient all other arcs such that along $\R$ the arcs alternately enter and exit $\R$.  
	By the definition of $\theta$, every meander with winding function $\theta$ decomposes into a pair of arc diagrams in $A^+_\theta \times A^-_\theta$, see Figure~\ref{fig-meander}. We claim that conversely, every pair of arc diagrams in $A^+_\theta \times A^-_\theta$ gives a meander with winding function $\theta$. Indeed, any pair of arc diagrams in $A^+_\theta \times A^-_\theta$ form a collection of one or more loops in $\C \cup \{\infty\}$; suppose for the sake of contradiction there is more than one loop. Pick a loop not containing $\infty$. By the definitions of $A_\theta^+, A_\theta^-$, when tracing the loop, each clockwise (resp.\ counterclockwise) arc should decrease (resp.\ increase) the winding function by $\pi$, so since the winding function at the start and end of the traversal is the same, there must be an equal number of clockwise and counterclockwise half-turns. This is topologically impossible: Hopf's Umlaufsatz states that the total curvature of a simple smooth loop is $\pm 2\pi$, meaning  $\#\text{clockwise turns} - \# \text{counterclockwise turns} =\pm2$. Thus, every pair in $A_\theta^+ \times A_\theta^-$ gives a single loop from $\infty$ to $\infty$. Each such loop has winding function $\theta$ by the definition of $A_\theta^\pm$. 
	
	We conclude that the set of all meanders with winding function $\theta$ is in bijection with $A_\theta^+ \times A_\theta^-$, giving the desired conditional independence.
\end{proof}

\begin{remark}
	A \emph{meandric system} is defined in a similar manner to a meander, except that one allows there to be multiple loops.
	For a meandric system, the arc diagrams above and below the real line are exactly independent, rather than just conditionally independent. It is conjectured in~\cite{bgp-meander-system} that uniform meandric systems converge in the scaling limit to the LQG sphere with central charge $\ccL = 28$ together with an SLE$_8$ curve and an independent CLE$_6$. Since $\cc(\kappa=8) = -2 = 26-28$, the LQG and the SLE$_8$ curve have ``matched" parameter values. The fact that we have conditional independence, rather than exact independence, in Proposition~\ref{prop-meander} is analogous to the fact that we have conditional independence, rather than exact independence, in the main results of this paper, as stated in Section~\ref{subsec-intro-results}. 
	
	We expect that it is possible to extend the results of this paper to get a continuum analog of Proposition~\ref{prop-meander}. See Problem~\ref{prob-mating} and the discussion just after. We plan to explore discrete analogs of the results of this paper in the setting of meanders and meandric systems further in future work. 
\end{remark}

\section{Dub\'edat's proof of locality}\label{appendix-locality}
In this appendix we explain the proof of Theorem~\ref{thm-loc-dom}. It is a special case of Theorem~\ref{thm-RN} below, which was proved by Dub\'edat  who used it to establish his seminal SLE-GFF coupling \cite{dubedat-coupling}. It is not stated as a theorem in \cite{dubedat-coupling}, but its discrete analog is stated in the last equation of Section 4.1 there. 

In the setting of Section~\ref{subsec-loc-planar}, we have four simply connected domains $D_{ij}$ having a common sub-domain $B$ bounded by a pair of cross cuts and common boundary point $x \in \partial B \cap \partial D_{ij}$.
Recall that $\mathrm{IG}_\mathbf{c}$ in $(D_{ij}, x)$ as defined in Definition~\ref{def-ig-disk} can be viewed as a probability measure on fields in $D_{ij}$ modulo $2 \pi \chi(\mathbf c)$; let $\Psi$ be a sample with its additive constant fixed by requiring that its boundary value infinitesimally counterclockwise of $x$ in $\partial D_{ij}$ lies in $[0,2\pi\chi(\mathbf c))$. Let  $\mathrm{IG}_{\mathbf c}^{B, ij}$ be the law of $\psi = \Psi|_B$.  Note that here,  $\mathrm{IG}_{\mathbf c}^{B, ij}$ is a measure on the space of generalized functions and not $\mathbf c$-IG fields; this is a slight change in notation from Theorem~\ref{thm-loc-dom}.

\begin{thm}[\cite{dubedat-coupling}]\label{thm-RN}
We have 
	\eqb\label{eq-locality-domain}
	\left(\frac{{\det}_\zeta(\Delta_{D_{11}}){\det}_\zeta(\Delta_{D_{22}})}{{\det}_\zeta(\Delta_{D_{21}}){\det}_\zeta(\Delta_{D_{12}})} \right)^{-\mathbf c/2}
	\frac{d\mathrm{IG}_{\mathbf c}^{B,11}}{d\mathrm{IG}_{\mathbf c}^{B,21}} (\psi) \frac{ d\mathrm{IG}_{\mathbf c}^{B,22}}{d\mathrm{IG}_{\mathbf c}^{B,12}} (\psi) = 1  \qquad \text{ for } \mathrm{IG}_\mathbf{c}^{B,ij}\text{-a.e. }
	\psi
	\eqe
	where ${\det}_\zeta(\Delta_D)$ is the {zeta regularized determinant of the Dirichlet Laplacian}.
\end{thm}
We note that Theorem~\ref{thm-RN} suggests the partition function of $\mathrm{IG}_\cc$ in a domain $D$ should be ${\det}_\zeta(\Delta_D)^{-\cc/2}$.

\begin{proof}[Proof of Theorem~\ref{thm-loc-dom}]
	Set $\mathbf c= 0$ in Theorem~\ref{thm-RN}.
\end{proof}

Let $D$ be one of the domains $D_{ij}$ and let $m = m_{ij}$ be the harmonic function on $D_{ij}$ whose boundary values agree with those of a sample from $\mathrm{IG}_{\mathbf c}$ in $(D_{ij}, x)$ with additive constant fixed as before. Note that $m_{ij}$ is continuous on $\partial D_{ij}$ except at $x$ where there is a jump of size $2 \pi \chi(\mathbf c)$. Let $\delta = \partial B \cap D_{ij}$ be the union of the pair of crosscuts bounding $B$. Let $\mu$ be the law of $\Psi|_{\delta}$ where $\Psi$ is a GFF with boundary conditions given by $m$, and $\mu^0$ the law of $\Psi^0|_{\delta}$ where $\Psi^0$ is a zero boundary GFF on $D$.
As explained in \cite[Section 4.3]{dubedat-coupling}, $\mu$ and $\mu^0$ are probability measures on the negative Sobolev space $H^{-s}(\delta)$ for any $s>0$.  We will write $w$ to denote a field on $\delta$. Let $Pw$ be the harmonic extension of $w$ to $D \backslash \delta$ with zero boundary conditions on $\partial D$. 

For a function $f$ on $D$ and a point $z \in \delta$ let $\partial_\ell f (z) := \lim_{\eps \to 0} \eps^{-1} (f(z + \eps n_\ell) - f(z))$ where $n_\ell$ is the left-pointing normal unit vector on $\delta$ at $w$, and define $\partial_r$ the same way with left replaced by right. Thus $\partial_\ell f$ and $\partial_r f$ are functions with domain $\delta$, and $(\partial_\ell + \partial_r)f = 0$ at points where $f$ is smooth.

Let $m' = m - P(m|_\delta)$. This is the function on $D$ such that $m'|_\delta = 0, m'|_{\partial D} = m|_{\partial D}$, and $m'$ is harmonic on $D \backslash \delta$.  The following computation is carried out in \cite[Proof of Lemma 6.3]{dubedat-coupling}. We repeat the proof for the reader's convenience. 
\begin{lemma}\label{lem-girsanov-bdy} 
	Define the regularized Dirichlet energy
	\[(m, m)_{\nabla}^\mathrm{reg} := \frac1{2\pi}\lim_{\eps \to 0} \left( \frac{(2\pi \chi(\mathbf c))^2}\pi \log \eps + \int_{D_{ij} \backslash B_\eps(x)} \nabla m \cdot \nabla m \right).\]
	With the following integral performed with respect to arc length on $\delta$, we have 
	\[\frac{d\mu}{d\mu^0}(w) = \exp\left(-\frac1{2\pi}\int_\delta w (\partial_\ell + \partial_r)m' - \frac12 (m', m')_\nabla^\mathrm{reg} + \frac12 (m, m)_\nabla^\mathrm{reg} \right) .\]
\end{lemma}
\begin{proof}
	Let $m'' = P(m|_\delta)$, so $m''|_\delta = m|_\delta$ and $m''|_{\partial D} = 0$.
	Let $\Psi^0$ be a zero boundary GFF on $D$. By Girsanov's theorem (see e.g.\ \cite[Proposition 2.9]{ig4}), the Radon-Nikodym derivative of the law of $\Psi^0+ m''$ with respect to that of $\Psi^0$ is
	\[
	\exp((\Psi^0, m'')_\nabla - \frac12 (m'', m'')_\nabla). 
	\]
	By Stoke's theorem in each of the components of $D \backslash \delta$ and since $\Delta m''|_{D \backslash \delta} = 0$ and  $\Phi^0|_{\partial D} = 0$,  
	\[(\Psi^0, m'')_\nabla = \frac1{2\pi} \int_\delta \Psi^0 (\partial_\ell + \partial_r) m''  = - \frac1{2\pi} \int_\delta \Psi^0 (\partial_\ell + \partial_r)m'.  \]
	Here, the second equality follows from $m = m' + m''$ and  $(\partial_\ell + \partial_r) m = 0$. 
	Thus, 
	\eqb\label{eq-girsanov-bdy}
	\frac{d\mu}{d\mu^0}(w) = \exp(-\frac1{2\pi}\int_\delta w (\partial_\ell + \partial_r)m' - \frac12 (m'', m'')_\nabla ) .
	\eqe
	
	Finally, we need to show that $(m'', m'')_\nabla = (m', m')_\nabla^\mathrm{reg} - (m, m)_\nabla^\mathrm{reg}$. Let $\eps > 0$ and let $m_\eps$ be a harmonic function on $D$ with continuous boundary data.
	With $m_\eps'' = P(m_\eps|_\delta)$ and $m_\eps' = m_\eps - m_\eps''$, Stoke's theorem gives  $2\pi(m_\eps, m_\eps'')_\nabla =  -\int_{\partial D} m_\eps'' \partial_n m_\eps + \int_Dm_\eps'' \Delta m_\eps$. This is zero because $m_\eps''|_{\partial D} = 0$ and $\Delta m_\eps = 0$, so expanding $(m_\eps, m_\eps'')_\nabla = 0$ gives $(m_\eps, m_\eps)_\nabla = (m_\eps, m_\eps')_\nabla$. This implies
	$(m_\eps'', m_\eps'')_\nabla = (m_\eps', m_\eps')_\nabla - (m_\eps, m_\eps)_\nabla$.

	We will now choose $m_\eps$ to approximate $m$ then send $\eps \to 0$. Assume the boundary data of $m_\eps$ agrees with that of $m$ in $\partial D \backslash B_{\eps^2}(x)$, and assume this boundary data is uniformly bounded for all $\eps$. 
	Since $m_\eps|_\delta$ converges uniformly to $m|_\delta$ we have $(m''_\eps, m''_\eps)_\nabla \to (m'', m'')_\nabla$. Next, since the harmonic measure of $\partial D \cap B_{\eps^2}(x)$ is $o_\eps(1)$ uniformly for all $z \in D \backslash B_{\eps}(x)$, we have $\left| \int_{D \backslash B_\eps(x)} \nabla m_\eps \cdot \nabla m_\eps - \nabla m \cdot \nabla m \right| = o_\eps(1)$, and similarly $\left| \int_{D \backslash B_\eps(x)} \nabla m_\eps' \cdot \nabla m_\eps' - \nabla m' \cdot \nabla m' \right| = o_\eps(1)$. Moreover, since Brownian motion started in $D \cap B_\eps(x)$ hits $\partial D$ before $B_{\sqrt \eps}(x)$ with probability $1-o_\eps(1)$ uniformly in the starting point, and the harmonic functions $m_\eps$ in $D$ and $m'_\eps$ in $D \backslash \delta$ have the same boundary data on $\partial D \cap B_{\sqrt\eps}(x)$, we conclude $\left|\int_{D \cap B_\eps(x)} \nabla m_\eps \cdot \nabla m_\eps -  \nabla m_\eps' \cdot \nabla m_\eps' \right| = o_\eps(1)$. To summarize, $(m_\eps', m_\eps')_\nabla  - (m_\eps, m_\eps)_\nabla \to (m', m')_\nabla^\mathrm{reg} - (m, m)_\nabla^\mathrm{reg} $. Therefore $(m'', m'')_\nabla = (m', m')_\nabla^\mathrm{reg} - (m, m)_\nabla^\mathrm{reg}$, which with~\eqref{eq-girsanov-bdy} completes the proof. 
\end{proof}
The following identity follows from setting $(a,b) = (0, \chi(\mathbf c)/\sqrt{2\pi})$ in  \cite[Proposition 5.2]{dubedat-coupling}.
\begin{lemma}[{\cite[Proposition 5.2]{dubedat-coupling}}]\label{lem-reg-dirichlet}
	There is a constant $\lambda$ such that 
	\[e^{-\frac12(m,m)_\nabla^\mathrm{reg}} = \lambda {\det}_\zeta (\Delta_D)^{3 \chi(\mathbf c)^2}.\]
\end{lemma}

Now we state two other results from \cite{dubedat-coupling}. Let $m^l(D, ; K_1, K_2)$ be the mass of Brownian loops in $D$ intersecting both $K_1$ and $K_2$, as introduced in~\cite{lawler-werner-soup}. 
\begin{lemma}[{\cite[Proposition 2.1]{dubedat-coupling}}]\label{lem-zeta}
	Let $C$ be a collar neighborhood of $\delta$ in $D$, that is, a neighborhood of $\delta$ which is diffeomorphic to $\delta \times (-1, 1)$ via a diffeomoephism which takes the intersection of the neighborhood with $\partial D$ to $(\partial \delta) \times (-1, 1)$. Then 
	\[\exp(-m^l (D; \delta, D\backslash C)) = \frac{{\det}_\zeta(\Delta_D) {\det}_\zeta(\Delta_{C \backslash \delta})}{{\det}_\zeta(\Delta_{D \backslash \delta})  {\det}_\zeta(\Delta_C)}.\]
\end{lemma}

For $i,j, \in \{1,2\}$, let $\mu_{ij}$ (resp.\  $\mu^0_{ij}$) be the law of $h$ sampled from $\mathrm{IG}_{\mathbf c}^{K, ij}$ (resp.\ zero boundary GFF on $D_{ij}$) restricted to $\delta$. Let $P_{ij} w$ be the function on $D_{ij}$ which is harmonic on $D_{ij} \backslash \delta$, equals $w$ on $\delta$  and equals zero on $\partial D_{ij}$. 
\begin{lemma}[{\cite[Lemma 4.4]{dubedat-coupling}}]\label{lem-girsanov-zero}
	Let $C$ be any common collar neighborhood of $\delta$ for $D_{ij}$ and $D_{i'j'}$. Then 
	\[
	\frac{d\mu_{i'j'}^0}{d\mu_{ij}^0}(w) = \exp\left(\pi ( ( P_{ij} w, P_{ij}w )_\nabla - ( P_{i'j'}w, P_{i'j'}w )_\nabla) +\frac12 (m^l(D_{ij}; \delta, D_{ij} \backslash C) - m^l(D_{i'j'}; \delta, D_{i'j'} \backslash C) ) \right).
	\]
\end{lemma}

We turn to the proof of Theorem~\ref{thm-RN}. It is essentially done in \cite[Proof of Lemma 6.3]{dubedat-coupling}.
\begin{proof}[{Proof of Theorem~\ref{thm-RN}}]
	All Radon-Nikodym derivative identities in this proof hold for $\mu_{ij}$-a.e.\ $w$. 
	We have 
	\[ \frac{d\mu_{11}}{d\mu_{21}} \times \frac{d\mu_{22}}{d\mu_{12}} = \frac{d\mu_{11}}{d\mu_{11}^0}  \frac{d\mu_{11}^0}{d\mu_{21}^0}  \frac{d\mu_{21}^0}{d\mu_{21}} 
	\times 
	\frac{d\mu_{22}}{d\mu_{22}^0}  \frac{d\mu_{22}^0}{d\mu_{12}^0}  \frac{d\mu_{12}^0}{d\mu_{12}}
	. \]
	Let $P_{ij}w$ denote the harmonic function on $D_{ij}$ with boundary conditions of $w$ on $\delta$ and zero boundary conditions on $\partial D_{ij}$. Then $m_{ij}':= m_{ij} - P_{ij}(m_{ij}|_\delta)$ is the harmonic function on $D_{ij} \backslash \delta$ with boundary conditions on $\partial D_{ij}$ agreeing with those of $m_{ij}$, and zero boundary conditions on $\delta$.
	Notice that the boundary values of $m_{ij}$ in the left subdomain of $D_{ij}$ do not depend on $j$ since the boundary values near $x$ do not depend on $j$ and the boundary values along $\partial \D$ in the clockwise direction from $x$ vary according to the total curvature (winding) of the boundary.
	Thus $m_{ij}'$ restricted to the left subdomain of $D_{ij}$ does not depend on $j$, and likewise $m_{ij}'$ restricted to the right subdomain does not depend on $i$.  Thus, by Lemma~\ref{lem-girsanov-bdy} 
	\[\frac{d\mu_{11}}{d\mu_{11}^0}
	\frac{d\mu_{21}^0}{d\mu_{21}}
	\frac{d\mu_{22}}{d\mu_{22}^0}
	\frac{d\mu_{12}^0}{d\mu_{12}} 
	= \exp(\frac12 (
	(m_{11}, m_{11})_\nabla^\mathrm{reg}
	- (m_{21}, m_{21})_\nabla^\mathrm{reg}
	+ (m_{22}, m_{22})_\nabla^\mathrm{reg}
	- (m_{12}, m_{12})_\nabla^\mathrm{reg}
	)  ),
	\]
	since all terms involving $m'_{ij}$ cancel. Next, since $P_{ij} w$ restricted to the left (resp.\ right) subdomain of $D_{ij}$ does not depend on $j$ (resp.\ $i$),  we have 
	\[(P_{11}w, P_{11}w)_\nabla - (P_{21}w, P_{21}w)_\nabla + (P_{22}w, P_{22}w)_\nabla - (P_{12}w, P_{12}w)_\nabla = 0.\]
	Thus, by Lemmas~\ref{lem-zeta} and~\ref{lem-girsanov-zero}, we have 
	\[\frac{d\mu_{11}^0}{d\mu_{21}^0} \frac{d\mu_{22}^0}{d\mu_{12}^0} = \sqrt{\frac{{\det}_\zeta(\Delta_{D_{11}}) {\det}_\zeta(\Delta_{D_{22}}) }{ {\det}_\zeta(\Delta_{D_{21}}) {\det}_\zeta(\Delta_{D_{12}}) }}.\]
	Combining the above identities and Lemma~\ref{lem-reg-dirichlet}, and using  $-\frac12 + 3\chi(\mathbf c)^2 = -\frac{\mathbf c}2$, we conclude 
	\[	\left(\frac{{\det}_\zeta(\Delta_{D_{11}}){\det}_\zeta(\Delta_{D_{22}})}{{\det}_\zeta(\Delta_{D_{21}}){\det}_\zeta(\Delta_{D_{12}})} \right)^{-\mathbf c/2} \frac{d\mu_{11}}{d\mu_{21}}\frac{d\mu_{22}}{d\mu_{12}} = 1.\]
	By the domain Markov property of the Gaussian free field, this implies~\eqref{eq-locality-domain}. 
\end{proof}

\bibliography{cibib,bib}
\bibliographystyle{hmralphaabbrv}

\end{document}